\documentclass[11pt,reqno]{amsart}
\usepackage{amsmath,amsfonts,amssymb}
\usepackage[cp1251]{inputenc}
\usepackage[english]{babel}
\usepackage{mathrsfs}
\newcommand{\esssup}{\mathop{\mathrm{ess}\:\mathrm{sup}}\limits}

\def\C{\mathbb C}

\def\x{\mathbf x}

\def\D{\mathbf D}

\def\1{\mathbf 1}

\def\1{\bold 1}

\def\eps{\varepsilon}

\theoremstyle{theorem}
\newtheorem{theorem}{Theorem}[section]
\newtheorem{proposition}[theorem]{Proposition}
\newtheorem{lemma}[theorem]{Lemma}
\newtheorem{condition}[theorem]{Condition}
\newtheorem{remark}[theorem]{Remark}
\newtheorem{corollary}[theorem]{Corollary}
\numberwithin{equation}{section}

\begin{document}

\title[Homogenization of initial boundary value problems]{Homogenization of initial boundary value problems
for parabolic systems \\ with periodic coefficients}

\author{Yu.~M.~Meshkova and T.~A.~Suslina}

\thanks{Supported by RFBR (grant no.~14-01-00760).
The first author is supported by the Chebyshev Laboratory  (Department of Mathematics and Mechanics, St.~Petersburg State University)  under RF Government grant 11.G34.31.0026 and JSC \glqq Gazprom Neft\grqq.}


\address{Chebyshev Laboratory, St. Petersburg State University, 14th Line, 29b, Saint Petersburg, 199178, Russia}
\email{juliavmeshke@yandex.ru}

 \address{St. Petersburg State University, Department of Physics, Ul'yanovskaya 3, Petrodvorets,
St.~Petersburg, 198504, Russia}

\email{suslina@list.ru}


\begin{abstract}
Let $\mathcal{O} \subset \mathbb{R}^d$ be a bounded domain of class $C^{1,1}$.
In the Hilbert space $L_2(\mathcal{O};\mathbb{C}^n)$,
we consider matrix elliptic second order differential operators $\mathcal{A}_{D,\varepsilon}$ and $\mathcal{A}_{N,\varepsilon}$
with the Dirichlet or Neumann boundary condition on $\partial \mathcal{O}$, respectively.
Here $\varepsilon>0$ is the small parameter. The coefficients of the operators are periodic and depend on
$\mathbf{x}/\varepsilon$. The behavior of the operator $e^{-\mathcal{A}_{\dag ,\varepsilon}t}$, $\dag =D,N$,
for small $\varepsilon$ is studied.
It is shown that, for fixed $t>0$, the operator $e^{-\mathcal{A}_{\dag ,\varepsilon}t}$
converges in the $L_2$-operator norm to $e^{-\mathcal{A}_{\dag}^0 t}$, as $\varepsilon \to 0$.
Here $\mathcal{A}_{\dag}^0$ is the effective operator with constant coefficients.
For the norm of the difference of the operators $e^{-\mathcal{A}_{\dag ,\varepsilon}t}$ and $e^{-\mathcal{A}_{\dag}^0 t}$
a sharp order estimate (of order $O(\varepsilon)$) is obtained.
Also, we find approximation for the exponential $e^{-\mathcal{A}_{\dag ,\varepsilon}t}$
in the $(L_2\rightarrow H^1)$-norm with error estimate of order $O(\varepsilon ^{1/2})$;
in this approximation, a corrector is taken into account.
The results are applied to homogenization of solutions of initial boundary value problems for parabolic systems.
\end{abstract}

\maketitle

\section*{Introduction}

The paper concerns homogenization theory of periodic differential operators (DO's).
A broad literature is devoted to homogenization theory.
First of all, we mention the books \cite{BaPa,BeLP,ZhKO} and \cite{Sa}.

\subsection*{0.1. The class of operators.}
Let $\mathcal{O}\subset \mathbb{R}^d$ be a bounded domain with the boundary of class $C^{1,1}$.
We study selfadjoint strongly elliptic second order operators  $\mathcal{A}_{D ,\varepsilon }$ and
$\mathcal{A}_{N ,\varepsilon }$ acting in the space $L_2(\mathcal{O};\mathbb{C}^n)$ and depending on the small parameter
$\varepsilon >0$. Formally, the operators $\mathcal{A}_{D ,\varepsilon }$ and
$\mathcal{A}_{N ,\varepsilon }$ are given by the differential expression  $b(\mathbf{D})^*g(\mathbf{x}/\varepsilon)b(\mathbf{D})$
with the Dirichlet or Neumann condition on $\partial \mathcal{O}$, respectively.
Here $g(\mathbf{x})$ is a Hermitian $(m\times m)$-matrix-valued function assumed to be bounded, positive definite and periodic
with respect to the lattice $\Gamma\subset\mathbb{R}^d$.
The operator $b(\mathbf{D})$ is an $(m\times n)$-matrix first order DO with constant coefficients.
It is assumed that $m \geqslant n$. The symbol $b(\boldsymbol{\xi})$ is subject to some condition ensuring
strong ellipticity of the operators under consideration.
The simplest example is the scalar elliptic divergence type operator
$-{\rm div}\,g(\mathbf{x}/\varepsilon)\nabla$ (the acoustics operator).

\textit{Our goal} is to find approximations in various operator norms for the operator exponential $e^{-\mathcal{A}_{\dag ,\varepsilon }t}$, $\dag =D,N$,
for small  $\varepsilon$ and positive $t$ and to apply the results to homogenization of solutions of the parabolic systems.

\subsection*{0.2. Operator error estimates. Survey.}
In a series of papers \hbox{[BSu1--3]} by M.~Sh.~Birman and T.~A.~Suslina
an operator-theoretic approach to homogenization problems was suggested.
Operators ${A}_\varepsilon = b(\mathbf{D})^* g(\mathbf{x}/\varepsilon)b(\mathbf{D})$
acting in $L_2(\mathbb{R}^d;\mathbb{C}^n)$ were studied.
In [BSu1], it was shown that, as $\varepsilon \to 0$,
the resolvent $(A_\varepsilon +I)^{-1}$ converges in the operator norm in $L_2(\mathbb{R}^d;\mathbb{C}^n)$
to the resolvent $(A^0+I)^{-1}$ of the effective operator $A^0 = b(\mathbf{D})^*g^0 b(\mathbf{D})$.
Here $g^0$ is a constant positive effective matrix. The following error estimate was obtained:
\begin{equation}
\label{0.1}
\Vert ({A}_{\varepsilon } +I)^{-1} - ({A}^0 +I)^{-1}\Vert _{L_2(\mathbb{R}^d)\rightarrow
L_2(\mathbb{R}^d)}\leqslant {C} \varepsilon.
\end{equation}
In [BSu3], approximation of the resolvent  $(A_\varepsilon +I)^{-1}$
in the norm of operators acting from $L_2(\mathbb{R}^d;\mathbb{C}^n)$ to the Sobolev space
$H^1(\mathbb{R}^d;\mathbb{C}^n)$ was obtained:
\begin{equation}
\label{0.2}
\Vert ({A}_{\varepsilon } +I)^{-1} - ({A}^0 +I)^{-1} - \varepsilon K(\varepsilon)\Vert _{L_2(\mathbb{R}^d)\rightarrow
H^1(\mathbb{R}^d)}\leqslant {C} \varepsilon.
\end{equation}
Here $K(\varepsilon)$ is the so called corrector; it contains rapidly oscillating factors, and so depends on
$\varepsilon$. We have $\|\varepsilon K(\varepsilon)\| _{L_2\to H^1} = O(1)$.

Estimates (\ref{0.1}), (\ref{0.2}) are order-sharp; the constants in estimates are controlled explicitly in terms of the
initial problem data. The results of this kind are called \textit{operator error estimates}. The method of [BSu1--3] was based on
the scaling transformation, the Floquet-Bloch theory and the analytic perturbation theory.

To parabolic problems in $\mathbb{R}^d$ this method was applied in [Su1--3], [V], [VSu], and [M].
In [Su1,2], it was shown that for fixed  $t>0$
the operator exponential $e^{- A_\varepsilon t}$ converges in the $L_2(\mathbb{R}^d;\mathbb{C}^n)$-operator norm to
$e^{-A^0 t}$, as $\varepsilon \to 0$; and the following estimate was obtained:
\begin{equation}
\label{0.3}
\Vert e^{-A_\varepsilon t}-e^{-A^0 t}\Vert _{L_2(\mathbb{R}^d)\rightarrow L_2(\mathbb{R}^d)}
\leqslant {C} \varepsilon (t+\varepsilon ^2)^{-1/2}.
\end{equation}
In [Su3], approximation of the exponential  $e^{- A_\varepsilon t}$
in the $(L_2\rightarrow H^1)$-norm with corrector taken into account was obtained:
\begin{equation}
\label{0.4}
\Vert e^{-A_\varepsilon t}-e^{-A^0 t}-\varepsilon {\mathcal K}(t;\varepsilon )\Vert _{L_2(\mathbb{R}^d)\rightarrow H^1(\mathbb{R}^d)}\leqslant
{C}\varepsilon (t^{-1}+t^{-1/2}),\quad t\geqslant \varepsilon ^2.
\end{equation}

A different approach to operator error estimates was suggested by V.~V.~Zhikov.
In [Zh1,2], [ZhPas1], estimates of the form (\ref{0.1}) and (\ref{0.2}) were obtained for
the acoustics operator and the elasticity operator.
In \hbox{[ZhPas2]}, estimates of the form  (\ref{0.3}) and (\ref{0.4}) were proved for the scalar operator
$A_\varepsilon =-\textnormal{div}\,g(\mathbf{x}/\varepsilon)\nabla$.

Operator error estimates have been also studied for  boundary value problems  in a bounded domain; see [ZhPas1], [Gr1,2], and [KeLiS].
(A more detailed survey can be found in  [Su6,8].)

For the class of matrix elliptic operators
$\mathcal{A}_{D,\varepsilon}$ and $\mathcal{A}_{N,\varepsilon}$ that we consider, homogenization problems were studied in recent papers
[PSu1,2], [Su4--8]. In [PSu1,2] and [Su4,5],
approximations of the operator $\mathcal{A}_{D,\varepsilon}^{-1}$ were obtained.
In [Su6], approximation for the resolvent of $\mathcal{A}_{N,\varepsilon}$ was found.
The method of the cited papers was based on application of the results for the problem in ${\mathbb R}^d$, introduction of the boundary
layer correction term, and a careful analysis of this term. Some technical tricks, in particular, using the Steklov smoothing,
were borrowed from [ZhPas1].

For the present paper, the most important are the results of [Su7,8], where approximations of the resolvents
$(\mathcal{A}_{D,\varepsilon} - \zeta I)^{-1}$ and $(\mathcal{A}_{N,\varepsilon}- \zeta I)^{-1}$ with two-parametric error estimates
(with respect to $\varepsilon$ and $\zeta$) were obtained. Let us dwell on the main results of  [Su7,8].
For $\zeta \in \mathbb{C}\setminus \mathbb{R}_+$, $\vert \zeta \vert \geqslant 1$, and
$0< \varepsilon \leqslant \varepsilon_0$ (where $\varepsilon_0$ is sufficiently small) we have
\begin{align}
\label{0.5}
\begin{split}
\Vert &(\mathcal{A}_{\dag ,\varepsilon}-\zeta I)^{-1}-(\mathcal{A}_\dag ^0 -\zeta I)^{-1}\Vert _{L_2(\mathcal{O})\rightarrow L_2(\mathcal{O})}
\leqslant {C}(\phi) \left(\varepsilon \vert \zeta \vert ^{-1/2}+\varepsilon ^2\right),
\end{split}
\\
\label{0.6}
\begin{split}
\Vert &(\mathcal{A}_{\dag ,\varepsilon}-\zeta I)^{-1}-(\mathcal{A}_\dag ^0 -\zeta I)^{-1} -
\varepsilon K_\dag (\varepsilon ;\zeta)\Vert _{L_2(\mathcal{O})\rightarrow H^1(\mathcal{O})}\\
&\leqslant {C}(\phi)\left(\varepsilon ^{1/2}\vert \zeta \vert^{-1/4}+\varepsilon \right).
\end{split}
\end{align}
Here $\dag =D,N$, $\phi = {\rm arg}\, \zeta$; $\mathcal{A}_\dag ^0$ is the effective operator given by the expression
$b(\mathbf{D})^*g^0 b(\mathbf{D})$ with the Dirichlet or Neumann condition; $K_\dag (\varepsilon ;\zeta)$ is the corresponding corrector.
These estimates are uniform with respect to $\phi$ in any domain of the form
$\lbrace \zeta \in \mathbb{C}\setminus \mathbb{R}_+,\,\vert \zeta \vert \geqslant 1,\,
\phi_0 \leqslant \phi \leqslant 2\pi -\phi _0\rbrace$ with arbitrarily small $\phi_0 >0$.
Note that for fixed $\zeta$ estimate (\ref{0.5}) is of order $O(\varepsilon)$; the order is the same as in estimate
(\ref{0.1}) in ${\mathbb R}^d$. At the same time, estimate (\ref{0.6}) is of order $O(\varepsilon^{1/2})$; the order is worse than
in estimate (\ref{0.2}) in ${\mathbb R}^d$. This is caused by the boundary influence.

\subsection*{0.3. Main results.} Main results of the paper are approximations of the operator exponential
$e^{-\mathcal{A}_{\dag ,\varepsilon }t}$, $\dag =D,N$, in the $L_2(\mathcal{O};\mathbb{C}^n)$-operator norm
and in the norm of operators acting from $L_2(\mathcal{O};\mathbb{C}^n)$ to $H^1(\mathcal{O};\mathbb{C}^n)$,
with the following twoparametric (with respect to $\varepsilon$ and $t$) error estimates:
\begin{align}
\label{Th_L2_vved}
\Vert &e^{-\mathcal{A}_{\dag ,\varepsilon }t} -e^{-\mathcal{A}_{\dag}^0 t}\Vert _{L_2(\mathcal{O})\rightarrow
L_2(\mathcal{O})}\leqslant {C} \varepsilon (t+\varepsilon ^2)^{-1/2} e^{-c_\dag t},\quad t\geqslant 0,
\\
\label{Th_korr_vved}
\begin{split}
\Vert & e^{-\mathcal{A}_{\dag ,\varepsilon }t} -e^{-\mathcal{A}_{\dag}^0 t} -\varepsilon \mathcal{K}_\dag (t;\varepsilon )\Vert _{L_2(\mathcal{O})\rightarrow H^1(\mathcal{O})}
\leqslant {C} \varepsilon ^{1/2}t^{-3/4}e^{-c_{\dag} t},\quad
 t\geqslant \varepsilon ^2,
\end{split}
\end{align}
$ 0< \varepsilon \leqslant \varepsilon _0$, $\dag =D,N$. Here $\mathcal{K}_\dag (t;\varepsilon )$ is the corresponding corrector.
Besides  (\ref{Th_korr_vved}), we obtain approximation of the operators
$g(\mathbf{x}/\varepsilon)b(\mathbf{D}) e^{-\mathcal{A}_{\dag ,\varepsilon }t}$ (corresponding to the ,,fluxes'') in the $(L_2 \to L_2)$-norm.

For bounded values of $t$, the order of estimate (\ref{Th_L2_vved})
is sharp and coincides with the order of estimate (\ref{0.3}) for the problem in ${\mathbb R}^d$.
The order of estimate \eqref{Th_korr_vved} deteriorates compared with (\ref{0.4}), which is explained
by the boundary influence.

Let us explain the origin of the factor $e^{-c_{\dag} t}$ in  (\ref{Th_L2_vved}) and \eqref{Th_korr_vved}.
In estimates  (\ref{0.3}) and (\ref{0.4}) for the problem in ${\mathbb R}^d$ such factor is absent,
because the point $\lambda=0$ is the bottom of the (continuous) spectra of the operators $A_\varepsilon$ and $A^0$.
The spectra of the operators $\mathcal{A}_{\dag ,\varepsilon }$ and $\mathcal{A}_{\dag}^0$ are discrete.
For $\dag =D$ these operators are positive definite, and one can take $c_D$ equal to
any positive number lying to the left of their spectra.
For $\dag =N$ the number $\lambda =0$ is an eigenvalue of these operators, but in the difference
$e^{-\mathcal{A}_{N ,\varepsilon }t} -e^{-\mathcal{A}_{N}^0 t}$
the parts of the operators in the kernel $Z= {\rm Ker}\, \mathcal{A}_{N ,\varepsilon }= {\rm Ker}\, \mathcal{A}_{N}^0$
cancel out; one can take $c_N$ equal to any positive number lying to the left of
the first nonnegative eigenvalues of the operators $\mathcal{A}_{N ,\varepsilon }$ and $\mathcal{A}_{N}^0$.

In general case, the corrector $\mathcal{K}_\dag (t;\varepsilon )$ contains an auxiliary smoothing operator.
We distinguish conditions under which it is possible to use the standard corrector (without the smoothing operator).

Moreover, for a strictly interior subdomain $\mathcal{O}'$ of the domain $\mathcal{O}$
we find approximation of the exponential $e^{-\mathcal{A}_{\dag ,\varepsilon }t}$ in the
$(L_2(\mathcal{O})\to H^1(\mathcal{O}'))$-operator norm with the following error estimate
\begin{equation}
\label{0.9}
\begin{split}
\Vert & e^{-\mathcal{A}_{\dag ,\varepsilon }t} -e^{-\mathcal{A}_{\dag}^0 t} -\varepsilon \mathcal{K}_\dag (t;\varepsilon )\Vert _{L_2(\mathcal{O})\rightarrow H^1(\mathcal{O}')}\\
&\leqslant {C}(\delta) \varepsilon t^{-1} e^{-c_{\dag} t},\quad t\geqslant \varepsilon ^2,
\quad 0< \varepsilon \leqslant \varepsilon_0.
\end{split}
\end{equation}
Here the constant depends on $\delta = {\rm dist}\, \{ \mathcal{O}'; \partial \mathcal{O}\}$.
For bounded values of $t$ the order of estimate (\ref{0.9}) is the same as in estimate
(\ref{0.4}) for the problem in ${\mathbb R}^d$.
This once again shows that deterioration of the order of estimate
\eqref{Th_korr_vved} is caused by the boundary layer effect.

\subsection*{0.4. Method.} We explain the method of investigation on the example of estimate \eqref{Th_L2_vved} in the
case of the Dirichlet boundary condition. Let $\gamma$ be a positively oriented contour in the complex plane
enclosing the spectra of the operators $\mathcal{A}_{D,\varepsilon}$ and $\mathcal{A}_D^0$. We have
$$
e^{-\mathcal{A}_{D,\varepsilon}t}=-\frac{1}{2\pi i}\int _\gamma e^{-\zeta t}(\mathcal{A}_{D,\varepsilon}-\zeta I)^{-1}\,d\zeta ,\quad t>0.
$$
The exponential of the effective operator satisfies a similar identity. Hence,
\begin{equation}
\label{0.10}
e^{-\mathcal{A}_{D,\varepsilon}t}-e^{-\mathcal{A}_D^0t}=-\frac{1}{2\pi i}\int _\gamma e^{-\zeta t}\left((\mathcal{A}_{D,\varepsilon}-\zeta I)^{-1}-(\mathcal{A}_D^0 -\zeta I)^{-1}\right)\,d\zeta .
\end{equation}
Applying the results of [Su8], we obtain approximation of the resolvent for $\zeta \in\gamma$,
and then use representation \eqref{0.10}. This leads to  \eqref{Th_L2_vved}.
Note that the character of dependence of the right-hand side of \eqref{0.5} on $\zeta$
for large values of $\vert \zeta\vert$ is important for us.
Approximation with corrector taken into account is obtained by the same way.

\subsection*{0.5. Application of the results to homogenization of solutions of initial boundary value problems.}
The results about approximation of the operator exponential $e^{-\mathcal{A}_{\dag,\varepsilon}t}$
are applied to the study of solutions $\mathbf{u}_{\varepsilon }(\mathbf{x},t)$
of the initial boundary value problems for the parabolic equation
\begin{equation}
\label{0.11}
\frac{\partial \mathbf{u}_{\varepsilon }(\mathbf{x},t)}{\partial t}=
-b(\mathbf{D})^*g(\mathbf{x}/\varepsilon)b(\mathbf{D})\mathbf{u}_{\varepsilon}(\mathbf{x},t) +
{\mathbf F}(\mathbf{x},t), \quad \mathbf{x}\in \mathcal{O},\, 0< t < T,
\end{equation}
with the initial condition
$\mathbf{u}_{\varepsilon }(\mathbf{x},0)=\boldsymbol{\varphi}(\mathbf{x})$
and the Dirichlet or Neumann boundary condition.
It is assumed that $\boldsymbol{\varphi}\in L_2(\mathcal{O};\mathbb{C}^n)$ and
${\mathbf F} \in L_p((0,T);L_2(\mathcal{O};\mathbb{C}^n))$ with some $p$.
Let $\mathbf{u}_0(\mathbf{x},t)$ be the solution of the corresponding effective problem.

We obtain approximations of the solution $\mathbf{u}_\varepsilon$ in various norms.
In principal order, we prove estimates of the difference $\mathbf{u}_\varepsilon -\mathbf{u}_0$
in the $L_2(\mathcal{O};\mathbb{C}^n)$-norm for fixed $t$ and in
the $L_p((0,T);L_2(\mathcal{O};\mathbb{C}^n))$-norm.
With corrector taken into account, we find approximations of the solution $\mathbf{u}_\varepsilon$ in $H^1(\mathcal{O};\mathbb{C}^n)$
for fixed $t$ and in $L_p((0,T);H^1(\mathcal{O};\mathbb{C}^n))$.
Moreover, for a strictly interior subdomain $\mathcal{O}'$ we obtain sharper error estimates for approximation of the
solution $\mathbf{u}_\varepsilon$ in $H^1(\mathcal{O}';\mathbb{C}^n)$ (for fixed $t$) and in $L_p((0,T);H^1(\mathcal{O}';\mathbb{C}^n))$.

\subsection*{0.6. Plan of the paper.} The paper consists of two chapters and Appendix. In Chapter 1 (\S\S1--4)
we study operators with the Dirichlet boundary condition. Chapter 2 (\S\S5--8) is devoted to operators with the Neumann
boundary condition. In \S1, the class of operators $\mathcal{A}_{D,\varepsilon}$ is described, the effective operator
is defined, and the Steklov smoothing operator is introduced. In \S2,  we formulate the results
about the resolvent of the operator $\mathcal{A}_{D,\varepsilon}$.
In \S3, the main results of the paper for the case of the Dirichlet boundary condition are obtained, namely,
approximations for the operator $e^{-\mathcal{A}_{D,\varepsilon}t}$ in the $(L_2\rightarrow L_2)$- and $(L_2\rightarrow H^1)$-operator
norms are found. Also, approximation of the ,,flux''
$g(\mathbf{x}/\varepsilon)b(\mathbf{D})e^{-\mathcal{A}_{D,\varepsilon}t}$ is obtained.
In \S4, the results of \S3 about approximation of the operator exponential are applied to homogenization of solutions of the first
initial boundary value problem for the parabolic equation (\ref{0.11}).

The order of presentation in Chapter~2 is similar to that of Chapter~1.
In \S5, the class of operators $\mathcal{A}_{N,\varepsilon}$ is described and the effective operator is introduced.
In \S6, the results  about the resolvent of the operator $\mathcal{A}_{N,\varepsilon}$ are given.
On the basis of these results, in \S7 we obtain approximations for the operator exponential $e^{-\mathcal{A}_{N,\varepsilon}t}$
and the corresponding fluxes. These are the main results of the paper for the case of the Neumann boundary condition.
In \S8, these results are applied to homogenization of solutions of the second initial boundary value problem for the equation \eqref{0.11}.

The proofs of the statements about removing the smoothing
operator in the corrector in the case of additional smoothness of the boundary and in the case
of estimates in a strictly interior subdomain are given in Appendix (\S 9).

\subsection*{0.7. Notation.}
Let $\mathfrak{H}$ and $\mathfrak{H}_*$ be complex
separable Hilbert spaces. The symbols $(\cdot,\cdot)_{\mathfrak{H}}$ and $\|\cdot\|_{\mathfrak{H}}$
stand for the inner product and the norm in $\mathfrak{H}$;
the symbol $\|\cdot\|_{\mathfrak{H} \to \mathfrak{H}_*}$
denotes the norm of a linear continuous operator acting from $\mathfrak{H}$ to $\mathfrak{H}_*$.

The symbols $\langle \cdot, \cdot \rangle$ and $|\cdot|$ stand for the inner product
and the norm in $\mathbb{C}^n$; $\mathbf{1} = \mathbf{1}_n$ is the identity $(n\times n)$-matrix.
If $a$ is an $m\times n$-matrix, then $\vert a\vert$ denotes the norm of $a$ viewed as a linear operator from $\mathbb{C}^n$ to $\mathbb{C}^m$.
We use the notation $\mathbf{x} = (x_1,\dots,x_d)\in \mathbb{R}^d$, $iD_j = \partial_j = \partial/\partial x_j$,
$j=1,\dots,d$, $\mathbf{D} = -i \nabla = (D_1,\dots,D_d)$.

The $L_2$-class of $\mathbb{C}^n$-valued functions in a domain ${\mathcal O} \subset \mathbb{R}^d$
is denoted by $L_2({\mathcal O};\mathbb{C}^n)$.
The Sobolev classes of $\mathbb{C}^n$-valued functions in a domain ${\mathcal O} \subset \mathbb{R}^d$
are denoted by $H^s({\mathcal O};\mathbb{C}^n)$.
By $H^1_0(\mathcal{O};\mathbb{C}^n)$ we denote the closure of $C_0^\infty(\mathcal{O};\mathbb{C}^n)$
in $H^1(\mathcal{O};\mathbb{C}^n)$.
If $n=1$, we write simply $L_2({\mathcal O})$, $H^s({\mathcal O})$, but sometimes
we use such abbreviated notation also for spaces of vector-valued or matrix-valued functions.
The symbol $L_p((0,T);\mathfrak{H})$, $1\leqslant p\leqslant\infty$, stands for the $L_p$-space of
$\mathfrak{H}$-valued functions on the interval $(0,T)$.

We denote $\mathbb{R}_+= [0,\infty)$. By $C$, $\mathcal{C}$, $\mathrm{C}$, $\mathfrak{C}$, $\mathscr{C}$, $c$, and $\mathfrak{c}$
(possibly, with indices and marks) we denote various constants in estimates.

A brief communication on the results of the present paper is published in \cite{MSu}.

\section*{Chapter~1. Homogenization of the first initial boundary value problem for parabolic systems}

\section*{\S 1. The class of operators $\mathcal{A}_{D,\varepsilon}$. The effective operator}
\setcounter{section}{1}
\setcounter{equation}{0}
\subsection*{1.1. The class of operators.} Let $\Gamma \subset \mathbb{R}^d$ be a lattice generated by the basis
${\mathbf a}_1, \dots, {\mathbf a}_d$:
$$
\Gamma = \{ {\mathbf a}\in \mathbb{R}^d:\  {\mathbf a}= \sum_{j=1}^d \nu_j {\mathbf a}_j,\ \nu_j \in {\mathbb Z} \}.
$$
Let $\Omega$ be the elementary cell of the lattice  $\Gamma$, i.~e.,
$$
\Omega = \{ {\mathbf x}\in \mathbb{R}^d:\  {\mathbf x}= \sum_{j=1}^d \tau_j {\mathbf a}_j,\ -\frac{1}{2}< \tau_j < \frac{1}{2} \}.
$$
We denote $\vert \Omega \vert =\textrm{meas}\, \Omega$.
The basis ${\mathbf b}_1, \dots, {\mathbf b}_d$ in $\mathbb{R}^d$ dual to the basis ${\mathbf a}_1, \dots, {\mathbf a}_d$
is defined by the relations $\langle {\mathbf b}_j, {\mathbf a}_k \rangle = \delta_{jk}$.
This basis generates a lattice $\widetilde{\Gamma}$ dual to the lattice $\Gamma$.
We use the following notation
\begin{equation}
\label{r}
r_0 = \frac{1}{2} \min_{0 \ne {\mathbf b}\in \widetilde{\Gamma}} |{\mathbf b}|,\quad r_1 = \frac{1}{2} {\rm diam}\, \Omega.
\end{equation}

Below $\widetilde{H}^1(\Omega)$ stands for the space of functions in $H^1(\Omega)$ whose  $\Gamma$-periodic extension
to $\mathbb{R}^d$ belongs to $H^1_{\textrm{loc}}(\mathbb{R}^d)$.
For any $\Gamma$-periodic function $f(\mathbf{x})$, $\mathbf{x}\in \mathbb{R}^d$, denote
$$
f ^\varepsilon (\mathbf{x}):=f(\varepsilon ^{-1}\mathbf{x}),\quad \varepsilon >0.
$$

Let $\mathcal{O}\subset \mathbb{R}^d$ be a bounded domain with the boundary of class $C^{1,1}$.
In the space $L_2(\mathcal{O};\mathbb{C}^n)$, consider the operator $\mathcal{A}_{D,\varepsilon}$
given formally by the differential expression
\begin{equation*}
\mathcal{A}_{D,\varepsilon} =b(\mathbf{D})^*g^\varepsilon (\mathbf{x})b(\mathbf{D})
\end{equation*}
with the Dirichlet condition on $\partial \mathcal{O}$. Here $g(\mathbf{x})$
is a measurable Hermitian $(m\times m)$-matrix-valued function  (in general, with complex entries)
assumed to be periodic with respect to the lattice $\Gamma$, uniformly positive definite and bounded.
Next, $b(\mathbf{D})$ is an $(m\times n)$-matrix first order DO of the form
\begin{equation}
\label{b(D)=}
b(\mathbf{D})=\sum \limits _{l=1}^d b_lD_l,
\end{equation}
where $b_l$ are constant matrices (in general, with complex entries).
The symbol $b(\boldsymbol{\xi})=\sum _{l=1}^d b_l\xi _l$, $\boldsymbol{\xi}\in \mathbb{R}^d$,
corresponds to the operator $b(\mathbf{D})$.  It is assumed that $m\geqslant n$ and
\begin{equation}
\label{rank b in R^d}
\textrm{rank}\,b(\boldsymbol{\xi})=n,\quad  0\neq \boldsymbol{\xi}\in\mathbb{R}^d.
\end{equation}
This is equivalent to the existence of the constants $\alpha _0$ and $\alpha _1$ such that
\begin{equation}
\label{<b^*b<}
\alpha _0\mathbf{1}_n\leqslant b(\boldsymbol{\theta})^*b(\boldsymbol{\theta})\leqslant \alpha _1\mathbf{1}_n,\quad \boldsymbol{\theta}\in\mathbb{S}^{d-1},\quad 0<\alpha _0\leqslant \alpha _1<\infty .
\end{equation}
From (\ref{b(D)=}) and (\ref{<b^*b<}) it follows that
\begin{equation}
\label{|b_l|<=}
\vert b_l\vert \leqslant \alpha _1^{1/2},\quad l=1,\dots ,d.
\end{equation}
The precise definition of the operator $\mathcal{A}_{D,\varepsilon}$ is given in terms of the quadratic form
\begin{equation*}
a_{D,\varepsilon} [\mathbf{u},\mathbf{u}]=\int _{\mathcal{O}} \left\langle g^\varepsilon (\mathbf{x})b(\mathbf{D})\mathbf{u},b(\mathbf{D})\mathbf{u}\right\rangle \,d\mathbf{x},\quad \mathbf{u}\in H^1_0(\mathcal{O};\mathbb{C}^n).
\end{equation*}
Under the above assumptions, we have
\begin{equation}
\label{<a_D,e<}
c_0\int _\mathcal{O}\vert \mathbf{D}\mathbf{u}\vert ^2\,d\mathbf{x}\leqslant a_{D,\varepsilon}[\mathbf{u},\mathbf{u}]\leqslant c_1\int _\mathcal{O}\vert \mathbf{D}\mathbf{u}\vert ^2\,d\mathbf{x},\quad \mathbf{u}\in H^1_0(\mathcal{O};\mathbb{C}^n),
\end{equation}
where $c_0=\alpha_0\Vert g^{-1}\Vert ^{-1}_{L_\infty}$ and $c_1=\alpha _1\Vert g\Vert _{L_\infty}$.
It is easy to check these estimates extending
$\mathbf{u}$ by zero to $\mathbb{R}^d\setminus \mathcal{O}$, using the Fourier transformation, and taking  (\ref{<b^*b<}) into account.
Thus, the form $a_{D,\varepsilon}$ is closed in $L_2(\mathcal{O};\mathbb{C}^n)$.
By the Friedrichs inequality and the lower estimate \eqref{<a_D,e<}, the form $a_{D,\varepsilon}$
is positive definite:
\begin{equation}
\label{a_D,e>=}
a_{D,\varepsilon}[\mathbf{u},\mathbf{u}]\geqslant c_* \Vert \mathbf{u}\Vert ^2_{L_2(\mathcal{O})},\quad\mathbf{u}\in H^1_0(\mathcal{O};\mathbb{C}^n),\quad c_*=c_0(\mathrm{diam}\,\mathcal{O})^{-2}.
\end{equation}

\subsection*{1.2. The effective matrix and its properties.}
To describe the effective operator, we need to introduce the effective matrix $g^0$.
Let an $(n\times m)$-matrix-valued function $\Lambda(\mathbf{x})$ be
a weak $\Gamma$-periodic solution of the problem
\begin{equation}
\label{Lambda_problem}
b(\mathbf{D})^*g(\mathbf{x})(b(\mathbf{D})\Lambda (\mathbf{x})+\mathbf{1}_m)=0,\quad \int _\Omega \Lambda (\mathbf{x})\,d\mathbf{x}=0.
\end{equation}
Denote
\begin{equation}
\label{tilde g}
\widetilde{g}(\mathbf{x}):=g(\mathbf{x})(b(\mathbf{D})\Lambda(\mathbf{x})+\mathbf{1}_m).
\end{equation}
The \textit{effective matrix} $g^0$ of size $m\times m$ is defined by the relation
\begin{equation}
\label{g^0}
g^0=\vert \Omega \vert ^{-1}\int _\Omega \widetilde{g}(\mathbf{x})\,d\mathbf{x}.
\end{equation}
It turns out that $g^0$ is positive definite.

We need the following properties of the effective matrix checked in \cite[Chapter~3, Theorem~1.5]{BSu}.

\begin{proposition} The effective matrix satisfies the following estimates
\begin{equation}
\label{Foigt-Reiss}
\underline{g}\leqslant g^0\leqslant \overline{g}.
\end{equation}
Here
\begin{equation*}
\overline{g}=\vert \Omega \vert ^{-1}\int _\Omega g(\mathbf{x})\,d\mathbf{x},\quad \underline{g}=\biggl( \vert \Omega \vert ^{-1}\int _\Omega g(\mathbf{x})^{-1}\,d\mathbf{x}\biggr) ^{-1}.
\end{equation*}
If $m=n$, then the effective matrix $g^0$ coincides with $\underline{g}$.
\end{proposition}

For specific DO's, estimates (\ref{Foigt-Reiss}) are known in homogenization theory as the Voight-Reuss bracketing.
Now we distinguish the cases where one of the inequalities in (\ref{Foigt-Reiss}) becomes an identity
(see \cite[Chapter~3, Propositions~1.6~and~1.7]{BSu}).

\begin{proposition} The identity $g^0 =\overline{g}$ is equivalent to the relations
\begin{equation}
\label{overline-g}
b(\D)^* {\mathbf g}_k(\x) =0,\ \ k=1,\dots,m,
\end{equation}
where ${\mathbf g}_k(\x)$, $k=1,\dots,m,$ are the columns of the matrix $g(\x)$.
\end{proposition}

\begin{proposition} The identity $g^0 =\underline{g}$ is equivalent to the relations
\begin{equation}
\label{underline-g}
{\mathbf l}_k(\x) = {\mathbf l}_k^0 + b(\D) {\mathbf w}_k,\ \ {\mathbf l}_k^0\in \C^m,\ \
{\mathbf w}_k \in \widetilde{H}^1(\Omega;\C^m),\ \ k=1,\dots,m,
\end{equation}
where ${\mathbf l}_k(\x)$, $k=1,\dots,m,$ are the columns of the matrix $g(\x)^{-1}$.
\end{proposition}

Obviously, (\ref{Foigt-Reiss}) implies the following estimates for the norms of the matrices $g^0$ and $(g^0)^{-1}$:
\begin{equation}
\label{|g^0|, |g^0^_1|<=}
\vert g^0\vert \leqslant \Vert g\Vert _{L_\infty},\quad \vert (g^0)^{-1}\vert \leqslant \Vert g^{-1}\Vert _{L_\infty}.
\end{equation}

\subsection*{1.3. The effective operator.}
The selfadjoint operator $\mathcal{A}_D^0$ in $L_2(\mathcal{O};\mathbb{C}^n)$
generated by the quadratic form
\begin{equation*}
a_{D}^0 [\mathbf{u},\mathbf{u}]=\int _{\mathcal{O}} \left\langle g^0 b(\mathbf{D})\mathbf{u},b(\mathbf{D})\mathbf{u}\right\rangle \,d\mathbf{x},\quad \mathbf{u}\in H^1_0(\mathcal{O};\mathbb{C}^n),
\end{equation*}
is called the \textit{effective operator} for the operator $\mathcal{A}_{D,\varepsilon}$
From (\ref{|g^0|, |g^0^_1|<=}) it follows that the form $a_D^0$ satisfies the following estimates
similar to (\ref{<a_D,e<}) and \eqref{a_D,e>=}:
\begin{align}
\label{<a_D,0<}
&c_0\int _\mathcal{O}\vert \mathbf{D}\mathbf{u}\vert ^2\,d\mathbf{x}\leqslant a_D^0[\mathbf{u},\mathbf{u}]\leqslant c_1\int _\mathcal{O}\vert \mathbf{D}\mathbf{u}\vert ^2\,d\mathbf{x},\quad \mathbf{u}\in H^1_0(\mathcal{O};\mathbb{C}^n),
\\
\label{a_D,0>=}
&a_D^0[\mathbf{u},\mathbf{u}]\geqslant c_*\Vert \mathbf{u}\Vert ^2_{L_2(\mathcal{O})},\quad\mathbf{u}\in H^1_0(\mathcal{O};\mathbb{C}^n).
\end{align}

By the condition $\partial \mathcal{O}\in C^{1,1}$, the operator $\mathcal{A}_D^0$
is given by the differential expression $b(\mathbf{D})^*g^0b(\mathbf{D})$
on the domain $H^1_0(\mathcal{O};\mathbb{C}^n)\cap H^2(\mathcal{O};\mathbb{C}^n)$.
The inverse operator satisfies
\begin{equation}
\label{A0^-1}
\Vert \left( \mathcal{A}_D^0\right)^{-1}\Vert _{L_2(\mathcal{O})\rightarrow H^2(\mathcal{O})}\leqslant\widehat{c}.
\end{equation}
Here the constant $\widehat{c}$ depends only on $\alpha _0$, $\alpha _1$, $\Vert g\Vert _{L_\infty}$, $\Vert g^{-1}\Vert _{L_\infty}$,
and the domain $\mathcal{O}$.
To justify this fact, we note that the operator $b(\mathbf{D})^*g^0 b(\mathbf{D})$
is a strongly elliptic matrix operator (with constant coefficients) and refer to
the theorems about regularity of solutions of strongly elliptic equations  (see, e.~g., \cite[Chapter~4]{McL}).

\begin{remark}
\textnormal{Instead of condition $\partial \mathcal{O}\in C^{1,1}$, one can impose the following implicit condition on the domain:
a bounded domain $\mathcal{O}$ with Lipschitz boundary is such that
estimate \eqref{A0^-1} is satisfied. For such domain the results of Chapter 1 remain valid
(except for the results of Subsection~3.6, Theorem 4.3, and Propositions~4.13 and 4.20, where
an additional smoothness of the boundary is required). In the case of scalar elliptic operators wide conditions
on   $\partial \mathcal{O}$ sufficient for estimate \eqref{A0^-1} can be found in
[KoE] and [MaSh, Chapter 7] (in particular, it suffices that $\partial \mathcal{O}\in C^{\alpha}$, $\alpha >3/2$).
}
\end{remark}

\subsection*{1.4. Properties of the matrix-valued function $\Lambda({\mathbf x})$.}
The matrix-valued function $\Lambda({\mathbf x})$ belongs to $\widetilde{H}^1(\Omega)$.
The following estimate was checked in  [BSu2, (6.28) and Subsection~7.3]:
\begin{equation}
\label{M}
\|\Lambda\|_{H^1(\Omega)} \leqslant M:=\left( |\Omega| (1+ (2r_0)^{-2})m \alpha_0^{-1} \|g\|_{L_\infty}
\|g^{-1}\|_{L_\infty}\right)^{1/2},
\end{equation}
where $r_0$ is defined by \eqref{r}.

We need the following statement proved in [PSu2, Lemma~2.3].

\begin{lemma} Suppose that the matrix-valued function $\Lambda(\mathbf{x})$ is a $\Gamma$-periodic solution of the problem \textnormal{(\ref{Lambda_problem})}. Then for any $\mathbf{v} \in C_0^\infty({\mathbb R}^d;{\mathbb C}^m)$
we have
\begin{equation*}
\begin{split}
\int_{{\mathbb R}^d} | ({\mathbf D}\Lambda)^\varepsilon(\mathbf{x}) \mathbf{v}(\mathbf{x})|^2\,d{\mathbf x}
&\leqslant
\beta_1 \int_{{\mathbb R}^d} |  \mathbf{v}(\mathbf{x})|^2\,d{\mathbf x} \\
&+
\beta_2 \varepsilon^2  \int_{{\mathbb R}^d} | \Lambda^\varepsilon(\mathbf{x})|^2 |{\mathbf D}\mathbf{v}(\mathbf{x})|^2\,d{\mathbf x}.
\end{split}
\end{equation*}
The constants  $\beta_1$ and $\beta_2$ depend only on $m$, $d$, $\alpha_0$, $\alpha_1$,
$\|g\|_{L_\infty}$, and $\|g^{-1}\|_{L_\infty}$.
\end{lemma}

Lemma 1.5 directly implies the following corollary.

\begin{corollary} Suppose that the assumptions of Lemma {\rm 1.5} are satisfied and $\Lambda\in L_\infty$.
Then for any $\mathbf{v} \in H^1({\mathbb R}^d;{\mathbb C}^m)$ we have
$\Lambda^\varepsilon \mathbf{v} \in H^1({\mathbb R}^d;{\mathbb C}^n)$ and
\begin{equation*}
\begin{split}
\|\Lambda^\varepsilon \mathbf{v} \|_{H^1({\mathbb R}^d)}^2
\leqslant
2 \beta_1 \varepsilon^{-2} \| \mathbf{v} \|_{L_2({\mathbb R}^d)}^2
+2 (1+\beta_2) \| \Lambda\|_{L_\infty}^2 \| \mathbf{v} \|_{H^1({\mathbb R}^d)}^2.
\end{split}
\end{equation*}
\end{corollary}

\subsection*{1.5. The Steklov smoothing.}
Let  $S_\varepsilon$ be the operator in $L_2(\mathbb{R}^d;\mathbb{C}^m)$ given by
\begin{equation}
\label{S_e}
(S_\varepsilon \mathbf{u})(\mathbf{x})=\vert \Omega \vert ^{-1}\int _\Omega \mathbf{u}(\mathbf{x}-\varepsilon \mathbf{z})\,d\mathbf{z}
\end{equation}
and called the \textit{Steklov smoothing operator.} Note that
$\Vert S_\varepsilon\Vert _{L_2(\mathbb{R}^d)\rightarrow L_2(\mathbb{R}^d)}\leqslant 1$.
Let $s\geqslant 0$. Obviously, $\mathbf{D}^\alpha S_\varepsilon \mathbf{u}=S_\varepsilon \mathbf{D}^\alpha \mathbf{u}$
for $\mathbf{u}\in H^s(\mathbb{R}^d;\mathbb{C}^m)$ and any multiindex $\alpha$ such that $\vert \alpha \vert \leqslant s$.
The operator $S_\eps$ is a continuous mapping of $H^s(\mathbb{R}^d;\mathbb{C}^m)$ to $H^s(\mathbb{R}^d;\mathbb{C}^m)$,
and $\Vert S_\varepsilon\Vert _{H^s(\mathbb{R}^d)\rightarrow H^s(\mathbb{R}^d)}\leqslant 1$.

We need the following properties of the operator (\ref{S_e}) (see \cite[Lemmas~1.1~and~1.2]{ZhPas} or
\cite[Propositions~3.1~and~3.2]{PSu}).

\begin{lemma} Let $f(\mathbf{x})$ be a $\Gamma$-periodic function in $\mathbb{R}^d$ such that $f\in L_2(\Omega)$.
Let $[f^\varepsilon]$ be the operator of multiplication by the function $f^\varepsilon (\mathbf{x})$.
Then the operator $[f^\varepsilon]S_\varepsilon$ is continuous in $L_2(\mathbb{R}^d;\mathbb{C}^m)$, and
\begin{equation*}
\Vert [f^\varepsilon ]S_\varepsilon\Vert _{L_2(\mathbb{R}^d)\rightarrow L_2(\mathbb{R}^d)}\leqslant \vert \Omega\vert ^{-1/2}\Vert f\Vert _{L_2(\Omega)}.
\end{equation*}
\end{lemma}

\begin{lemma}
Let $r_1$ be defined by \textnormal{(\ref{r})}. Then for any ${\mathbf v}\in H^1(\mathbb{R}^d;\mathbb{C}^m)$ we have
\begin{equation*}
\Vert (S_\varepsilon -I) {\mathbf v} \Vert _{L_2(\mathbb{R}^d)}
\leqslant  \varepsilon r_1 \Vert {\mathbf D}{\mathbf v}\Vert _{L_2(\mathbb{R}^d)}.
\end{equation*}
\end{lemma}

\section*{\S 2. Preliminaries. \\Approximation of the resolvent of the operator $\mathcal{A}_{D,\varepsilon}$}
\setcounter{section}{2}
\setcounter{equation}{0}
\setcounter{theorem}{0}

We choose the numbers $\varepsilon _0, \varepsilon _1\in (0,1]$ according to the following condition.
\begin{condition}
The number $\varepsilon _0\in(0,1]$ is such that the strip
\begin{equation*}
(\partial \mathcal{O})_{\varepsilon _0} :=\lbrace \mathbf{x}\in \mathbb{R}^d :\textnormal{dist}\,\lbrace \mathbf{x} ;\partial \mathcal{O}\rbrace <\varepsilon _0 \rbrace
\end{equation*}
can be covered by a finite number of open sets admitting diffeomorphisms
of class $C^{0,1}$ rectifying the boundary $\partial \mathcal{O}$.
Let $\varepsilon _1=\varepsilon _0(1+r_1)^{-1}$, where $2r_1=\textnormal{diam}\,\Omega$.
\end{condition}
Obviously, $\varepsilon _1$ depends only on the domain $\mathcal{O}$ and the lattice $\Gamma$.

Note that Condition 2.1 would be provided only by the condition that $\partial {\mathcal O}$ is Lipschitz;
we have imposed a stronger restriction $\partial {\mathcal O}\in C^{1,1}$ in order to ensure estimate \eqref{A0^-1}.

\subsection*{2.1. Approximation of the resolvent of the operator $\mathcal{A}_{D,\varepsilon}$ in the $L_2(\mathcal{O};\mathbb{C}^n)$-operator
norm} We need the results about the behaviour of the resolvent $(\mathcal{A}_{D,\varepsilon}-\zeta I)^{-1}$ obtained in
[Su7,8].  In \cite[Theorem 4.1]{Su14-2},  the principal term of approximation of the resolvent was found
for $\zeta\in \mathbb{C}\setminus \mathbb{R}_+$ such that $\vert \zeta\vert \geqslant 1$. Besides, in \cite[Theorem 8.1]{Su14-2}
the $(L_2\rightarrow L_2)$-estimate for the difference of the resolvents
$(\mathcal{A}_{D,\varepsilon}-\zeta I)^{-1}$ and $(\mathcal{A}_D^0-\zeta I)^{-1}$ was obtained for
$\zeta\in\mathbb{C}\setminus [c_\circ,\infty)$, where $c_\circ$ is a common lower bound of the operators $\mathcal{A}_{D,\varepsilon}$ and $\mathcal{A}_D^0$. Using \eqref{a_D,e>=} and \eqref{a_D,0>=},
we take $c_\circ$ equal to $c_*= c_0(\textnormal{diam}\,\mathcal{O})^{-2}$.
Below the following notation is systematically used, where $\varphi \in (0,2\pi)$:
\begin{equation}
\label{c(phi)}
c(\varphi)=\begin{cases} \vert \sin \varphi \vert ^{-1}, &\varphi \in (0,\pi /2)\cup (3\pi /2,2\pi),\\
1,&\varphi\in [\pi/2 ,3\pi /2 ] .
\end{cases}
\end{equation}

For convenience of further references, the following set of parameters is called the \textit{initial data}:
\begin{equation}
\label{data}
\begin{aligned}
&m,\ d,\ \alpha _0,\ \alpha _1,\ \Vert g\Vert _{L_\infty},\ \Vert g^{-1}\Vert _{L_\infty};
\\
&\textnormal{the parameters of the lattice}\ \Gamma; \ \  \textnormal{the domain} \ \mathcal{O}.
\end{aligned}
\end{equation}

\begin{theorem}[\cite{Su14-2}] Suppose that  $\mathcal{O}\subset \mathbb{R}^d$ is a bounded domain with the boundary of class $C^{1,1}$.
Suppose that the matrix-valued function $g(\mathbf{x})$ and DO $b(\mathbf{D})$ satisfy the assumptions of Subsection~\textnormal{1.1}.
Let $\mathcal{A}_{D,\varepsilon}$ and $\mathcal{A}_{D}^0$ be the operators defined in Subsections~\textnormal{1.1, 1.3}.
Suppose that the number $\varepsilon_1$ is subject to Condition~\textnormal{2.1}.

\noindent $1^\circ$. Let $\zeta \in \mathbb{C}\setminus\mathbb{R}_+$, $\zeta =\vert \zeta \vert e^{i\phi}$, $\phi\in (0,2\pi)$,
and $\vert \zeta \vert \geqslant 1$. Then for $0<\varepsilon\leqslant\varepsilon _1$ we have
\begin{align*}
\Vert (\mathcal{A}_{D,\varepsilon }-\zeta I)^{-1}-(\mathcal{A}_D^0-\zeta I)^{-1}\Vert _{L_2(\mathcal{O})\rightarrow L_2(\mathcal{O})}\leqslant
C_1c(\phi)^5(\vert \zeta \vert ^{-1/2}\varepsilon +\varepsilon ^2).
\end{align*}

\noindent $2^\circ$. Now, let $\zeta \in \mathbb{C}\setminus [c_*,\infty)$, where
$c_*=\alpha _0\Vert g^{-1}\Vert ^{-1}_{L_\infty}(\textnormal{diam}\,\mathcal{O})^{-2}$.
We put $\zeta -c_*=\vert \zeta -c_*\vert e^{i\psi}$, $\psi\in (0,2\pi)$, and denote
\begin{equation*}
\rho _*(\zeta)=\begin{cases}
c(\psi)^2\vert \zeta -c_*\vert ^{-2}, &\vert \zeta - c_*\vert <1,\\
c(\psi)^2, &\vert \zeta -c_*\vert \geqslant 1.
\end{cases}
\end{equation*}
Then for $0<\varepsilon\leqslant \varepsilon _1$ we have
\begin{align}
\label{Th_Su_small_zeta}
&\Vert (\mathcal{A}_{D,\varepsilon }-\zeta I)^{-1}-(\mathcal{A}_D^0-\zeta I)^{-1}\Vert _{L_2(\mathcal{O})\rightarrow L_2(\mathcal{O})}\leqslant
C_2\rho _*(\zeta )\varepsilon .
\end{align}
The constants $C_1$ and $C_2$ depend only on the initial data \textnormal{\eqref{data}}.
\end{theorem}

\subsection*{2.2. Approximation of the resolvent of the operator $\mathcal{A}_{D,\varepsilon}$ in the norm of operators acting from
$L_2(\mathcal{O};\mathbb{C}^n)$ to $H^1(\mathcal{O};\mathbb{C}^n)$}
Approximation of the resolvent $(\mathcal{A}_{D,\varepsilon}-\zeta I)^{-1}$ in the $(L_2\rightarrow H^1)$-norm was obtained in
Theorems 4.2 and 8.1 from \cite{Su14-2}. In order to formulate the result, we introduce the corrector $K_D(\eps ;\zeta)$.
For this, we fix a linear continuous extension operator
\begin{equation}
\label{P_O H^1, H^2}
\begin{split}
P_\mathcal{O}: H^s(\mathcal{O};\mathbb{C}^n)\rightarrow H^s(\mathbb{R}^d;\mathbb{C}^n),\quad s \geqslant 0.
\end{split}
\end{equation}
Such a ,,universal'' extension operator exists for any bounded domain with Lipschitz boundary (see \cite{St} or [R]).
Herewith,
\begin{equation}
\label{PO}
\| P_{\mathcal O} \|_{H^s({\mathcal O}) \to H^s({\mathbb R}^d)} \leqslant C_{\mathcal O}^{(s)},
\end{equation}
where the constant $C_{\mathcal O}^{(s)}$ depends only on $s$ and the domain ${\mathcal O}$.
By $R_\mathcal{O}$ we denote the operator of restriction of functions in $\mathbb{R}^d$ onto the domain $\mathcal{O}$.
We put
\begin{equation}
\label{K_D(e)}
K_D(\eps ;\zeta)=R_\mathcal{O}[\Lambda ^\varepsilon] S_\varepsilon b(\mathbf{D})P_\mathcal{O}(\mathcal{A}_D^0-\zeta I)^{-1}.
\end{equation}
The corrector $K_D(\eps ;\zeta)$ is a continuous mapping of $L_2(\mathcal{O};\mathbb{C}^n)$ to $H^1(\mathcal{O};\mathbb{C}^n)$.
Indeed, the operator $b(\mathbf{D})P_\mathcal{O}(\mathcal{A}_D^0-\zeta I)^{-1}$ is continuous from
$L_2(\mathcal{O};\mathbb{C}^n)$ to $H^1(\mathbb{R}^d;\mathbb{C}^m)$. By Lemma~1.7 and relation
$\Lambda \in \widetilde{H}^1(\Omega)$, it is easily seen that the operator
$[\Lambda^\varepsilon ]S_\varepsilon $ is continuous from $H^1(\mathbb{R}^d;\mathbb{C}^m)$ to $H^1(\mathbb{R}^d;\mathbb{C}^n)$.

\begin{theorem}[\cite{Su14-2}] Suppose that the assumptions of Theorem~\textnormal{2.2} are satisfied.
Let $K_D(\eps; \zeta)$ be the operator~\textnormal{(\ref{K_D(e)})}, and let
$\widetilde{g}$ be the matrix-valued function~\textnormal{(\ref{tilde g})}.

\noindent $1^\circ$. Let $\zeta \in \mathbb{C}\setminus\mathbb{R}_+$, $\zeta =\vert \zeta \vert e^{i\phi}$, $\phi\in (0,2\pi)$,
and $\vert \zeta \vert \geqslant 1$. Then for $0<\varepsilon\leqslant\varepsilon _1$ we have
\begin{align*}
\begin{split}
\Vert &(\mathcal{A}_{D,\varepsilon}-\zeta I)^{-1}-(\mathcal{A}_D^0-\zeta I)^{-1}-\varepsilon K_D(\varepsilon;\zeta )\Vert _{L_2(\mathcal{O})\rightarrow H^1(\mathcal{O})}\\
&\leqslant
C_3c(\phi)^2\vert \zeta \vert ^{-1/4}\varepsilon ^{1/2}+C_4c(\phi)^4\varepsilon ,
\end{split}
\\
\begin{split}
\Vert &g^\varepsilon b(\mathbf{D})(\mathcal{A}_{D,\varepsilon}-\zeta I)^{-1}-\widetilde{g}^\varepsilon S_\varepsilon b(\mathbf{D})P_\mathcal{O}(\mathcal{A}_D^0-\zeta I)^{-1}\Vert _{L_2(\mathcal{O})\rightarrow L_2(\mathcal{O})}\\
&\leqslant
\widetilde{C}_3c(\phi)^2\vert \zeta \vert ^{-1/4}\varepsilon ^{1/2}+\widetilde{C}_4c(\phi)^4\varepsilon .
\end{split}
\end{align*}
\noindent $2^\circ$. Now, let $\zeta \in \mathbb{C}\setminus [c_*,\infty)$. Then for $0<\varepsilon\leqslant \varepsilon _1$ we have
\begin{align}
\label{24a}
&\Vert (\mathcal{A}_{D,\varepsilon}-\zeta I)^{-1}-(\mathcal{A}_D^0-\zeta I)^{-1}-\varepsilon K_D(\varepsilon;\zeta )\Vert _{L_2(\mathcal{O})\rightarrow H^1(\mathcal{O})}\leqslant
C_5\rho _*(\zeta)\varepsilon ^{1/2},\\
\label{24b}
\begin{split}
&\Vert g^\varepsilon b(\mathbf{D})(\mathcal{A}_{D,\varepsilon}-\zeta I)^{-1}-\widetilde{g}^\varepsilon S_\varepsilon b(\mathbf{D})P_\mathcal{O}(\mathcal{A}_D^0-\zeta I)^{-1}\Vert _{L_2(\mathcal{O})\rightarrow L_2(\mathcal{O})}\\
&\leqslant \widetilde{C}_5\rho _*(\zeta)\varepsilon ^{1/2}.
\end{split}
\end{align}
The constants $C_3$, $C_4$, $\widetilde{C}_3$, $\widetilde{C}_4$, $C_5$, and $\widetilde{C}_5$ depend only on the initial data~\textnormal{\eqref{data}}.
\end{theorem}

\begin{remark}
\textnormal{In Theorems 2.2 and 2.3, the number
$c_*=\alpha _0\Vert g^{-1}\Vert ^{-1}_{L_\infty}(\textnormal{diam}\,\mathcal{O})^{-2}$
can be replaced by any common lower bound $c_\circ$
of the operators $\mathcal{A}_{D,\varepsilon}$ and $\mathcal{A}_D^0$. Let $\kappa >0$ be arbitrarily small number.
Due to the resolvent convergence, if $\varepsilon$ is sufficiently small, one can take $c_\circ = \lambda _1^0-\kappa$,
where $\lambda _1^0$ is the first eigenvalue of the operator $\mathcal{A}_D^0$.
Under such choice of $c_\circ$ the constants in estimates
(\ref{Th_Su_small_zeta}), (\ref{24a}), and (\ref{24b}) will depend on $\kappa$.
}
\end{remark}

\subsection*{2.3. The case where $\Lambda \in L_\infty$.} It turns out that,
under some additional assumption on the matrix-valued function $\Lambda(\mathbf{x})$, the smoothing operator
$S_\varepsilon$ in the corrector can be removed.

\begin{condition} Suppose that the $\Gamma$-periodic solution $\Lambda(\mathbf{x})$ of problem
\textnormal{(\ref{Lambda_problem})} is bounded: $\Lambda \in L_\infty$.
\end{condition}

We introduce the operator
\begin{equation}
\label{K_D^0(e)}
K_D^0(\varepsilon ;\zeta )=[\Lambda ^\varepsilon]b(\mathbf{D})(\mathcal{A}_D^0-\zeta I)^{-1}.
\end{equation}
By (\ref{A0^-1}), the operator $b(\mathbf{D})(\mathcal{A}_D^0-\zeta I)^{-1}$ is continuous from
$L_2(\mathcal{O};\mathbb{C}^n)$ to $H^1(\mathcal{O};\mathbb{C}^m)$.
Applying the extension operator $P_\mathcal{O}$ and Corollary~1.6, we check that under Condition~2.5 the operator
$[\Lambda ^\varepsilon]$ of multiplication by the matrix-valued function $\Lambda ^\varepsilon (\mathbf{x})$ is continuous from
$H^1(\mathcal{O};\mathbb{C}^m)$ to $H^1(\mathcal{O};\mathbb{C}^n)$. Hence, the operator (\ref{K_D^0(e)}) is a continuous mapping of $L_2(\mathcal{O};\mathbb{C}^n)$ to $H^1(\mathcal{O};\mathbb{C}^n)$.
The following result was obtained in \cite[Theorems~6.1~and~8.3]{Su14-2}.

\begin{theorem}[\cite{Su14-2}] Suppose that the assumptions of Theorem~\textnormal{2.2} are satisfied.
Suppose that the matrix-valued function $\Lambda(\mathbf{x})$ satisfies Condition~\textnormal{2.5}.
Let $K_D^0(\varepsilon ;\zeta)$ be the operator~\textnormal{(\ref{K_D^0(e)})}. Let $\widetilde{g}$ be the matrix-valued function~\textnormal{\eqref{tilde g}}.

\noindent $1^\circ$. Let $\zeta \in \mathbb{C}\setminus\mathbb{R}_+$, $\zeta =\vert \zeta \vert e^{i\phi}$, $\phi\in (0,2\pi)$, and
$\vert \zeta \vert \geqslant 1$. Then for $0<\varepsilon\leqslant\varepsilon _1$ we have
\begin{align*}
\begin{split}
\Vert &(\mathcal{A}_{D,\varepsilon}-\zeta I)^{-1}-(\mathcal{A}_D^0-\zeta I)^{-1}-\varepsilon K_D^0(\varepsilon;\zeta )\Vert _{L_2(\mathcal{O})\rightarrow H^1(\mathcal{O})}\\
&\leqslant
C_3c(\phi)^2\vert \zeta \vert ^{-1/4}\varepsilon ^{1/2}+C_6c(\phi)^4\varepsilon ,
\end{split}
\\
\begin{split}
\Vert &g^\varepsilon b(\mathbf{D})(\mathcal{A}_{D,\varepsilon}-\zeta I)^{-1}-\widetilde{g}^\varepsilon b(\mathbf{D}) (\mathcal{A}_D^0-\zeta I)^{-1}\Vert _{L_2(\mathcal{O})\rightarrow L_2(\mathcal{O})}\\
&\leqslant
\widetilde{C}_3c(\phi)^2\vert \zeta \vert ^{-1/4}\varepsilon ^{1/2}+\widetilde{C}_6c(\phi)^4\varepsilon
.
\end{split}
\end{align*}

\noindent $2^\circ$. Now, let $\zeta \in \mathbb{C}\setminus [c_*,\infty)$. Then for $0<\varepsilon\leqslant \varepsilon _1$ we have
\begin{align*}
&\Vert (\mathcal{A}_{D,\varepsilon}-\zeta I)^{-1}-(\mathcal{A}_D^0-\zeta I)^{-1}-\varepsilon K_D^0(\varepsilon;\zeta )\Vert _{L_2(\mathcal{O})\rightarrow H^1(\mathcal{O})}\leqslant
C_7\rho _*(\zeta)\varepsilon ^{1/2},\\
&\Vert g^\varepsilon b(\mathbf{D})(\mathcal{A}_{D,\varepsilon}-\zeta I)^{-1}-\widetilde{g}^\varepsilon  b(\mathbf{D})(\mathcal{A}_D^0-\zeta I)^{-1}\Vert _{L_2(\mathcal{O})\rightarrow L_2(\mathcal{O})}\leqslant \widetilde{C}_7\rho _*(\zeta)\varepsilon ^{1/2}.
\end{align*}
The constants $C_3$ and $\widetilde{C}_3$ are the same as in Theorem~\textnormal{2.3}. The constants $C_6$, $\widetilde{C}_6$, $C_7$, and $\widetilde{C}_7$ depend only on the initial data~\textnormal{\eqref{data}} and $\Vert \Lambda \Vert _{L_\infty}$.
\end{theorem}

In some cases Condition 2.5 is valid automatically. The next statement was proved in \cite[Lemma~8.7]{BSu06};
see also Remark~8.8 in \cite{BSu06}.

\begin{proposition} Condition \textnormal{2.5} is a fortiori valid if at least one of the following
assumptions is fulfilled:

\noindent $1^\circ$ dimension does not exceed two, i.~e., $d\leqslant 2;$

\noindent $2^\circ$
dimension $d$ is arbitrary, and the operator has the form
$\mathcal{A}_{D,\varepsilon }=\mathbf{D}^*g^\varepsilon (\mathbf{x})\mathbf{D}$, where the matrix $g(\mathbf{x})$
has real entries\textnormal{;}

\noindent $3^\circ$ dimension $d$ is arbitrary, and the effective matrix $g^0$ satisfies $g^0=\underline{g}$, i.~e.,
relations \textnormal{(\ref{underline-g})} hold.
\end{proposition}

\begin{remark}
\textnormal{
Under the assumptions of assertion $2^\circ$ the norm $\|\Lambda\|_{L_\infty}$ is controlled in terms of $d$,
$\|g\|_{L_\infty}$, $\|g^{-1}\|_{L_\infty}$, and $\Gamma$.
Under the assumptions of assertion $3^\circ$ the norm $\|\Lambda\|_{L_\infty}$ is controlled in terms of $d$, $m$, $n$, $\alpha_0$, $\alpha_1$,
$\|g\|_{L_\infty}$, $\|g^{-1}\|_{L_\infty}$, and the parameters of the lattice $\Gamma$.}
\end{remark}

\subsection*{2.4. Estimates in a strictly interior subdomain}
In a strictly interior subdomain $\mathcal{O}'$ of the domain $\mathcal{O}$ it is possible to improve error estimates in $H^1$.
The following result was proved in Theorems~7.1 and 8.6 from \cite{Su14-2}.

\begin{theorem}[\cite{Su14-2}]
Suppose that the assumptions of Theorem~\textnormal{2.3} are satisfied. Let $\mathcal{O}'$ be a strictly interior subdomain of the domain $\mathcal{O}$.
Denote $\delta :=\textnormal{dist}\,\lbrace \mathcal{O}';\partial\mathcal{O}\rbrace$.

\noindent $1^\circ$. Let $\zeta \in \mathbb{C}\setminus\mathbb{R}_+$ and $\vert \zeta \vert \geqslant 1$. Then for
$0<\varepsilon\leqslant\varepsilon _1$ we have
\begin{align*}
\begin{split}
\Vert &(\mathcal{A}_{D,\varepsilon}-\zeta I)^{-1}-(\mathcal{A}_D^0-\zeta I)^{-1}-\varepsilon K_D(\varepsilon;\zeta )\Vert _{L_2(\mathcal{O})\rightarrow H^1(\mathcal{O}')}\\
&\leqslant
(C_8\delta ^{-1}+C_9)c(\phi)^6\varepsilon ,
\end{split}
\\
\begin{split}
\Vert &g^\varepsilon b(\mathbf{D})(\mathcal{A}_{D,\varepsilon}-\zeta I)^{-1}-\widetilde{g}^\varepsilon S_\varepsilon b(\mathbf{D}) P_\mathcal{O}(\mathcal{A}_D^0-\zeta I)^{-1}\Vert _{L_2(\mathcal{O})\rightarrow L_2(\mathcal{O}')}\\
&\leqslant
(\widetilde{C}_8 \delta ^{-1}+\widetilde{C}_9 )c(\phi)^6\varepsilon.
\end{split}
\end{align*}

\noindent $2^\circ$. Let $\zeta \in \mathbb{C}\setminus [c_*,\infty)$. Denote
$\widehat{\rho}(\zeta):= c(\psi)\rho _*(\zeta)+c(\psi)^{5/2}\rho _*(\zeta )^{3/4}$.
For $0<\varepsilon\leqslant \varepsilon _1$ we have
\begin{align*}
\begin{split}
\Vert &(\mathcal{A}_{D,\varepsilon}-\zeta I)^{-1}-(\mathcal{A}_D^0-\zeta I)^{-1}-\varepsilon K_D(\varepsilon;\zeta )\Vert _{L_2(\mathcal{O})\rightarrow H^1(\mathcal{O}')}\\
&\leqslant
\left( C_{10}\delta ^{-1}  \widehat{\rho}(\zeta)
+C_{11}c(\psi)^{1/2}\rho _*(\zeta)^{5/4}\right)\varepsilon ,
\end{split}\\
\begin{split}
\Vert &g^\varepsilon b(\mathbf{D})(\mathcal{A}_{D,\varepsilon}-\zeta I)^{-1}-\widetilde{g}^\varepsilon S_\varepsilon b(\mathbf{D})P_\mathcal{O}(\mathcal{A}_D^0-\zeta I)^{-1}\Vert _{L_2(\mathcal{O})\rightarrow L_2(\mathcal{O}')}\\
&\leqslant
\left( \widetilde{C}_{10} \delta ^{-1} \widehat{\rho}(\zeta)
+\widetilde{C}_{11}c(\psi)^{1/2}\rho _*(\zeta)^{5/4}\right)\varepsilon.
\end{split}
\end{align*}
The constants $C_8$, $C_9$, $\widetilde{C}_8$, $\widetilde{C}_9$, $C_{10}$, $C_{11}$, $\widetilde{C}_{10}$, and
$\widetilde{C}_{11}$ depend only on the initial data~\textnormal{\eqref{data}}.
\end{theorem}

In the case where $\Lambda \in L_\infty$ a similar result holds with a simpler corrector
$K_D^0(\varepsilon;\zeta)$. This was shown in Theorems 7.2 and 8.7 from \cite{Su14-2}.

\begin{theorem}[\cite{Su14-2}]
Suppose that the assumptions of Theorem \textnormal{2.9} are satisfied.
Suppose that the matrix-valued function $\Lambda(\mathbf{x})$ satisfies Condition \textnormal{2.5}.
Let $K_D^0(\varepsilon ;\zeta)$ be the operator \eqref{K_D^0(e)}.

\noindent $1^\circ$. Let $\zeta \in \mathbb{C}\setminus\mathbb{R}_+$ and $\vert \zeta \vert \geqslant 1$.
Then for $0<\varepsilon\leqslant\varepsilon _1$ we have
\begin{align*}
\begin{split}
\Vert &(\mathcal{A}_{D,\varepsilon}-\zeta I)^{-1}-(\mathcal{A}_D^0-\zeta I)^{-1}-\varepsilon K_D^0(\varepsilon;\zeta )\Vert _{L_2(\mathcal{O})\rightarrow H^1(\mathcal{O}')}\\
&\leqslant
( {C}_8\delta ^{-1}+\check{C}_9)c(\phi)^6\varepsilon ,
\end{split}
\\
\begin{split}
\Vert &g^\varepsilon b(\mathbf{D})(\mathcal{A}_{D,\varepsilon}-\zeta I)^{-1}-\widetilde{g}^\varepsilon b(\mathbf{D}) (\mathcal{A}_D^0-\zeta I)^{-1}\Vert _{L_2(\mathcal{O})\rightarrow L_2(\mathcal{O}')}\\
&\leqslant
(\widetilde{C}_8 \delta ^{-1}+ \widehat{C}_9)c(\phi)^6\varepsilon.
\end{split}
\end{align*}

\noindent $2^\circ$. For $\zeta \in \mathbb{C}\setminus [c_*,\infty)$  and $0<\varepsilon\leqslant \varepsilon _1$ we have
\begin{align*}
\begin{split}
\Vert &(\mathcal{A}_{D,\varepsilon}-\zeta I)^{-1}-(\mathcal{A}_D^0-\zeta I)^{-1}-\varepsilon K_D^0(\varepsilon;\zeta )\Vert _{L_2(\mathcal{O})\rightarrow H^1(\mathcal{O}')}\\
&\leqslant
\left( C_{10}\delta ^{-1}  \widehat{\rho}(\zeta)
+\check{C}_{11}c(\psi)^{1/2}\rho _*(\zeta)^{5/4}\right)\varepsilon ,
\end{split}\\
\begin{split}
\Vert &g^\varepsilon b(\mathbf{D})(\mathcal{A}_{D,\varepsilon}-\zeta I)^{-1}-\widetilde{g}^\varepsilon  b(\mathbf{D})(\mathcal{A}_D^0-\zeta I)^{-1}\Vert _{L_2(\mathcal{O})\rightarrow L_2(\mathcal{O}')}\\
&\leqslant
\left( \widetilde{C}_{10}\delta ^{-1} \widehat{\rho}(\zeta)
+\widehat{C}_{11}c(\psi)^{1/2}\rho _*(\zeta)^{5/4}\right)\varepsilon .
\end{split}
\end{align*}
The constants $C_8$, $\widetilde{C}_8$, ${C}_{10}$, and $\widetilde{C}_{10}$ are the same as in Theorem \textnormal{2.9}.
The constants $\check{C}_9$, $\widehat{C}_9$, $\check{C}_{11}$, and $\widehat{C}_{11}$ depend only on the initial data \textnormal{\eqref{data}}
and $\Vert \Lambda \Vert _{L_\infty}$.
\end{theorem}

\section*{\S 3. Approximation of the exponential $e^{-\mathcal{A}_{D,\varepsilon}t}$}
\setcounter{section}{3}
\setcounter{equation}{0}
\setcounter{theorem}{0}

\subsection*{3.1. The properties of the operator exponential.}
We start with the following simple statement about estimates for the operator exponentials
$e^{-\mathcal{A}_{D,\varepsilon}t}$ and $e^{-\mathcal{A}_{D}^0t}$ in various operator norms.

\begin{lemma}
Suppose that the assumptions of Theorem \textnormal{2.2} are satisfied.
Then for $t>0$ and $\varepsilon >0$ we have
\begin{align}
\label{3.1}
&\Vert e^{-\mathcal{A}_{D,\varepsilon}t}\Vert _{L_2(\mathcal{O})\rightarrow L_2(\mathcal{O})}
\leqslant e^{-c_* t},
\\
\label{3.2}
&\Vert e^{-\mathcal{A}_{D,\varepsilon}t}\Vert _{L_2(\mathcal{O})\rightarrow H^1(\mathcal{O})}
\leqslant (c_0^{-1}+ c_*^{-1})^{1/2} t^{-1/2} e^{-c_* t/2},
\\
\label{3.3}
&\Vert e^{-\mathcal{A}_{D}^0 t}\Vert _{L_2(\mathcal{O})\rightarrow L_2(\mathcal{O})}
\leqslant e^{-c_* t},
\\
\label{3.4}
&\Vert e^{-\mathcal{A}_{D}^0 t}\Vert _{L_2(\mathcal{O})\rightarrow H^1(\mathcal{O})}
\leqslant (c_0^{-1}+ c_*^{-1})^{1/2}  t^{-1/2} e^{-c_* t/2},
\\
\label{3.5}
&\Vert e^{-\mathcal{A}_{D}^0 t}\Vert _{L_2(\mathcal{O})\rightarrow H^2(\mathcal{O})}
\leqslant \widehat{c} \,t^{-1} e^{-c_* t/2}.
\end{align}
\end{lemma}

\begin{proof}
Estimates \eqref{3.1} and \eqref{3.3} follow directly from \eqref{a_D,e>=} and \eqref{a_D,0>=}.

Combining (\ref{<a_D,e<}) and (\ref{a_D,e>=}), we conclude that
\begin{equation*}
\Vert \mathbf{u}\Vert ^2_{H^1(\mathcal{O})}\leqslant (c_0^{-1}+c_*^{-1}) \Vert \mathcal{A}_{D,\varepsilon}^{1/2}\mathbf{u}\Vert ^2_{L_2(\mathcal{O})},\quad \mathbf{u}\in H^1_0(\mathcal{O};\mathbb{C}^n).
\end{equation*}
Hence,
\begin{equation}
\label{3.6}
\Vert e^{-\mathcal{A}_{D,\varepsilon}t} \Vert _{L_2(\mathcal{O})\rightarrow H^1(\mathcal{O})}\leqslant
(c_0^{-1} + c_*^{-1})^{1/2}
\Vert \mathcal{A}_{D,\varepsilon}^{1/2} e^{-\mathcal{A}_{D,\varepsilon}t} \Vert _{L_2(\mathcal{O})\rightarrow L_2(\mathcal{O})}.
\end{equation}
Next, by \eqref{a_D,e>=},
\begin{equation}
\label{3.7}
\begin{split}
\Vert &\mathcal{A}_{D,\varepsilon}^{1/2} e^{-\mathcal{A}_{D,\varepsilon}t} \Vert _{L_2(\mathcal{O})\rightarrow L_2(\mathcal{O})}
\leqslant \sup_{\mu \geqslant c_*} \mu^{1/2} e^{-\mu t}\\
&\leqslant e^{-c_{*}t/2}\sup_{\mu \geqslant 0} \mu^{1/2} e^{-\mu t/2}\leqslant t^{-1/2}e^{-c_{*}t/2}.
\end{split}
\end{equation}
Relations (\ref{3.6}) and (\ref{3.7}) imply \eqref{3.2}.
Similarly, combining the inequality
\begin{equation}
\label{3.8}
\Vert (\mathcal{A}_{D}^0)^{1/2} e^{-\mathcal{A}_{D}^0 t} \Vert _{L_2(\mathcal{O})\rightarrow L_2(\mathcal{O})}
\leqslant t^{-1/2}e^{-c_{*}t/2},
\end{equation}
\eqref{<a_D,0<}, and \eqref{a_D,0>=}, we obtain \eqref{3.4}.

By \eqref{A0^-1},
\begin{equation}
\label{3.8a}
\begin{split}
\Vert  e^{-\mathcal{A}_{D}^0 t} \Vert _{L_2(\mathcal{O})\rightarrow H^2(\mathcal{O})}&\leqslant
\widehat{c}
\Vert \mathcal{A}_{D}^0 e^{-\mathcal{A}_{D}^0 t} \Vert _{L_2(\mathcal{O})\rightarrow L_2(\mathcal{O})}
\leqslant \widehat{c} \sup_{\mu \geqslant c_*} \mu e^{-\mu t}
\\
&\leqslant \widehat{c} e^{-c_{*}t/2} \sup_{\mu \geqslant 0} \mu e^{-\mu t/2}\leqslant
\widehat{c}\, t^{-1}e^{-c_{*}t/2}.
\end{split}
\end{equation}
This implies (\ref{3.5}).
\end{proof}

\subsection*{3.2. Approximation of the operator exponential in the  $L_2(\mathcal{O};\mathbb{C}^n)$-operator norm.}
Our goal in this subsection is to prove the following theorem.

\begin{theorem}
Suppose that the assumptions of Theorem \textnormal{2.2} are satisfied.
Then for \hbox{$0<\varepsilon \leqslant \varepsilon_1$} we have
\begin{equation}
\label{Th_exp_L_2}
\Vert e^{-\mathcal{A}_{D,\varepsilon}t}-e^{-\mathcal{A}_D^0t}\Vert _{L_2(\mathcal{O})\rightarrow L_2(\mathcal{O})}\leqslant C_{12}\varepsilon  (t+\varepsilon ^2)^{-1/2}e^{-c_* t/2},\quad t\geqslant 0.
\end{equation}
The constant $C_{12}$ depends only on the initial data \textnormal{\eqref{data}}.
\end{theorem}

\begin{proof}
The proof is based on Theorem 2.2 and the representations for exponentials of the operators $\mathcal{A}_{D,\varepsilon}$ and $\mathcal{A}_D^0$
in terms of the contour integrals of the corresponding resolvents.

We have (see, e.~g., \cite[Chapter~IX, \S 1.6]{K})
\begin{equation}
\label{2.1a}
e^{-\mathcal{A}_{D,\varepsilon}t}=-\frac{1}{2\pi i}\int _\gamma e^{-\zeta t}(\mathcal{A}_{D,\varepsilon}-\zeta I)^{-1}\,d\zeta ,\quad t>0.
\end{equation}
Here $\gamma$ is the positively oriented contour in the complex plane  enclosing the spectrum of $\mathcal{A}_{D,\varepsilon}$.
By (\ref{a_D,e>=}) and (\ref{a_D,0>=}), the spectra of the operators $\mathcal{A}_{D,\varepsilon}$ and $\mathcal{A}_D^0$
lie on the ray $[c_*,\infty)$. Therefore, we take
\begin{equation*}
\begin{split}
\gamma& =\lbrace \zeta \in\mathbb{C} : \textnormal{Im}\,\zeta \geqslant 0,\, \textnormal{Re}\,\zeta =\textnormal{Im}\,\zeta +c_*/2\rbrace \\
&\cup \lbrace \zeta \in \mathbb{C} : \textnormal{Im}\,\zeta \leqslant 0,\, \textnormal{Re}\,\zeta =-\textnormal{Im}\,\zeta +c_*/2\rbrace .
\end{split}
\end{equation*}
The exponential of the effective operator $\mathcal{A}_D^0$ admits a representation similar to (\ref{2.1a})
with the same contour $\gamma$. Hence,
\begin{equation}
\label{raznost exp=int}
e^{-\mathcal{A}_{D,\varepsilon}t}-e^{-\mathcal{A}_D^0t}=-\frac{1}{2\pi i}\int _\gamma e^{-\zeta t}\left((\mathcal{A}_{D,\varepsilon}-\zeta I)^{-1}-(\mathcal{A}_D^0-\zeta I)^{-1}\right)\,d\zeta .
\end{equation}

Using Theorem 2.2, we estimate the difference of the resolvents for $\zeta \in \gamma$ uniformly with respect to $ \textnormal{arg}\,\zeta$.
Recall the notation $\psi =\textnormal{arg}\,(\zeta -c_*)$.
Note that for $\zeta \in \gamma$ and $\psi =\pi/2$ or $\psi =3\pi/2$ we have $\vert \zeta \vert =\sqrt{5} c_*/2$.
We apply assertion~$2^\circ$ of Theorem~2.2 for $\zeta \in \gamma$ such that
$\vert \zeta \vert \leqslant \check{c}:= \max \lbrace 1; \sqrt{5}c_* /2\rbrace $.
Obviously, $\psi \in (\pi/4,7\pi/4)$ and $\rho _*(\zeta)\leqslant 2 \max \lbrace 1; 8c_*^{-2}\rbrace $ on the contour $\gamma$.
Therefore, for $0<\varepsilon \leqslant \varepsilon _1 $ relation \eqref{Th_Su_small_zeta} implies that
\begin{equation}
\label{ots_zeta_levee_c2}
\begin{split}
\Vert &(\mathcal{A}_{D,\varepsilon}-\zeta I)^{-1}-(\mathcal{A}_D^0-\zeta I)^{-1}\Vert _{L_2\rightarrow L_2}\leqslant 2C_2\max \lbrace 1; 8c_*^{-2}\rbrace \varepsilon
\\
&\leqslant C_{12}' \left(\vert \zeta \vert ^{-1/2}\varepsilon +\varepsilon ^2\right),\quad
 \zeta \in \gamma ,\; \vert \zeta \vert \leqslant \check{c},
\end{split}
\end{equation}
where $C_{12}' =  2C_2\max \lbrace 1; 8c_*^{-2}\rbrace \check{c}^{1/2}$.
(The last step is simply coarsening of the estimate.)

For other $\zeta \in \gamma $ we have $\vert \sin \phi \vert \geqslant 5^{-1/2}$, and, by assertion $1^\circ$ of Theorem 2.2,
\begin{equation}
\label{ots_zeta_pravee_c2}
\begin{split}
\Vert (\mathcal{A}_{D,\varepsilon}-\zeta I)^{-1}-(\mathcal{A}_D^0-\zeta I)^{-1}\Vert _{L_2\rightarrow L_2}\leqslant
C_{12}'' (\vert \zeta \vert ^{-1/2}\varepsilon +\varepsilon ^2),\\
\zeta \in \gamma ,\; \vert \zeta \vert > \check{c} ,\; 0<\varepsilon \leqslant \varepsilon _1,
\end{split}
\end{equation}
where $C_{12}'' =  5^{5/2}C_1$.
As a result, combining (\ref{ots_zeta_levee_c2}) and (\ref{ots_zeta_pravee_c2}), for $0<\varepsilon \leqslant \varepsilon _1$ we have
\begin{equation}
\label{Raznost res na gamma}
\Vert (\mathcal{A}_{D,\varepsilon}-\zeta I)^{-1}-(\mathcal{A}_D^0-\zeta I)^{-1}\Vert _{L_2\rightarrow L_2}\leqslant\widehat{C}_{12}(\vert \zeta \vert ^{-1/2}\varepsilon +\varepsilon ^2),\quad \zeta \in \gamma ,
\end{equation}
where $\widehat{C}_{12}=\max\lbrace C_{12}'; C_{12}'' \rbrace$.

From (\ref{raznost exp=int}) and (\ref{Raznost res na gamma}) it follows that
\begin{equation*}
\Vert e^{-\mathcal{A}_{D,\varepsilon} t}-e^{-\mathcal{A}_D^0 t }\Vert _{L_2(\mathcal{O})\rightarrow L_2(\mathcal{O})}\leqslant 2\pi ^{-1}\widehat{C}_{12} \left(\varepsilon t^{-1/2}\Gamma (1/2) +\varepsilon ^2 t^{-1}\right) e^{-c_*t/2}.
\end{equation*}
Note that for $t\geqslant \varepsilon ^2$ we have $\varepsilon^2 t^{-1}\leqslant\varepsilon t^{-1/2}$.
Thus, using the identity $\Gamma (1/2)=\pi ^{1/2}$, we obtain
\begin{equation}
\label{Ots.3 L_2}
\begin{split}
\Vert e^{-\mathcal{A}_{D,\varepsilon}t}-e^{-\mathcal{A}_D^0t}\Vert _{L_2(\mathcal{O})\rightarrow L_2(\mathcal{O})}&\leqslant
4\pi ^{-1/2}\widehat{C}_{12}\varepsilon t^{-1/2}e^{-c_*t/2}
\\
&\leqslant \check{C}_{12}\varepsilon (t+\varepsilon^2)^{-1/2} e^{-c_*t/2}
,\quad t\geqslant \varepsilon ^2,
\end{split}
\end{equation}
where $\check{C}_{12}=4\sqrt{2}\pi ^{-1/2}\widehat{C}_{12}$.
For $t\leqslant \varepsilon ^2$ we apply \eqref{3.1} and \eqref{3.3}:
\begin{equation}
\label{Ots.4 L_2}
\Vert e^{-\mathcal{A}_{D,\varepsilon}t}-e^{-\mathcal{A}_D^0t}\Vert _{L_2(\mathcal{O})\rightarrow L_2(\mathcal{O})}\leqslant
2e^{-c_*t}\leqslant 2\sqrt{2} \varepsilon (t+\varepsilon^2)^{-1/2} e^{-c_*t/2},\quad t\leqslant \varepsilon ^2.
\end{equation}
Relations (\ref{Ots.3 L_2}) and (\ref{Ots.4 L_2}) imply the required inequality
(\ref{Th_exp_L_2}) with $C_{12} = \max \lbrace \check{C}_{12}; 2\sqrt{2} \rbrace $.
\end{proof}

\subsection*{3.3. Approximation of the operator $e^{-\mathcal{A}_{D,\varepsilon}t}$ in the $(L_2 \to H^1)$-norm.}
We introduce the \textit{corrector}
\begin{equation}
\label{K_D(t,e)}
\mathcal{K}_D(t;\varepsilon)=R_\mathcal{O}[\Lambda ^\varepsilon]S_\varepsilon b(\mathbf{D})P_\mathcal{O}e^{-\mathcal{A}_D^0t}.
\end{equation}
By (\ref{3.5}), for $t>0$
the operator $e^{-\mathcal{A}_D^0 t}$ is continuous from
$L_2(\mathcal{O};\mathbb{C}^n)$ to $H^2(\mathcal{O};\mathbb{C}^n)$. Hence, the operator $b({\mathbf D})P_\mathcal{O} e^{-\mathcal{A}_D^0 t}$
is continuous from $L_2(\mathcal{O};\mathbb{C}^n)$ to $H^1(\mathbb{R}^d;\mathbb{C}^m)$. As has been already mentioned,
the operator $[\Lambda ^\varepsilon]S_\varepsilon$ is continuous from $H^1(\mathbb{R}^d;\mathbb{C}^m)$ to $H^1(\mathbb{R}^d;\mathbb{C}^n)$.
Thus, the operator $\mathcal{K}_D(t;\varepsilon)$ is a continuous mapping of $L_2(\mathcal{O};\mathbb{C}^n)$ to $H^1(\mathcal{O};\mathbb{C}^n)$.

\begin{theorem}
Suppose that the assumptions of Theorem \textnormal{2.2} are satisfied. Let $\mathcal{K}_D(t;\varepsilon)$ be the operator \textnormal{(\ref{K_D(t,e)})}. Let $\widetilde{g}$ be the matrix-valued function \textnormal{(\ref{tilde g})}.
Then for \hbox{$0<\varepsilon\leqslant\varepsilon _1$} and $t>0$ we have
\begin{align}
\label{Th_exp_korrector}
\begin{split}
\Vert & e^{-\mathcal{A}_{D,\varepsilon}t}-e^{-\mathcal{A}_D^0t}-\varepsilon \mathcal{K}_D(t;\varepsilon)\Vert _{L_2(\mathcal{O})\rightarrow H^1(\mathcal{O})}\\
&\leqslant C_{13}( \varepsilon ^{1/2}t^{-3/4}+\varepsilon t^{-1})e^{-c_* t/2},
\end{split}\\
\begin{split}
\label{Th_exp_potok}
\Vert & g^\varepsilon b(\mathbf{D})e^{-\mathcal{A}_{D,\varepsilon}t}-\widetilde{g}^\varepsilon S_\varepsilon b(\mathbf{D})P_\mathcal{O}e^{-\mathcal{A}_D^0 t}\Vert _{L_2(\mathcal{O})\rightarrow L_2(\mathcal{O})}\\
&\leqslant \widetilde{C}_{13}( \varepsilon ^{1/2}t^{-3/4}+\varepsilon t^{-1})e^{-c_* t/2} .
\end{split}
\end{align}
The constants $C_{13}$ and $\widetilde{C}_{13}$ depend only on the initial data \textnormal{\eqref{data}}.
\end{theorem}

\begin{proof}
 As in the proof of Theorem 3.2, we shall use representations for the operators $e^{-\mathcal{A}_{D,\varepsilon}t}$ and
  $e^{-\mathcal{A}_D^0 t}$ in terms of the integrals of the corresponding resolvents
  over the contour $\gamma$. We have
\begin{equation}
\label{exp -exp-K=int}
\begin{split}
&e^{-\mathcal{A}_{D,\varepsilon}t}-e^{-\mathcal{A}_D^0t}-\varepsilon \mathcal{K}_D(t;\varepsilon)\\
&=-\frac{1}{2 \pi i}\int _\gamma e^{-\zeta t}\left( (\mathcal{A}_{D,\varepsilon}-\zeta I)^{-1}-(\mathcal{A}_D^0-\zeta I)^{-1}-\varepsilon K_D(\varepsilon;\zeta )\right) \,d\zeta ,
\end{split}
\end{equation}
where $K_D(\varepsilon;\zeta )$ is the operator (\ref{K_D(e)}).

Similarly to  \eqref{ots_zeta_levee_c2}--\eqref{Raznost res na gamma}, from Theorem 2.3 it follows that
\begin{equation}
\label{Th Su 14 zeta in gamma}
\begin{split}
 \Vert &(\mathcal{A}_{D,\varepsilon}-\zeta I)^{-1}-(\mathcal{A}_D^0-\zeta I)^{-1}-\varepsilon K_D(\varepsilon ;\zeta )\Vert _{L_2\rightarrow H^1}\\
 &\leqslant
 \widehat{C}_{13}\left(\varepsilon ^{1/2}\vert \zeta \vert ^{-1/4}+\varepsilon\right) ,\quad
 \zeta \in \gamma ,\quad 0<\varepsilon \leqslant\varepsilon _1,
 \end{split}
\end{equation}
with the constant $\widehat{C}_{13}=\max \lbrace C_{13}'; C_{13}'' \rbrace$, where $C_{13}'= 2C_5
\max \lbrace 1; 8c_*^{-2}\rbrace \check{c}^{1/4}$, $C_{13}''= \max \lbrace 5C_3;25C_4\rbrace$.
Relations (\ref{exp -exp-K=int}) and (\ref{Th Su 14 zeta in gamma}) imply the required estimate (\ref{Th_exp_korrector})
with the constant $C_{13}=2\pi ^{-1}\Gamma (3/4) \widehat{C}_{13}$.

In a similar way, inequality (\ref{Th_exp_potok}) is deduced from the identity
\begin{equation}
\label{tozd potoki D int}
\begin{split}
&g^\varepsilon b(\mathbf{D})e^{-\mathcal{A}_{D,\varepsilon }t}-\widetilde{g}^\varepsilon S_\varepsilon b(\mathbf{D})P_\mathcal{O}e^{-\mathcal{A}_D^0 t}\\
&=-\frac{1}{2\pi i}\int _\gamma e^{-\zeta t} \left( g^\varepsilon b(\mathbf{D})(\mathcal{A}_{D,\varepsilon}-\zeta I)^{-1}-\widetilde{g}^\varepsilon S_\varepsilon b(\mathbf{D})P_\mathcal{O} (\mathcal{A}_D^0-\zeta I)^{-1}\right) \,d\zeta
\end{split}
\end{equation}
and Theorem 2.3.
\end{proof}

\begin{remark}\textnormal{Suppose that $\lambda _1^0$ is the first eigenvalue of the operator $\mathcal{A}_D^0$.
Due to the resolvent convergence, for sufficiently small $\kappa >0$ there exists a number $\varepsilon _*$, depending on
   $\kappa $ and on the problem data, such that for \hbox{$0<\varepsilon\leqslant \varepsilon _*$} the first eigenvalue
   $\lambda _{1,\varepsilon}$ of the operator $\mathcal{A}_{D,\varepsilon}$ satisfies $\lambda _{1,\varepsilon} >\lambda _1^0-\kappa /2$.
   Therefore, one can transform the contour of integration so that
   it will intersect the real axis at the point $\lambda _1^0-\kappa $ (instead of $c_*/2$).
   By this way, one can obtain estimates of the form  (\ref{Th_exp_L_2}), (\ref{Th_exp_korrector}), and (\ref{Th_exp_potok})
with $e^{-c_* t/2}$ replaced by $e^{-(\lambda _1^0 -\kappa )t}$.
The constants in such estimates will depend on $\kappa$.
}
\end{remark}

\subsection*{3.4. Estimates for small $t$.}
Note that for $0< t < \eps ^2$ it makes no sense to apply estimates (\ref{Th_exp_korrector}) and (\ref{Th_exp_potok}),
since it is better to use the following simple statement
(which is valid, however, for all $t>0$).

\begin{proposition}
Suppose that the assumptions of Theorem \textnormal{2.2} are satisfied. Then for $t>0$ and $\varepsilon >0$ we have
\begin{align}
\label{small_t}
&\Vert e^{-\mathcal{A}_{D,\varepsilon}t}- e^{-\mathcal{A}_D^0 t}\Vert _{L_2(\mathcal{O})\rightarrow H^1(\mathcal{O})}
\leqslant C_{14} t^{-1/2} e^{-c_* t/2},
\\
\label{small_t_potok}
&\Vert g^\varepsilon b({\mathbf D}) e^{-\mathcal{A}_{D,\varepsilon}t}  \Vert _{L_2(\mathcal{O})\rightarrow L_2(\mathcal{O})}
\leqslant \widetilde{C}_{14} t^{-1/2} e^{-c_* t/2},
\\
\label{small_t_potok_0}
&\Vert g^0 b({\mathbf D}) e^{-\mathcal{A}_{D}^0 t}  \Vert _{L_2(\mathcal{O})\rightarrow L_2(\mathcal{O})}
\leqslant \widetilde{C}_{14} t^{-1/2} e^{-c_* t/2},
\end{align}
where $C_{14}=2\alpha _0^{-1/2}\Vert g^{-1}\Vert _{L_\infty}^{1/2}\left(1+(\textnormal{diam}\,
\mathcal{O})^2\right)^{1/2},$ $\widetilde{C}_{14}=\Vert g\Vert _{L_\infty}^{1/2}$.
\end{proposition}

\begin{proof}
Estimate \eqref{small_t} follows from \eqref{3.2}, \eqref{3.4}, and expressions for the constants $c_0$ and $c_*$.

Next, we have
$$
\Vert g^\eps b(\D) e^{-\mathcal{A}_{D,\varepsilon}t} \Vert _{L_2(\mathcal{O})\rightarrow L_2(\mathcal{O})}\leqslant \Vert g\Vert _{L_\infty}^{1/2}\Vert \mathcal{A}_{D,\varepsilon}^{1/2} e^{-\mathcal{A}_{D,\varepsilon}t}\|_{L_2(\mathcal{O})\rightarrow L_2(\mathcal{O})}.
$$
Combining this with (\ref{3.7}), we obtain (\ref{small_t_potok}). Taking (\ref{|g^0|, |g^0^_1|<=}) and  \eqref{3.8} into account,
it is easy to check (\ref{small_t_potok_0}) in a similar way.
\end{proof}

\subsection*{3.5. The case where $\Lambda \in L_\infty$.} Now, assume in addition that the matrix-valued function
$\Lambda (\mathbf{x})$ satisfies Condition 2.5. In this case it is possible to remove the smoothing operator  $S_\varepsilon$ in the corrector
  and to use the following corrector
\begin{equation}
\label{K_D^0(t,e)}
\mathcal{K}_D^0(t;\varepsilon)=[\Lambda ^\varepsilon]b(\mathbf{D})e^{-\mathcal{A}_D^0 t} .
\end{equation}
For $t>0$ the operator $b(\mathbf{D})e^{-\mathcal{A}_D^0t}$ is continuous from $L_2(\mathcal{O};\mathbb{C}^n)$ to $H^1(\mathcal{O};\mathbb{C}^m)$. As has been already mentioned in Subsection~2.3,  under Condition~2.5 the operator $[\Lambda ^\varepsilon ]$ is continuous from  $H^1(\mathcal{O};\mathbb{C}^m)$ to $H^1(\mathcal{O};\mathbb{C}^n)$. Hence, the operator $\mathcal{K}_D^0(t;\varepsilon)$
is a continuous mapping of  $L_2(\mathcal{O};\mathbb{C}^n)$ to $H^1(\mathcal{O};\mathbb{C}^n)$.

\begin{theorem}
Suppose that the assumptions of Theorem \textnormal{2.2} are satisfied. Suppose that the matrix-valued function $\Lambda(\mathbf{x})$
satisfies Condition~\textnormal{2.5}. Let $\mathcal{K}_D^0(t;\varepsilon)$ be the operator \textnormal{(\ref{K_D^0(t,e)})}.
Then for $0<\varepsilon \leqslant \varepsilon _1$ and $t>0$ we have
\begin{align}
\label{th3.5}
&\Vert e^{-\mathcal{A}_{D,\varepsilon}t}-e^{-\mathcal{A}_D^0t}-\varepsilon \mathcal{K}_D^0(t;\varepsilon)\Vert _{L_2(\mathcal{O})\rightarrow H^1(\mathcal{O})}\leqslant C_{15}(\varepsilon ^{1/2}t^{-3/4}+\varepsilon t^{-1})
e^{-c_{*}t/2},
\\
\label{th3.5_2}
&\Vert g^\varepsilon b(\mathbf{D})e^{-\mathcal{A}_{D,\varepsilon}t}-\widetilde{g}^\varepsilon b(\mathbf{D})e^{-\mathcal{A}_D^0t}\Vert _{L_2(\mathcal{O})\rightarrow L_2(\mathcal{O})}\leqslant \widetilde{C}_{15}(\varepsilon ^{1/2}t^{-3/4}+\varepsilon t^{-1})
e^{-c_{*}t/2},
\end{align}
where $\widetilde{g}$ is the matrix-valued function \textnormal{(\ref{tilde g})}.
The constants $C_{15}$ and $\widetilde{C}_{15}$ depend only on the initial data
\textnormal{\eqref{data}} and $\Vert \Lambda \Vert _{L_\infty}$.
\end{theorem}

The proof of Theorem 3.6 is similar to that of Theorem~3.3.
The difference is that instead of Theorem~2.3 one should apply Theorem~2.6.

\subsection*{3.6. The case of smooth boundary.}
It is also possible to remove the smoothing operator $S_\eps$ in the corrector
under the assumption of additional smoothness of the boundary.
In this subsection we consider the case where $d \geqslant 3$, because
for $d \leqslant 2$ Theorem 3.6 is applicable (see Proposition~2.7($1^\circ$)).

\begin{lemma}
Suppose that $k \geqslant 2$ is integer. Let ${\mathcal O} \subset {\mathbb R}^d$ be a bounded domain with
the boundary $\partial {\mathcal O}$ of class $C^{k-1,1}$.
Then for $t>0$ the operator $e^{-\mathcal{A}^0_{D}t}$ is continuous from $L_2({\mathcal O};{\mathbb C}^n)$
to $H^s({\mathcal O};{\mathbb C}^n)$, $0\leqslant s \leqslant k$, and the following estimate is valid:
\begin{equation}
\label{*.1}
\|e^{-\mathcal{A}^0_{D}t}\|_{L_2({\mathcal O}) \to H^s({\mathcal O})}
\leqslant \widehat{\mathrm C}_s  t^{-s/2} e^{-c_{*}t/2} ,\quad t>0.
\end{equation}
The constant $\widehat{\mathrm C}_s$ depends only on
$s$,  $\alpha _0$, $\alpha _1$, $\Vert g\Vert _{L_\infty}$, $\Vert g^{-1}\Vert _{L_\infty}$, and the domain $\mathcal{O}$.
\end{lemma}

\begin{proof}
It suffices to check estimate \eqref{*.1} for integer $s\in [0,k]$; then the result for non-integer $s$ follows by interpolation.
For $s=0,1,2$ estimate \eqref{*.1} has been already proved (see Lemma~3.1).

So, let $s$ be integer and let $2 \leqslant s \leqslant k$.
We apply theorems about regularity of solutions of strongly elliptic equations (see, e.~g., [McL, Chapter~4]). By these theorems,
the operator $({\mathcal A}_D^0)^{-1}$ is continuous from
$H^\sigma({\mathcal O};{\mathbb C}^n)$ to $H^{\sigma+2}({\mathcal O};{\mathbb C}^n)$ if $\partial {\mathcal O} \in C^{\sigma+1,1}$.
Here $\sigma\geqslant 0$ is integer. Note also that the operator
$({\mathcal A}_D^0)^{-1/2}$ is continuous from $L_2({\mathcal O};{\mathbb C}^n)$ to
$H^{1}({\mathcal O};{\mathbb C}^n)$.
It follows that, under the assumptions of Lemma, for integer $s \in [2,k]$
the operator $({\mathcal A}_D^0)^{-s/2}$ is continuous from $L_2({\mathcal O};{\mathbb C}^n)$ to
$H^{s}({\mathcal O};{\mathbb C}^n)$. Herewith,
\begin{equation}
\label{*.2}
\| (\mathcal{A}^0_{D})^{-s/2}\|_{L_2({\mathcal O}) \to H^s({\mathcal O})} \leqslant \widetilde{\mathrm C}_s,
\end{equation}
where the constant $\widetilde{\mathrm C}_s$ depends only on
$s$,  $\alpha _0$, $\alpha _1$, $\Vert g\Vert _{L_\infty}$, $\Vert g^{-1}\Vert _{L_\infty}$, and the domain $\mathcal{O}$.
From \eqref{*.2} it follows that
\begin{equation}
\label{3.32a}
\begin{split}
\|e^{-\mathcal{A}^0_{D}t}\|_{L_2({\mathcal O}) \to H^s({\mathcal O})}
&\leqslant \widetilde{\mathrm C}_s
\| (\mathcal{A}^0_{D})^{s/2}e^{-\mathcal{A}^0_{D}t}\|_{L_2({\mathcal O}) \to L_2({\mathcal O})}
\leqslant
\widetilde{\mathrm C}_s \sup_{x \geqslant c_*} x^{s/2} e^{-xt}
\\
&\leqslant
\widetilde{\mathrm C}_s  t^{-s/2} e^{-c_{*}t/2} \sup_{\alpha \geqslant 0} \alpha^{s/2} e^{-\alpha/2}
\leqslant
\widehat{\mathrm C}_s  t^{-s/2} e^{-c_{*}t/2},
\end{split}
\end{equation}
where $\widehat{\mathrm C}_s=\widetilde{\mathrm C}_s  (s/e)^{s/2}$.
\end{proof}

Using the properties of the matrix-valued function $\Lambda({\mathbf x})$ and the operator $S_\eps$,
it is possible to deduce the estimate for the difference of the correctors \eqref{K_D(t,e)} and \eqref{K_D^0(t,e)}
from Lemma~3.7 under the assumption of additional smoothness of the boundary.

\begin{lemma}
Let $d \geqslant 3$.
Let ${\mathcal O} \subset {\mathbb R}^d$ be a bounded domain with the boundary
$\partial {\mathcal O}$ of class $C^{d/2,1}$ if $d$ is even and of class $C^{(d+1)/2,1}$ if $d$ is odd.
Let $\mathcal{K}_D(t;\varepsilon)$ be the operator \eqref{K_D(t,e)}, and let
$\mathcal{K}_D^0(t;\varepsilon)$ be the operator \eqref{K_D^0(t,e)}. Then
\begin{equation}
\label{*.3}
\| \mathcal{K}_D(t;\varepsilon) - \mathcal{K}_D^0(t;\varepsilon)
\|_{L_2({\mathcal O}) \to H^1({\mathcal O})}
\leqslant \check{\mathrm C}_d   (t^{-1} +  t^{-d/4 -1/2}) e^{-c_{*}t/2},\quad t>0.
\end{equation}
The constant $\check{\mathrm C}_d$ depends only on $d$,  $m$,
$\alpha _0$, $\alpha _1$, $\Vert g\Vert _{L_\infty}$, $\Vert g^{-1}\Vert _{L_\infty}$, and the domain $\mathcal{O}$.
\end{lemma}

Lemma~3.8 and Theorem~3.3 imply the following result.

\begin{theorem}
Suppose that the assumptions of Theorem~\textnormal{2.2} are satisfied, and let $d \geqslant 3$.
Suppose that the domain ${\mathcal O}$ satisfies the assumptions of Lemma~\textnormal{3.8}.
Let $\mathcal{K}_D^0(t;\varepsilon)$ be the operator~\textnormal{(\ref{K_D^0(t,e)})}.
Let $\widetilde{g}$ be the matrix-valued function~\textnormal{(\ref{tilde g})}.
Then for $0<\varepsilon \leqslant \varepsilon _1$ and $t>0$ we have
\begin{align}
\label{*.4}
\begin{split}
&\Vert e^{-\mathcal{A}_{D,\varepsilon}t}-e^{-\mathcal{A}_D^0t}-\varepsilon \mathcal{K}_D^0(t;\varepsilon)\Vert _{L_2(\mathcal{O})\rightarrow H^1(\mathcal{O})}
\\
&\leqslant
{\mathrm C}^*_{d}(\varepsilon ^{1/2}t^{-3/4}+\varepsilon t^{-1} + \varepsilon t^{-d/4 -1/2})
e^
{-c_{*}t/2},
\end{split}
\\
\label{*.5}
\begin{split}
&\Vert g^\varepsilon b(\mathbf{D})e^{-\mathcal{A}_{D,\varepsilon}t}-\widetilde{g}^\varepsilon b(\mathbf{D})e^{-\mathcal{A}_D^0t}\Vert _{L_2(\mathcal{O})\rightarrow L_2(\mathcal{O})}
\\
&\leqslant
\widetilde{\mathrm C}^*_{d}(\varepsilon ^{1/2}t^{-3/4}+\varepsilon t^{-1} + \varepsilon t^{-d/4 -1/2})
e^{-c_{*}t/2}.
\end{split}
\end{align}
The constants ${\mathrm C}^*_{d}$ and $\widetilde{\mathrm C}^*_{d}$ depend only on the initial data
\textnormal{\eqref{data}}.
\end{theorem}

The proofs of Lemma~3.8 and Theorem~3.9 are postponed until Appendix (see \S 9) in order
not to overload the main exposition.
Clearly, it is convenient to apply Theorem 3.9 if $t$ is separated from zero.
For small values of $t$ the order of the factor $(\varepsilon ^{1/2}t^{-3/4}+\varepsilon t^{-1} + \varepsilon t^{-d/4 -1/2})$
grows with dimension. This is ,,charge'' for removing the smoothing operator.

\begin{remark}
\textnormal{In Lemma~3.8, instead of the smoothness condition on $\partial \mathcal{O}$, one could
impose an implicit condition on the domain:
suppose that a bounded domain $\mathcal{O}$ with Lipschitz boundary is such that
estimate \eqref{*.2} with $s=d/2+1$ is valid.
In such domain all the statements of Lemma~3.8 and Theorem~3.9 are true.}
\end{remark}

\subsection*{3.7. Special cases.}
The case where the corrector is equal to zero is distinguished by the condition $g^0=\overline{g}$.
By Proposition~1.2, in this case relations (\ref{overline-g}) are satisfied, and then the periodic solution $\Lambda$ of problem (\ref{Lambda_problem}) is equal to zero.
Applying Theorem~3.3, we obtain the following statement.

\begin{proposition}
Suppose that the assumptions of Theorem \textnormal{3.3} are satisfied.
Let $g^0=\overline{g}$, i.~e. relations \textnormal{(\ref{overline-g})} are satisfied.
Then $\mathcal{K}_D(t;\varepsilon)=0$, and for $0<\varepsilon\leqslant \varepsilon _1$ we have
$$
\Vert  e^{-\mathcal{A}_{D,\varepsilon}t}-e^{-\mathcal{A}_D^0t}
\Vert _{L_2(\mathcal{O})\rightarrow H^1(\mathcal{O})}\leqslant C_{13} (\varepsilon ^{1/2}t^{-3/4}+\varepsilon t^{-1}) e^{-c_{*}t/2},
\ \  t>0.
$$
\end{proposition}

If $g^0 = \underline{g}$, by Proposition~1.3, relations (\ref{underline-g}) are satisfied.
Then the matrix (\ref{tilde g}) is constant: $\widetilde{g}(\x)=g^0 = \underline{g}$.
Besides, by Proposition~2.7($3^\circ$), in this case Condition~2.5 is satisfied.
Applying Theorem~3.6, we arrive at the following statement.

\begin{proposition}
Suppose that the assumptions of Theorem~\textnormal{2.2} are satisfied.
Let $g^0=\underline{g}$, i.~e., relations~\textnormal{(\ref{underline-g})} are satisfied.
Then for $0 < \eps \leqslant \eps_1$ and $t>0$ we have
$$
\Vert g^\eps b(\D) e^{-\mathcal{A}_{D,\varepsilon}t}- g^0 b(\D) e^{-\mathcal{A}_D^0t}
\Vert _{L_2(\mathcal{O})\rightarrow L_2(\mathcal{O})}\leqslant \widetilde{C}_{15} (\varepsilon ^{1/2}t^{-3/4}+\varepsilon t^{-1}) e^{-c_{*}t/2}.
$$
\end{proposition}

\subsection*{3.8. Estimates in a strictly interior subdomain.}
Applying Theorem~2.9, we improve error estimates in a strictly interior subdomain.

\begin{theorem}
Suppose that the assumptions of Theorem~\textnormal{3.3} are satisfied. Let $\mathcal{O}'$
be a strictly interior subdomain of the domain $\mathcal{O}$, and let $\delta =\textnormal{dist}\,\lbrace\mathcal{O}';\partial \mathcal{O}\rbrace$. Then for $0<\varepsilon\leqslant\varepsilon _1$ and $t>0$
we have
\begin{equation}
\label{th3.8}
\Vert  e^{-\mathcal{A}_{D,\varepsilon}t}-e^{-\mathcal{A}_D^0t}-\varepsilon \mathcal{K}_D(t;\varepsilon)\Vert _{L_2(\mathcal{O})\rightarrow H^1(\mathcal{O}')}\leqslant (C_{16}\delta ^{-1}+ C_{17} )\varepsilon t^{-1}e^{-c_{*}t/2},
\end{equation}
\begin{equation*}
\begin{split}
\Vert& g^\varepsilon b(\mathbf{D})e^{-\mathcal{A}_{D,\varepsilon}t}-\widetilde{g}^\varepsilon S_\varepsilon b(\mathbf{D})P_\mathcal{O}e^{-\mathcal{A}_D^0 t}\Vert _{L_2(\mathcal{O})\rightarrow L_2(\mathcal{O}')}
\\
&\leqslant (\widetilde{C}_{16} \delta ^{-1}+\widetilde{C}_{17})\varepsilon t^{-1}e^{-c_{*}t/2} .
\end{split}
\end{equation*}
The constants $C_{16}$, ${C}_{17}$, $\widetilde{C}_{16}$, and $\widetilde{C}_{17}$
depend only on the initial data \textnormal{\eqref{data}}.
\end{theorem}

\begin{proof}
The proof is based on application of Theorem 2.9 and identities \eqref{exp -exp-K=int},  \eqref{tozd potoki D int}.
Besides,  we use estimates $c(\psi)\leqslant 2^{1/2}$ and $\rho _*(\zeta)\leqslant 2\max \lbrace 1;8c_*^{-2}\rbrace$
for $\zeta\in \gamma$, and estimate $c(\phi)\leqslant 5^{1/2}$ for $\zeta \in \gamma$ such that $\vert \zeta \vert > \check{c} $.
These estimates have been checked in the proof of Theorem~3.2. We omit the details.
\end{proof}

The following result is checked similarly with the help of Theorem 2.10.

\begin{theorem}
Suppose that the assumptions of Theorem~\textnormal{3.13} are satisfied.
Suppose that the matrix-valued function $\Lambda (\mathbf{x})$ satisfies Condition~\textnormal{2.5}.
Let $\mathcal{K}_D^0(t;\varepsilon)$ be the operator~\eqref{K_D^0(t,e)}. Then for $0<\varepsilon\leqslant\varepsilon _1$ and $t>0$
we have
\begin{align*}
\Vert  &e^{-\mathcal{A}_{D,\varepsilon}t}-e^{-\mathcal{A}_D^0t}-\varepsilon \mathcal{K}_D^0(t;\varepsilon)\Vert _{L_2(\mathcal{O})\rightarrow H^1(\mathcal{O}')}\leqslant (C_{16}\delta ^{-1}+ \check{C}_{17} )\varepsilon t^{-1}e^{-c_{*}t/2},\\
\Vert &g^\varepsilon b(\mathbf{D})e^{-\mathcal{A}_{D,\varepsilon}t}-\widetilde{g}^\varepsilon  b(\mathbf{D})e^{-\mathcal{A}_D^0 t}\Vert _{L_2(\mathcal{O})\rightarrow L_2(\mathcal{O}')}
\leqslant (\widetilde{C}_{16} \delta ^{-1}+ \widehat{C}_{17})\varepsilon t^{-1}e^{-c_{*}t/2} .
\end{align*}
The constants $C_{16}$ and $\widetilde{C}_{16}$ are the same as in Theorem  \textnormal{3.13}.
The constants $\check{C}_{17}$ and $\widehat{C}_{17}$ depend only on the initial data \textnormal{\eqref{data}}
and $\Vert \Lambda \Vert _{L_\infty}$.
\end{theorem}

Note that it is possible to remove the smoothing operator~$S_\eps$ from the corrector in estimates of Theorem~3.13
without Condition~2.5. Now we consider the case where $d \geqslant 3$ (because in the opposite case, by
Proposition~2.7($1^\circ$), Theorem~3.14 is applicable).
We know that for $t>0$ the operator $e^{-\mathcal{A}^0_{D}t}$ is continuous from $L_2({\mathcal O};{\mathbb C}^n)$
to $H^2({\mathcal O};{\mathbb C}^n)$, and estimate~\eqref{3.5} holds.
Besides, the following property of ,,interior regularity'' is true:
for $t>0$ the operator $e^{-\mathcal{A}^0_{D}t}$ is continuous from $L_2({\mathcal O};{\mathbb C}^n)$
to $H^s({\mathcal O}';{\mathbb C}^n)$ for any integer $s \geqslant 3$. We have
\begin{equation}
\label{apriori}
\|e^{-\mathcal{A}^0_{D}t}\|_{L_2({\mathcal O}) \to H^s({\mathcal O}')}
\leqslant {\mathrm C}'_s  t^{-1/2}(\delta^{-2} + t^{-1})^{(s-1)/2} e^{-c_{*}t/2} ,\quad t>0.
\end{equation}
The constant ${\mathrm C}'_s$ depends on $s$, $d$, $\alpha _0$, $\alpha _1$, $\Vert g\Vert _{L_\infty}$, $\Vert g^{-1}\Vert _{L_\infty}$, and the
domain $\mathcal{O}$. For scalar parabolic equations the property of ,,interior regularity''
was obtained in \cite{LSoUr} (see Chapter~3, \S\,12). In a similar way it can be checked for the operator
$\mathcal{A}^0_{D}$; it is easy to deduce the qualified estimates \eqref{apriori} using that
the derivatives ${\mathbf D}^\alpha {\mathbf u}_0$ (where ${\mathbf u}_0 = e^{-\mathcal{A}^0_{D}t} {\boldsymbol \varphi}$ with $\boldsymbol{\varphi}\in L_2(\mathcal{O};\mathbb{C}^n)$) satisfy the parabolic equation
$\partial {\mathbf D}^\alpha {\mathbf u}_0/ \partial t = - b({\mathbf D})^* g^0 b({\mathbf D}) {\mathbf D}^\alpha {\mathbf u}_0$.
One should multiply this equation by $\chi^2 {\mathbf D}^\alpha {\mathbf u}_0$ and integrate over the cylinder
${\mathcal O}\times (0,t)$. Here $\chi$ is a smooth cut-off function which equals zero near the  walls and the bottom
of the cylinder. A standard analysis of the corresponding integral identity together with
the already known estimates of Lemma~3.1 yield estimates~\eqref{apriori}.

The following statement is deduced from \eqref{apriori} by using
the properties of the matrix-valued function $\Lambda({\mathbf x})$ and the operator $S_\eps$.

\begin{lemma}
Suppose that the assumptions of Theorem \textnormal{3.13} are satisfied, and let $d\geqslant 3$.
Let $\mathcal{K}_D^0(t;\varepsilon)$ be the operator~\eqref{K_D^0(t,e)}.
Let $r_1$ be defined by \textnormal{(\ref{r})}. Then for $0<\varepsilon \leqslant (4r_1)^{-1}\delta$ and
$t>0$ we have
\begin{equation}
\label{K_minus_K}
\| \mathcal{K}_D(t;\varepsilon) - \mathcal{K}_D^0(t;\varepsilon)
\|_{L_2({\mathcal O}) \to H^1({\mathcal O}')}
\leqslant {\mathrm C}''_d   (t^{-1} +  t^{-1/2}(\delta^{-2} + t^{-1})^{d/4})e^{-c_{*}t/2} .
\end{equation}
The constant ${{\mathrm C}}''_d$ depends only on $d$, $\alpha _0$, $\alpha _1$, $\Vert g\Vert _{L_\infty}$,
$\Vert g^{-1}\Vert _{L_\infty}$, and the domain $\mathcal{O}$.
\end{lemma}

The following result is deduced from Lemma 3.15 and Theorem 3.13.

\begin{theorem}
Suppose that the assumptions of Theorem \textnormal{3.13} are satisfied, and let $d \geqslant 3$.
Let $\mathcal{K}_D^0(t;\varepsilon)$ be the operator~\eqref{K_D^0(t,e)}.
Then for $0<\varepsilon\leqslant \min\{\varepsilon _1;(4r_1)^{-1}\delta\}$ and $t>0$ we have
\begin{align}
\label{3.26}
&\Vert  e^{-\mathcal{A}_{D,\varepsilon}t}-e^{-\mathcal{A}_D^0t}-\varepsilon \mathcal{K}_D^0(t;\varepsilon)\Vert _{L_2(\mathcal{O})\rightarrow H^1(\mathcal{O}')}\leqslant
{\mathrm C}_d \eps h_d(\delta;t) e^{-c_{*}t/2},
\\
\label{3.27}
&\Vert g^\varepsilon b(\mathbf{D})e^{-\mathcal{A}_{D,\varepsilon}t}-\widetilde{g}^\varepsilon
b(\mathbf{D}) e^{-\mathcal{A}_D^0 t}\Vert _{L_2(\mathcal{O})\rightarrow L_2(\mathcal{O}')}
\leqslant
\widetilde{{\mathrm C}}_d \eps h_d(\delta;t) e^{-c_{*}t/2}.
\end{align}
Here $h_d(\delta;t)$ is defined by
\begin{equation}
\label{h_d(delta;t)}
h_d(\delta;t) = t^{-1}(\delta^{-1}+1) +  t^{-1/2}(\delta^{-2} + t^{-1})^{d/4}.
\end{equation}
The constants ${{\mathrm C}}_d$ and $\widetilde{{\mathrm C}}_d$
depend only on the initial data \textnormal{\eqref{data}}.
\end{theorem}

Proofs of Lemma 3.15 and Theorem 3.16 are postponed until Appendix (see \S 9) in order
not to overload the main exposition.
Clearly, it is convenient to apply Theorem 3.16 if $t$ is separated from zero.
For small values of $t$ the order of the factor $h_d(\delta;t)$
grows with dimension. This is ,,charge'' for removing the smoothing operator.

\section*{\S 4. Homogenization of solutions of the first initial boundary value problem for parabolic systems}
\setcounter{section}{4}
\setcounter{equation}{0}
\setcounter{theorem}{0}

Now we apply the results of \S 3 to homogenization of solutions of the first initial boundary value problem for parabolic systems.

\subsection*{4.1. The problem for the homogeneous parabolic equation.}
We start with the homogeneous equation.
We are interested in the behaviour for small $\varepsilon$ of the generalized solution $\mathbf{u}_\varepsilon$ of the problem
\begin{equation}
\label{problem}
\begin{cases}
\frac{\partial \mathbf{u}_\varepsilon(\mathbf{x},t)}{\partial t }=-b(\mathbf{D})^*g^\varepsilon (\mathbf{x})b(\mathbf{D})
\mathbf{u}_\varepsilon (\mathbf{x},t),\quad \mathbf{x}\in \mathcal{O},\quad t>0,\\
\mathbf{u}_\varepsilon (\mathbf{x},0)=\boldsymbol{\varphi}(\mathbf{x}),\quad \mathbf{x}\in \mathcal{O},\\
\mathbf{u}_\varepsilon (\cdot ,t)\vert _{\partial \mathcal{O}}=0,\quad t>0,
\end{cases}
\end{equation}
where $\boldsymbol{\varphi}\in L_2(\mathcal{O};\mathbb{C}^n)$. Then $\mathbf{u}_\varepsilon (\cdot ,t)=e^{-\mathcal{A}_{D,\varepsilon}t}\boldsymbol{\varphi}$.

Let $\mathbf{u}_0(\mathbf{x} ,t)$ be the solution of the so called  \glqq homogenized\grqq \, problem
\begin{equation}
\label{eff.problem}
\begin{cases}\frac{\partial \mathbf{u}_0(\mathbf{x},t)}{\partial t}=-b(\mathbf{D})^*g^0b(\mathbf{D})
\mathbf{u}_0(\mathbf{x} ,t),\quad \mathbf{x}\in \mathcal{O},\quad t>0,\\
\mathbf{u}_0(\mathbf{x},0)=\boldsymbol{\varphi}(\mathbf{x}),\quad \mathbf{x}\in \mathcal{O},\\
\mathbf{u}_0(\cdot ,t)\vert _{\partial \mathcal{O}}=0,\quad t>0.
\end{cases}
\end{equation}
Then $\mathbf{u}_0(\cdot ,t)=e^{-\mathcal{A}_D^0t}\boldsymbol{\varphi}$, and Theorems~3.2~and~3.3 directly imply
the following result.

\begin{theorem}
Suppose that the assumptions of Theorem \textnormal{2.2} are satisfied. Let $\mathbf{u}_\varepsilon$ be the solution of
problem~\textnormal{(\ref{problem})}, and let $\mathbf{u}_0$ be the solution of problem~\textnormal{(\ref{eff.problem})},
where $\boldsymbol{\varphi}\in L_2(\mathcal{O};\mathbb{C}^n)$. Then for $0<\varepsilon \leqslant \varepsilon _1$ and $t\geqslant 0$
we have
\begin{equation*}
\Vert \mathbf{u}_\varepsilon (\cdot ,t)-\mathbf{u}_0(\cdot ,t)\Vert _{L_2(\mathcal{O})}\leqslant C_{12}
\varepsilon (t+\varepsilon ^2)^{-1/2}e^{-c_{*} t/2}\Vert \boldsymbol{\varphi}\Vert _{L_2(\mathcal{O})}.
\end{equation*}
Suppose that the matrix-valued function $\Lambda (\mathbf{x})$ is the $\Gamma$-periodic solution of problem~\textnormal{(\ref{Lambda_problem})}. Let 
$\widetilde{\mathbf u}_0(\cdot,t)=P_\mathcal{O} {\mathbf u}_0(\cdot,t)$, where
$P_\mathcal{O}$ is the extension operator~\textnormal{(\ref{P_O H^1, H^2})}. Let $S_\varepsilon$ be the smoothing operator~\textnormal{(\ref{S_e})}.
Then for $0<\varepsilon \leqslant \varepsilon _1$ and $t>0$ we have
\begin{equation*}
\begin{split}
\Vert &\mathbf{u}_\varepsilon (\cdot ,t)- \mathbf{u}_0(\cdot ,t) - \varepsilon \Lambda ^\varepsilon S_\varepsilon b(\mathbf{D}) \widetilde{\mathbf{u}}_0(\cdot ,t)\Vert _{H^1(\mathcal{O})}\\
&\leqslant C_{13}(\varepsilon ^{1/2}t^{-3/4}+\varepsilon t^{-1}) e^{-c_{*} t/2}\Vert \boldsymbol{\varphi}\Vert _{L_2(\mathcal{O})}.
\end{split}
\end{equation*}
Let $\widetilde{g}$ be the matrix-valued function \textnormal{(\ref{tilde g})}. Then the flux $\mathbf{p}_\varepsilon:=g^\varepsilon b(\mathbf{D})\mathbf{u}_\varepsilon$ admits the following approximation for $0<\varepsilon \leqslant \varepsilon _1$ and $t>0$:
\begin{equation*}
\begin{split}
\Vert \mathbf{p}_\varepsilon (\cdot ,t)-\widetilde{g}^\varepsilon S_\varepsilon b(\mathbf{D}) \widetilde{\mathbf{u}}_0(\cdot ,t)\Vert _{L_2(\mathcal{O})}
\leqslant \widetilde{C}_{13}
(\varepsilon^{1/2}t^{-3/4}+\varepsilon t^{-1}) e^{-c_{*} t/2}\Vert \boldsymbol{\varphi}\Vert _{L_2(\mathcal{O})}.
\end{split}
\end{equation*}
\end{theorem}

Under the assumption that $\Lambda \in L_\infty$ Theorem 3.6 implies the following result.

\begin{theorem}
Suppose that the assumptions of Theorem~\textnormal{4.1} are satisfied.
Suppose that the matrix-valued function $\Lambda (\mathbf{x})$ satisfies Condition~\textnormal{2.5}.
Then for $0<\varepsilon \leqslant \varepsilon _1$ and $t>0$ we have
\begin{align*}
\Vert &\mathbf{u}_\varepsilon (\cdot ,t)-\mathbf{u}_0(\cdot ,t)-\varepsilon \Lambda ^\varepsilon b(\mathbf{D})\mathbf{u}_0(\cdot ,t)\Vert _{H^1(\mathcal{O})}\\
&\leqslant C_{15}(\varepsilon ^{1/2}t^{-3/4}+\varepsilon t^{-1}) e^{-c_{*} t/2}\Vert \boldsymbol{\varphi}\Vert _{L_2(\mathcal{O})},\\
\Vert & \mathbf{p}_\varepsilon (\cdot ,t)-\widetilde{g}^\varepsilon b(\mathbf{D})\mathbf{u}_0(\cdot ,t)\Vert _{L_2(\mathcal{O})}
\leqslant \widetilde{C}_{15} (\varepsilon ^{1/2}t^{-3/4}+\varepsilon t^{-1}) e^{-c_{*} t/2}\Vert \boldsymbol{\varphi}\Vert _{L_2(\mathcal{O})}.
\end{align*}
\end{theorem}

In the case of sufficiently smooth boundary from Theorem 3.9 we obtain the following result.

\begin{theorem}
Suppose that the assumptions of Theorem~\textnormal{4.1} are satisfied, and let
$d\geqslant 3$. Suppose that $\partial \mathcal{O}\in C^{d/2,1}$ if $d$ is even and $\partial\mathcal{O}\in C^{(d+1)/2,1}$ if
$d$ is odd.  Then for $0<\varepsilon \leqslant \varepsilon _1$ and $t>0$ we have
\begin{align*}
\begin{split}
\Vert &\mathbf{u}_\varepsilon (\cdot ,t)-\mathbf{u}_0(\cdot ,t)-\varepsilon \Lambda ^\varepsilon b(\mathbf{D})\mathbf{u}_0(\cdot ,t)\Vert _{H^1(\mathcal{O})}\\
&\leqslant \mathrm{C}_d^*(\varepsilon ^{1/2}t^{-3/4}+\varepsilon t^{-1}+\varepsilon t^{-d/4-1/2})e^{-c_*t/2}\Vert \boldsymbol{\varphi}\Vert _{L_2(\mathcal{O})},
\end{split}
\\
\begin{split}
\Vert &\mathbf{p}_\varepsilon (\cdot ,t)-\widetilde{g}^\varepsilon b(\mathbf{D})\mathbf{u}_0(\cdot ,t)\Vert _{L_2(\mathcal{O})}\\
&\leqslant \widetilde{\mathrm{C}}_d^*(\varepsilon ^{1/2}t^{-3/4}+\varepsilon t^{-1}+\varepsilon t^{-d/4-1/2})e^{-c_*t/2}\Vert \boldsymbol{\varphi}\Vert _{L_2(\mathcal{O})}.
\end{split}
\end{align*}
\end{theorem}

Special cases are distinguished with the help of Propositions 3.11 and 3.12.

\begin{proposition}
Suppose that the assumptions of Theorem \textnormal{4.1} are satisfied.

\noindent
$1^\circ$. If $g^0=\overline{g}$, i.~e., relations~\textnormal{(\ref{overline-g})} are satisfied, then for $0< \eps \leqslant \eps_1$ and $t>0$
we have
$$
\Vert \mathbf{u}_\varepsilon (\cdot ,t)-\mathbf{u}_0(\cdot ,t)\Vert _{H^1(\mathcal{O})}\leqslant C_{13}(\varepsilon^{1/2} t^{-3/4}+\varepsilon t^{-1}) e^{-c_{*} t/2}
\Vert \boldsymbol{\varphi}\Vert _{L_2(\mathcal{O})}.
$$

\noindent
$2^\circ$. If $g^0=\underline{g}$, i.~e., relations \textnormal{(\ref{underline-g})} are satisfied,
then for \hbox{$0< \eps \leqslant \eps_1$} and $t>0$ we have
$$
\Vert \mathbf{p}_\varepsilon (\cdot ,t)-\mathbf{p}_0(\cdot ,t)\Vert _{L_2(\mathcal{O})}\leqslant\widetilde{C}_{15}(\varepsilon^{1/2} t^{-3/4}+\varepsilon t^{-1})e^{-c_{*} t/2}
\Vert \boldsymbol{\varphi}\Vert _{L_2(\mathcal{O})},
$$
where $\mathbf{p}_0 = g^0 b(\D) \mathbf{u}_0$.
\end{proposition}

Approximation of the solutions in a strictly interior subdomain~$\mathcal{O}'$ of the domain~$\mathcal{O}$
is deduced from Theorem 3.13.

\begin{theorem}
Suppose that the assumptions of Theorem \textnormal{4.1} are satisfied.
Let $\mathcal{O}'$ be a strictly interior subdomain of the domain $\mathcal{O}$.
Denote $\delta :=\textnormal{dist}\,\lbrace \mathcal{O}';\partial \mathcal{O}\rbrace$.
Then for  $0<\varepsilon\leqslant\varepsilon _1$ and $t>0$ we have
\begin{align*}
\Vert &\mathbf{u}_\varepsilon (\cdot ,t)-\mathbf{u}_0(\cdot ,t)-\varepsilon\Lambda ^\varepsilon S_\varepsilon b(\mathbf{D})\widetilde{\mathbf{u}}_0(\cdot ,t)\Vert _{H^1(\mathcal{O}')}\\
&\leqslant (C_{16}\delta ^{-1}+C_{17})\varepsilon t^{-1}e^{-c_{*}t/2}\Vert \boldsymbol{\varphi}\Vert _{L_2(\mathcal{O})},\\
\Vert &\mathbf{p}_\varepsilon (\cdot ,t)-\widetilde{g}^\varepsilon S_\varepsilon b(\mathbf{D})\widetilde{\mathbf{u}}_0(\cdot ,t)\Vert _{L_2(\mathcal{O}')}\leqslant (\widetilde{C}_{16}\delta ^{-1}+ \widetilde{C}_{17})\varepsilon t^{-1}e^{-c_{*}t/2}\Vert \boldsymbol{\varphi}\Vert _{L_2(\mathcal{O})}.
\end{align*}
\end{theorem}

In the case where $\Lambda\in L_\infty$ Theorem 3.14 implies the following result.

\begin{theorem}
Suppose that the assumptions of Theorem \textnormal{4.5} are satisfied.
Suppose that the matrix-valued function $\Lambda (\mathbf{x})$ satisfies Condition~\textnormal{2.5}.
Then for \hbox{$0<\varepsilon\leqslant\varepsilon _1$} and $t>0$ we have
\begin{align*}
\Vert &\mathbf{u}_\varepsilon (\cdot ,t)-\mathbf{u}_0(\cdot ,t)-\varepsilon\Lambda ^\varepsilon  b(\mathbf{D})\mathbf{u}_0(\cdot ,t)\Vert _{H^1(\mathcal{O}')}\\
&\leqslant (C_{16}\delta ^{-1}+ \check{C}_{17})\varepsilon t^{-1}e^{-c_{*}t/2}\Vert \boldsymbol{\varphi}\Vert _{L_2(\mathcal{O})},\\
\Vert &\mathbf{p}_\varepsilon (\cdot ,t)-\widetilde{g}^\varepsilon  b(\mathbf{D})\mathbf{u}_0(\cdot ,t)\Vert _{L_2(\mathcal{O}')}\leqslant (\widetilde{C}_{16} \delta ^{-1}+ \widehat{C}_{17})\varepsilon t^{-1}e^{-c_{*}t/2}\Vert \boldsymbol{\varphi}\Vert _{L_2(\mathcal{O})}.
\end{align*}
\end{theorem}

Finally, for $d\geqslant 3$ Theorem~3.16 implies the following statement.

\begin{theorem}
Suppose that the assumptions of Theorem \textnormal{4.5} are satisfied, and let $d\geqslant 3$.
Then for $0<\varepsilon \leqslant \min\lbrace \varepsilon _1;(4r_1)^{-1}\delta \rbrace$ and $t>0$ we have
\begin{align*}
&\Vert \mathbf{u}_\varepsilon (\cdot ,t)-\mathbf{u}_0(\cdot ,t)-\varepsilon \Lambda ^\varepsilon b(\mathbf{D})\mathbf{u}_0(\cdot ,t)\Vert _{H^1(\mathcal{O}')}
\leqslant \mathrm{C}_d\varepsilon h_d(\delta ;t)e^{-c_*t/2}\Vert \boldsymbol{\varphi}\Vert _{L_2(\mathcal{O})},
\\
&\Vert \mathbf{p}_\varepsilon (\cdot ,t)-\widetilde{g}^\varepsilon b(\mathbf{D})\mathbf{u}_0(\cdot ,t)\Vert _{L_2(\mathcal{O}')}
\leqslant \widetilde{\mathrm{C}}_d \varepsilon h_d(\delta ;t)e^{-c_*t/2}\Vert \boldsymbol{\varphi}\Vert _{L_2(\mathcal{O})}.
\end{align*}
Here $h_d(\delta ;t)$ is given by~ \eqref{h_d(delta;t)}.
\end{theorem}

\subsection*{4.2. The problem for nonhomogeneous parabolic equation.
Approximation of solutions in $L_2({\mathcal O};{\mathbb C}^n)$.} Now, we consider the problem
\begin{equation}
\label{nonhomogeneous_problem}
\begin{cases}
\frac{\partial \mathbf{u}_\varepsilon(\mathbf{x},t)}{\partial t }=-b(\mathbf{D})^*g^\varepsilon (\mathbf{x})b(\mathbf{D})
\mathbf{u}_\varepsilon (\mathbf{x},t)+\mathbf{F}(\mathbf{x},t),\quad \mathbf{x}\in \mathcal{O},\quad 0<t< T,\\
\mathbf{u}_\varepsilon (\mathbf{x},0)=\boldsymbol{\varphi}(\mathbf{x}),\quad \mathbf{x}\in \mathcal{O},\\
\mathbf{u}_\varepsilon (\cdot ,t)\vert _{\partial \mathcal{O}}=0,\quad 0<t < T,
\end{cases}
\end{equation}
where $\boldsymbol{\varphi}\in L_2(\mathcal{O};\mathbb{C}^n)$ and
 $\mathbf{F}\in\mathfrak{H}_p(T):= L_p((0,T);L_2(\mathcal{O};\mathbb{C}^n))$, $0<T\leqslant \infty$, for some $1\leqslant p\leqslant \infty$. Then
\begin{equation}
\label{u_e=}
\mathbf{u}_\varepsilon (\cdot ,t)=e^{-\mathcal{A}_{D,\varepsilon}t}\boldsymbol{\varphi}(\cdot)+\int _0^t e^{-\mathcal{A}_{D,\varepsilon}(t-\widetilde{t})}\mathbf{F}(\cdot ,\widetilde{t})\,d\widetilde{t}.
\end{equation}

The corresponding effective problem is given by
\begin{equation}
\label{nonhomogeneous_eff_problem}
\begin{cases}
\frac{\partial \mathbf{u}_0(\mathbf{x},t)}{\partial t }=-b(\mathbf{D})^*g^0b(\mathbf{D})
\mathbf{u}_0 (\mathbf{x},t)+\mathbf{F}(\mathbf{x},t),
\quad \mathbf{x}\in \mathcal{O},\quad 0<t < T,\\
\mathbf{u}_0 (\mathbf{x},0)=\boldsymbol{\varphi}(\mathbf{x}),\quad \mathbf{x}\in \mathcal{O},\\
\mathbf{u}_0 (\cdot ,t)\vert _{\partial \mathcal{O}}=0,\quad 0<t< T.
\end{cases}
\end{equation}
Similarly to  (\ref{u_e=}), we have
\begin{equation}
\label{u^0=}
\mathbf{u}_0 (\cdot ,t)=e^{-\mathcal{A}_D^0t}\boldsymbol{\varphi}(\cdot)+\int _0^t e^{-\mathcal{A}_D^0(t-\widetilde{t})}\mathbf{F}(\cdot ,\widetilde{t})\,d\widetilde{t}.
\end{equation}
Subtracting (\ref{u^0=}) from (\ref{u_e=}) and applying Theorem 3.2, we conclude that
\begin{equation}
\label{u_e-u_0<= I+}
\Vert \mathbf{u}_\varepsilon (\cdot ,t)-\mathbf{u}_0(\cdot ,t)\Vert _{L_2(\mathcal{O})}\leqslant C_{12}\varepsilon (t+\varepsilon ^2)^{-1/2}
e^{-c_{*} t/2}\Vert \boldsymbol{\varphi}\Vert _{L_2(\mathcal{O})}+C_{12} \eps \mathcal{L}(\varepsilon;t;{\mathbf F}),
\end{equation}
\begin{equation}
\label{I(e,t)=}
\begin{split}
\mathcal{L}(\varepsilon ; t; {\mathbf F})&:=  \int _0^t e^{-c_{*}(t-\widetilde{t})/2} (\varepsilon ^2 +t-\widetilde{t})^{-1/2}\Vert \mathbf{F}(\cdot ,\widetilde{t})\Vert _{L_2(\mathcal{O})}\,d\widetilde {t},
\end{split}
\end{equation}
for \hbox{$0< \eps \leqslant \eps_1$} and $t>0$.

For $1< p \leqslant \infty$ the following result is true.

\begin{theorem} Suppose that the assumptions of Theorem \textnormal{2.2} are satisfied.
Let $\mathbf{u}_\varepsilon$ be the solution of problem \textnormal{(\ref{nonhomogeneous_problem})},
and let $\mathbf{u}_0$ be the solution of problem \textnormal{(\ref{nonhomogeneous_eff_problem})} for $\boldsymbol{\varphi}\in L_2(\mathcal{O};\mathbb{C}^n)$ and $\mathbf{F}\in \mathfrak{H}_p(T)$, $0<T\leqslant \infty$, with some $1<p\leqslant \infty$.
Then for $0<\varepsilon \leqslant \varepsilon _1$ and $0<t < T$ we have
\begin{equation}
\label{th4.6}
\Vert \mathbf{u}_\varepsilon (\cdot,t)-\mathbf{u}_0(\cdot ,t)\Vert _{L_2(\mathcal{O})}\leqslant C_{12}\varepsilon (t+\varepsilon ^2)^{-1/2}e^{-c_{*}t/2}\Vert \boldsymbol{\varphi}\Vert _{L_2(\mathcal{O})}+
c_p \theta (\varepsilon ,p)\Vert \mathbf{F}\Vert _{\mathfrak{H}_p(t)}.
\end{equation}
Here $\theta (\varepsilon ,p)$ is given by
\begin{equation}
\label{theta (varepsilon ,p)}
\theta (\varepsilon ,p)=\begin{cases}
\varepsilon ^{2-2/p}, &1<p<2 ,\\
\varepsilon (\vert\ln \varepsilon \vert +1)^{1/2},&p=2 ,\\
\varepsilon , &2<p\leqslant \infty .
\end{cases}
\end{equation}
The constant $c_p$ depends only on the initial data \textnormal{\eqref{data}} and $p$.
\end{theorem}

\begin{proof}
By the H\"older inequality, from (\ref{I(e,t)=}) we deduce
\begin{equation}
\label{p=infty}
\mathcal{L}(\varepsilon ; t; {\mathbf F}) \leqslant \Vert \mathbf{F}\Vert _{\mathfrak{H}_p (t)}
\mathcal{I}_{p'}(\varepsilon ; t)^{1/p'},
\end{equation}
\begin{equation}
\label{I gelder 1<p<infty}
\mathcal{I}_{p'}(\varepsilon ; t):=
\int _0^t(\varepsilon ^2 +\tau)^{-p'/2}e^{-c_{*}p'\tau /2}\,d\tau.
\end{equation}
Here $p^{-1}+(p')^{-1}=1$ (if $p=\infty$, then $p'=1$).
For $1<p<2$, estimating the exponential under the integral sign in (\ref{I gelder 1<p<infty}) by 1, we have
\begin{equation}
\label{1<p<2}
\mathcal{I}_{p'}(\varepsilon ;t)\leqslant (p'/2 -1)^{-1} \varepsilon ^{2-p'},\quad 1<p<2.
\end{equation}
For $p=2$ the integral in (\ref{I gelder 1<p<infty})
does not exceed the sum of the integrals $\int_0^1 (\tau + \eps^2)^{-1}\,d\tau$ and $\int_1^\infty e^{-c_* \tau} \,d\tau$.
Hence,
\begin{equation}
\label{I p=2}
\mathcal{I}_2(\varepsilon ;t)\leqslant 2 \vert \ln\varepsilon \vert + \ln 2 + c_*^{-1}
\leqslant \max \{ 2; \ln 2 + c_*^{-1}\} (\vert \ln\varepsilon \vert+1).
\end{equation}
For $2<p \leqslant \infty$ ($1\leqslant p'<2$) we have
\begin{equation}
\label{I 2<p}
\mathcal{I}_{p'}(\varepsilon ;t)\leqslant
\int _0^\infty \tau^{-p'/2}e^{-c_{*}p'\tau /2}\,d\tau =
(c_* p'/2)^{p'/2-1} \Gamma (1-p'/2),\quad 2<p\leqslant \infty.
\end{equation}

Combining (\ref{u_e-u_0<= I+}), (\ref{p=infty}), and (\ref{1<p<2})--(\ref{I 2<p}), we arrive at (\ref{th4.6}). Herewith,
\begin{equation}
\label{c_p}
c_p =\begin{cases}
 C_{12}(p'/2-1)^{-1/p'}, &1<p<2 ,\\
C_{12} \max \{ \sqrt{2}; (\ln 2 + c_*^{-1})^{1/2}\}, &p=2 ,\\
 C_{12}(c_* p'/2)^{1/2-1/p'} (\Gamma (1-p'/2))^{1/p'}, &2<p<  \infty, \\
 C_{12}(2\pi)^{1/2} c_*^{-1/2}, & p = \infty.
\end{cases}
\end{equation}
\end{proof}

\subsection*{4.3. Convergence of the solutions of nonhomogeneous equation in $L_p((0,T);L_2(\mathcal{O};\mathbb{C}^n))$.}
Now, let $1\leqslant p<\infty$.
Using relations \eqref{u_e=} and \eqref{u^0=}, positive definiteness of the operators $\mathcal{A}_{D,\varepsilon}$ and $\mathcal{A}_D^0$, and
the H\"older inequality, it is easy to check that
for $\mathbf{F}\in \mathfrak{H}_p(T)$ the solutions $\mathbf{u}_\varepsilon$ and $\mathbf{u}_0$ also belong to $\mathfrak{H}_p(T)$.
Let us prove the following result.

\begin{theorem}
Suppose that the assumptions of Theorem \textnormal{2.2} are satisfied.
Let $\mathbf{u}_\varepsilon$ be the solution of problem \textnormal{(\ref{nonhomogeneous_problem})}, and let $\mathbf{u}_0$
be the solution of problem \textnormal{(\ref{nonhomogeneous_eff_problem})} for $\boldsymbol{\varphi}\in L_2(\mathcal{O};\mathbb{C}^n)$ and $\mathbf{F}\in \mathfrak{H}_p(T)$, $0<T\leqslant \infty$, with some $1\leqslant p< \infty$. Then
the solutions $\mathbf{u}_\varepsilon$ converge to $\mathbf{u}_0$ in the $\mathfrak{H}_p(T)$-norm, as $\varepsilon \rightarrow 0$.
For $0<\varepsilon\leqslant \varepsilon _1$ we have
\begin{equation*}
\Vert \mathbf{u}_\varepsilon -\mathbf{u}_0\Vert _{\mathfrak{H}_p(T)}\leqslant
c_{p'}\Theta (\varepsilon,p)\Vert \boldsymbol{\varphi}\Vert _{L_2(\mathcal{O})}+ C_{18}\varepsilon \Vert \mathbf{F}\Vert _{\mathfrak{H}_p(T)}.
\end{equation*}
Here $\Theta (\varepsilon,p)=\theta (\varepsilon,p')$ is given by
\begin{equation}
\label{vartheta (varepsilon,p)}
\Theta (\varepsilon,p)=\begin{cases}
\varepsilon ,&1\leqslant p<2,\\
\varepsilon (\vert \ln \varepsilon \vert +1)^{1/2},&p=2,\\
\varepsilon ^{2/p} ,&2<p<\infty;
\end{cases}
\end{equation}
$C_{18}=C_{12} (2\pi )^{1/2} c_*^{-1/2}${\rm ;}
$c_{p'}$ is equal to $c_p$ \textnormal{(}see \textnormal{\eqref{c_p})} with $p$ replaced by $p'$, where \textnormal{$p^{-1}+(p')^{-1}=1$}.
\end{theorem}

\begin{proof}
By (\ref{u_e-u_0<= I+}),
\begin{equation}
\label{70}
\Vert \textbf{u}_\varepsilon -\textbf{u}_0\Vert _{\mathfrak{H}_p(T)} \leqslant C_{12}\varepsilon
\Vert \boldsymbol{\varphi}\Vert _{L_2(\mathcal{O})}
{\mathcal I}_p(\varepsilon;T)^{1/p}+ C_{12}\varepsilon
\left( \int_0^T {\mathcal L}(\varepsilon; t; {\mathbf F})^p\, dt \right)^{1/p},
\end{equation}
where ${\mathcal I}_p(\varepsilon;T)= \int_0^T (\varepsilon ^2+t)^{-p/2} e^{-p c_{*}t/2}\,dt$
(the notation agrees with  (\ref{I gelder 1<p<infty})).
Using estimates (\ref{1<p<2})--(\ref{I 2<p}) (with $p'$ replaced by $p$), we see that the first summand in the right-hand side
of (\ref{70}) does not exceed
$c_{p'}\Theta (\varepsilon,p)\Vert \boldsymbol{\varphi}\Vert _{L_2(\mathcal{O})}$.

Now, we consider the second summand in  (\ref{70}).
For $p=1$, substituting (\ref{I(e,t)=}) and changing the order of integration, we obtain
 \begin{equation*}
\begin{split}
\int_0^T {\mathcal L}(\varepsilon; t; \mathbf{F})\,dt &=
\int_0^T d\widetilde{t}\,\Vert \mathbf{F}(\cdot ,\widetilde{t})\Vert _{L_2(\mathcal{O})}
\int   _{\widetilde{t}}^T dt\, e^{-c_{*}(t-\widetilde{t})/2}(\varepsilon ^2+t-\widetilde{t})^{-1/2}
\\
&\leqslant (2/c_*)^{1/2} \Gamma(1/2) \Vert \mathbf{F} \Vert _{\mathfrak{H}_1(T)}
= (2\pi/c_*)^{1/2}  \Vert \mathbf{F} \Vert _{\mathfrak{H}_1(T)}.
\end{split}
\end{equation*}
For $1<p<\infty$ we apply the H\"older inequality to (\ref{I(e,t)=}):
\begin{equation}
\label{93}
\begin{split}
{\mathcal L}(\varepsilon; t; \mathbf{F}) &\leqslant
\left(\int  _0^t e^{-c_{*}(t-\widetilde{t})/2}(\varepsilon^2 + t-\widetilde{t})^{-1/2}\,d\widetilde{t}\right)^{1/p'}
\\
&\times
 \left(\int  _0^t e^{-c_{*}(t-\widetilde{t})/2}(\varepsilon^2 + t-\widetilde{t})^{-1/2}\Vert \mathbf{F}(\cdot ,\widetilde{t})\Vert _{L_2(\mathcal{O})}^p\,d\widetilde{t}\right)^{1/p}.
\end{split}
\end{equation}
The integral in the first parenthesis on the right does not exceed $(2\pi/c_*)^{1/2}$.
Using (\ref{93}) and changing the order of integration, we have
\begin{equation*}
\begin{split}
&\int_0^T {\mathcal L}(\varepsilon; t; \mathbf{F})^p \,dt
\\
&\leqslant (2\pi/c_*)^{p/2p'}
\int  _0^T d\widetilde{t} \, \Vert \mathbf{F}(\cdot ,\widetilde{t})\Vert _{L_2(\mathcal{O})}^p \int_{\widetilde{t}}^T d{t}\, e^{-c_{*}(t-\widetilde{t})/2}(\varepsilon^2+ t-\widetilde{t})^{-1/2}
\\
&\leqslant (2\pi/c_*)^{p/2} \Vert \mathbf{F}\Vert _{\mathfrak{H}_p(T)}^p, \quad 1< p <\infty.
\end{split}
\end{equation*}
As a result, the second summand in the right-hand side of  (\ref{70})
is estimated by $C_{18} \varepsilon \Vert \mathbf{F}\Vert _{\mathfrak{H}_p(T)}$.
\end{proof}

\begin{remark}\textnormal{For $\boldsymbol{\varphi}=0$ and $\mathbf{F}\in \mathfrak{H}_\infty (T)$
it is possible to prove also convergence of the solutions in $\mathfrak{H}_\infty (T)$. In this case Theorem 4.8 implies that
\begin{equation*}
\Vert \textbf{u}_\varepsilon -\textbf{u}_0\Vert _{\mathfrak{H}_\infty (T)}\leqslant c_\infty \varepsilon
\Vert \mathbf{F}\Vert _{\mathfrak{H}_\infty (T)},\quad 0<\varepsilon \leqslant\varepsilon _1.
\end{equation*}
}
\end{remark}

\subsection*{4.4. The problem for nonhomogeneous parabolic equation. Approximation of the solutions in $H^1({\mathcal O};{\mathbb C}^n)$.}
Now we obtain approximation of the solution of problem (\ref{nonhomogeneous_problem}) in $H^1({\mathcal O};{\mathbb C}^n)$
with the help of Theorem 3.3. Herewith, we have to assume that $2< p \leqslant \infty$.
A difficulty arises in approximation of the integral term in (\ref{u_e=}), since estimate (\ref{Th_exp_korrector})
deteriorates for small values of $t$.
\textit{Assuming that} $t\geqslant \eps ^2$, we represent the integral in (\ref{u_e=}) as the sum of two integrals,
over the intervals $(0,t-\eps ^2)$ and $(t-\eps ^2,t)$;
the first one is estimated with the help of (\ref{Th_exp_korrector}),
and the second one is estimated by using (\ref{small_t}).

We introduce the notation
\begin{equation}
\label{u_ptichka}
{\mathbf w}_\eps(\cdot,t) = e^{-\mathcal{A}_D^0 \varepsilon ^2} \mathbf{u}_0(\cdot,t-\eps ^2),
\end{equation}
where ${\mathbf u}_0$ is the solution of problem (\ref{nonhomogeneous_eff_problem}).
By (\ref{u^0=}),
\begin{equation}
\label{u_ptichka=}
{\mathbf w}_\eps(\cdot,t) = e^{-\mathcal{A}_D^0 t}\boldsymbol{\varphi}(\cdot)+
\int _0^{t-\eps ^2} e^{-\mathcal{A}_{D}^0 (t-\widetilde{t})}\mathbf{F}(\cdot ,\widetilde{t})\,d\widetilde{t}.
\end{equation}
Note that ${\mathbf w}_\eps(\cdot,t)$ is the value at the point $t$ of
the solution of the first initial boundary value problem of the form
(\ref{nonhomogeneous_eff_problem}) on $(0,t)$ with the new right-hand side ${\mathbf F}_\eps$,
where ${\mathbf F}_\eps(\cdot,\tau)= {\mathbf F}(\cdot,\tau)$ for $0< \tau < t-\eps^2$
and ${\mathbf F}_\eps(\cdot,\tau)= 0$ for $t-\eps^2< \tau < t$.

\begin{theorem}
Suppose that the assumptions of Theorem \textnormal{2.2} are satisfied.
Let $\mathbf{u}_\varepsilon$ be the solution of problem \textnormal{(\ref{nonhomogeneous_problem})},
and let $\mathbf{u}_0$ be the solution of problem \textnormal{(\ref{nonhomogeneous_eff_problem})} for
$\boldsymbol{\varphi}\in L_2(\mathcal{O};\mathbb{C}^n)$ and $\mathbf{F}\in \mathfrak{H}_p(T)$, $0<T\leqslant \infty$,
with some $2<p\leqslant \infty$.
Let ${\mathbf w}_\eps$ be defined by \textnormal{(\ref{u_ptichka})}.
Suppose that $\Lambda (\mathbf{x})$ is the $\Gamma$-periodic solution of problem \textnormal{(\ref{Lambda_problem})}.
Let $P_\mathcal{O}$ be the extension operator \textnormal{(\ref{P_O H^1, H^2})},
and let $S_\varepsilon$ be the operator \textnormal{(\ref{S_e})}.
We put $\widetilde{\mathbf{w}}_\varepsilon =P_\mathcal{O}\mathbf{w}_\varepsilon $.
Let ${\mathbf p}_\eps = g^\eps b(\D){\mathbf u}_\eps$, and let
$\widetilde{g}$ be the matrix-valued function \eqref{tilde g}.
Then for $0<\varepsilon \leqslant \varepsilon _1$ and $\eps ^2\leqslant t < T$ we have
\begin{align}
\label{th4.9}
\begin{split}
\|& {\mathbf u}_\eps(\cdot,t) - {\mathbf u}_0(\cdot,t)
- \eps \Lambda^\eps S_\eps b(\D) \widetilde{ {\mathbf w}}_\eps(\cdot,t)\|_{H^1({\mathcal O})}
\\
&\leqslant
2C_{13} \eps^{1/2} t^{-3/4} e^{-c_{*} t/2} \| {\boldsymbol \varphi}\|_{L_2({\mathcal O})}
+ \check{c}_p \rho(\eps,p) \| {\mathbf F} \|_{\mathfrak{H}_p(t)},
\end{split}
\\
\label{th4.9potok}
\begin{split}
\|& {\mathbf p}_\eps(\cdot,t) - \widetilde{g}^\eps S_\eps b(\D) \widetilde{ {\mathbf w}}_\eps(\cdot,t)\|_{L_2({\mathcal O})}
\\
&\leqslant
2\widetilde{C}_{13} \eps^{1/2} t^{-3/4} e^{-c_{*} t/2} \| {\boldsymbol \varphi}\|_{L_2({\mathcal O})}
+ \widetilde{c}_p \rho(\eps,p) \| {\mathbf F} \|_{\mathfrak{H}_p(t)}.
\end{split}
\end{align}
Here
\begin{equation}
\label{rho}
\rho (\varepsilon ,p)=\begin{cases}
\varepsilon ^{1-2/p}, &2<p<4,
\\
\varepsilon ^{1/2}(\vert \ln \varepsilon \vert +1)^{3/4}, &p=4,
\\
\varepsilon ^{1/2},  &4<p\leqslant \infty .
\end{cases}
\end{equation}
The constants $\check{c}_p$ and $\widetilde{c}_p$ depend only on the initial data \textnormal{\eqref{data}} and $p$.
\end{theorem}

\begin{proof}
By (\ref{u_e=}), (\ref{u^0=}), and (\ref{u_ptichka=}), we have
\begin{equation}
\label{315}
\begin{aligned}
\|& {\mathbf u}_\eps(\cdot,t) - {\mathbf u}_0(\cdot,t)
- \eps \Lambda^\eps S_\eps b(\D) \widetilde{ {\mathbf w}}_\eps(\cdot,t)\|_{H^1({\mathcal O})}
\\
&\leqslant \left\| \left( e^{-\mathcal{A}_{D,\eps} t} - e^{-\mathcal{A}_{D}^0 t} - \eps {\mathcal K}_D(t;\eps) \right) {\boldsymbol{\varphi}}\right\|_{H^1({\mathcal O})}
\\
&+ \int_0^{t-\eps ^2 } \left\| \left( e^{-\mathcal{A}_{D,\eps} (t-\widetilde{t})}
- e^{-\mathcal{A}_{D}^0 (t-\widetilde{t})} - \eps {\mathcal K}_D(t-\widetilde{t};\eps) \right) {\mathbf F}(\cdot,\widetilde{t})\right\|_{H^1({\mathcal O})}\, d\widetilde{t}
\\
&+ \int_{t-\eps ^2}^t\left \| \left( e^{-\mathcal{A}_{D,\eps} (t-\widetilde{t})}
- e^{-\mathcal{A}_{D}^0 (t-\widetilde{t})} \right) {\mathbf F}(\cdot,\widetilde{t})\right\|_{H^1({\mathcal O})}\, d\widetilde{t}.
\end{aligned}
\end{equation}
Denote the consecutive summands in the right-hand side of
(\ref{315}) by $\mathcal{J}_1(t;\eps; {\boldsymbol \varphi})$, $\mathcal{J}_2(t;\eps; {\mathbf F})$,
and $\mathcal{J}_3(t;\eps; {\mathbf F})$.
From (\ref{Th_exp_korrector}) and the estimate $\varepsilon t^{-1}\leqslant \varepsilon ^{1/2}t^{-3/4}$,
which is valid for $t\geqslant \varepsilon ^2$, it follows that
\begin{equation}
\label{J_1}
{\mathcal J}_1(t;\eps;{\boldsymbol \varphi}) \leqslant
2C_{13} \eps^{1/2} t^{-3/4} e^{-c_{*} t/2} \| {\boldsymbol \varphi}\|_{L_2({\mathcal O})},\ \ 0< \eps \leqslant \eps_1.
\end{equation}
Let us estimate $\mathcal{J}_2(t;\eps; {\mathbf F})$.
For $\widetilde{t}\in (0,t-\varepsilon ^2)$ we have $t-\widetilde{t}\geqslant \varepsilon ^2$, whence
$\varepsilon (t-\widetilde{t})^{-1}\leqslant \varepsilon ^{1/2}(t-\widetilde{t})^{-3/4}$. Then, by (\ref{Th_exp_korrector}),
\begin{equation}
\label{J_2}
\mathcal{J}_2(t;\eps; {\mathbf F}) \leqslant 2C_{13} \eps^{1/2}
\int_0^{t-\eps ^2} (t- \widetilde{t})^{-3/4} e^{-c_{*} (t-\widetilde{t})/2}
\| {\mathbf F}(\cdot, \widetilde{t})\|_{L_2({\mathcal O})}\,d\widetilde{t},
\end{equation}
for $0< \eps \leqslant \eps_1$.
By the H\"older inequality, from (\ref{J_2}) we obtain
\begin{equation}
\label{J_22}
\begin{split}
\mathcal{J}_2(t;\eps; {\mathbf F}) &\leqslant 2C_{13} \eps^{1/2}
 \| {\mathbf F}\|_{{\mathfrak H}_p(t)} I_{p'}(t;\eps)^{1/p'},
\\
 I_{p'}(t;\eps)&:=
\int_{\eps^2}^{t} \tau^{-3p'/4} e^{-c_{*} p' \tau/2} \,d\tau.
\end{split}
\end{equation}
(For $p=\infty$ and $p'=1$ the inequality remains true.)

Now, we estimate $I_{p'}(t;\eps)$.
If $2<p<4$, then $3p'/4 > 1$, whence
\begin{equation}
\label{I-2_2<p<4}
I_{p'}(t;\eps) \leqslant
\int_{\varepsilon^2}^\infty \tau^{-3p'/4}\,d\tau =
 \nu_p \varepsilon ^{2-3p'/2},\quad 2<p<4.
\end{equation}
If $p=4$, then $p'=4/3$ and
\begin{equation}
\label{I2_p=4}
I_{4/3}(t;\eps)
\leqslant \int _{\varepsilon ^2}^1 \tau ^{-1}\,d \tau +\int _1^\infty e^{-2c_{*}\tau /3}\,d\tau
\leqslant 2\vert \ln \varepsilon \vert + \frac{3}{2 c_{*}}
 \leqslant \nu_4 (\vert \ln \varepsilon \vert +1).
\end{equation}
Finally, if $4<p \leqslant \infty$, then $3p'/4 <1$, and
\begin{equation}
\label{I2_1319}
I_{p'}(t;\eps)\leqslant
\int_0^{\infty} \tau^{-3p'/4} e^{-c_{*} p' \tau/2} \,d\tau =
\nu_p,\quad 4<p \leqslant \infty.
\end{equation}
Here
\begin{equation}
\label{nu_p}
\nu_p=\begin{cases}
(3p'/4-1)^{-1}, &2<p<4,
\\
\max \{ 2; 3(2c_*)^{-1}\}, &p=4,
\\
(c_* p'/2)^{3p'/4-1} \Gamma(1- 3p'/4),  &4<p\leqslant \infty .
\end{cases}
\end{equation}
As a result, relations (\ref{J_22})--(\ref{I2_1319}) imply that
\begin{equation}
\label{J_2_p_infty}
\mathcal{J}_2(t;\eps; {\mathbf F}) \leqslant 2C_{13} \nu_p^{1/p'} \rho(\eps,p)\Vert \mathbf{F}\Vert _{\mathfrak{H}_p (t)},
\quad 2 < p \leqslant \infty,
\end{equation}
where $\rho(\eps,p)$ is defined by (\ref{rho}).

We proceed to the term ${\mathcal J}_3(t;\eps;{\mathbf F})$. Taking \eqref{small_t} into account, we have
\begin{equation}
\label{J_3}
{\mathcal J}_3 (t;\eps;{\mathbf F}) \leqslant C_{14} \int_{t-\eps ^2}^t (t- \widetilde{t})^{-1/2} \| {\mathbf F}(\cdot, \widetilde{t})\|_{L_2({\mathcal O})}\,d\widetilde{t},\quad 0< \eps \leqslant 1.
\end{equation}
By the H\"older inequality, we arrive at
\begin{equation}
\label{J_3<=}
{\mathcal J}_3 (t;\eps;{\mathbf F}) \leqslant C_{14} (1-p'/2)^{-1/p'} \varepsilon^{1-2/p}
\| {\mathbf F} \|_{\mathfrak{H}_p(t)},\ \ 0< \eps \leqslant 1 ,
\end{equation}
which is true for all $2< p \leqslant \infty$.
Note that, if $ p\geqslant 4$, then $\varepsilon^{1-2/p} \leqslant \varepsilon^{1/2}$.

Combining (\ref{315}), (\ref{J_1}), (\ref{J_2_p_infty}), and (\ref{J_3<=}), we arrive at (\ref{th4.9})
with $ \check{c}_p= 2 C_{13} \nu_p^{1/p'} + C_{14}(1-p'/2)^{-1/p'}$.

It remains to check (\ref{th4.9potok}).
According to (\ref{u_e=}) and (\ref{u_ptichka=}),
\begin{equation}
\label{323}
\begin{aligned}
\|& {\mathbf p}_\eps(\cdot,t) - \widetilde{g}^\eps  S_\eps b(\D) \widetilde{ {\mathbf w}}_\eps(\cdot,t) \|_{L_2({\mathcal O})}
\\
&\leqslant \left\| \left( g^\eps b(\D) e^{-\mathcal{A}_{D,\eps} t} -
\widetilde{g}^\eps  S_\eps b(\D) P_{\mathcal O} e^{-\mathcal{A}_{D}^0 t} \right) {\boldsymbol{\varphi}}\right\|_{L_2({\mathcal O})}
\\
&+ \int_0^{t-\eps ^2}\left \| \left( g^\eps b(\D) e^{-\mathcal{A}_{D,\eps} (t-\widetilde{t})}
- \widetilde{g}^\eps  S_\eps b(\D) P_{\mathcal O} e^{-\mathcal{A}_{D}^0 (t-\widetilde{t})} \right)
{\mathbf F}(\cdot,\widetilde{t})\right\|_{L_2({\mathcal O})}\, d\widetilde{t}
\\
&+ \int_{t-\eps ^2}^t\left \| g^\eps b(\D) e^{-\mathcal{A}_{D,\eps} (t-\widetilde{t})}
{\mathbf F}(\cdot,\widetilde{t})\right\|_{L_2({\mathcal O})}\, d\widetilde{t}.
\end{aligned}
\end{equation}
The first two summands in (\ref{323}) are estimated by using
 (\ref{Th_exp_potok}), and the third one is estimated with the help of (\ref{small_t_potok}).
 (Cf. (\ref{J_1})--(\ref{J_3<=}).) This yields (\ref{th4.9potok}) with the constant
$ \widetilde{c}_p= 2 \widetilde{C}_{13} \nu_p^{1/p'} + \widetilde{C}_{14}(1-p'/2)^{-1/p'}$.
\end{proof}

Similarly to the proof of Theorem 4.11, under condition $\Lambda \in L_\infty$ Theorem 3.6 and Proposition 3.5
imply the following result.

\begin{theorem}
Suppose that the assumptions of Theorem \textnormal{4.11} are satisfied.
Suppose that the matrix-valued function $\Lambda (\mathbf{x})$ satisfies Condition \textnormal{2.5}.
Then for \hbox{$0<\varepsilon \leqslant \varepsilon _1$} and $\eps ^2 \leqslant t < T$ we have
\begin{align*}
\begin{split}
\|& {\mathbf u}_\eps(\cdot,t) - {\mathbf u}_0(\cdot,t)
- \eps \Lambda^\eps b(\D) {\mathbf w}_\eps(\cdot,t)\|_{H^1({\mathcal O})}
\\
&\leqslant
2C_{15} \eps^{1/2} t^{-3/4}
 e^{-c_{*} t/2} \| {\boldsymbol \varphi}\|_{L_2({\mathcal O})}
+ c_p' \rho(\eps,p) \| {\mathbf F} \|_{\mathfrak{H}_p(t)},
\end{split}
\\
\begin{split}
\|& {\mathbf p}_\eps(\cdot,t) - \widetilde{g}^\eps b(\D) {\mathbf w}_\eps(\cdot,t)\|_{L_2({\mathcal O})}
\\
&\leqslant
2\widetilde{C}_{15} \eps^{1/2} t^{-3/4} e^{-c_{*} t/2} \| {\boldsymbol \varphi}\|_{L_2({\mathcal O})}
+ c_p'' {\rho}(\eps,p) \| {\mathbf F} \|_{\mathfrak{H}_p(t)}.
\end{split}
\end{align*}
The constants ${c}_p'$ and ${c}_p''$ depend only on the initial data \textnormal{\eqref{data}}, $\|\Lambda\|_{L_\infty}$, and  $p$.
\end{theorem}

In the case of additional smoothness of the boundary one can apply Theorem 3.9.
However, because of the strong growth of the right-hand sides in estimates \eqref{*.4} and \eqref{*.5}
for small $t$, we obtain a meaningful result only in the threedimensional case and only for $p=\infty$.
The proof is similar to that of  Theorem 4.11.

\begin{proposition}
Suppose that the assumptions of Theorem \textnormal{4.11} are satisfied, and let $d=3$, $p=\infty$.
Suppose that $\partial {\mathcal O} \in C^{2,1}$.
Then for $0<\varepsilon \leqslant \varepsilon _1$ and $\eps ^2 \leqslant t < T$ we have
\begin{equation*}
\begin{aligned}
\|& {\mathbf u}_\eps(\cdot,t) - {\mathbf u}_0(\cdot,t)
- \eps \Lambda^\eps b(\D) {\mathbf w}_\eps(\cdot,t)\|_{H^1({\mathcal O})}
\\
&\leqslant
2 {\mathrm C}_{3}^* \left(\eps^{1/2} t^{-3/4}+ \eps t^{-5/4}\right)
 e^{-c_{*} t/2} \| {\boldsymbol \varphi}\|_{L_2({\mathcal O})}
+ c' \eps^{1/2} \| {\mathbf F} \|_{\mathfrak{H}_\infty(t)},
\end{aligned}
\end{equation*}
\begin{equation*}
\begin{aligned}
\|& {\mathbf p}_\eps(\cdot,t) - \widetilde{g}^\eps b(\D) {\mathbf w}_\eps(\cdot,t)\|_{L_2({\mathcal O})}
\\
&\leqslant
2 \widetilde{\mathrm C}_{3}^* \left(\eps^{1/2} t^{-3/4}+ \eps t^{-5/4}\right)
 e^{-c_{*} t/2} \| {\boldsymbol \varphi}\|_{L_2({\mathcal O})}
+ \widetilde{c}' \eps^{1/2} \| {\mathbf F} \|_{\mathfrak{H}_\infty(t)}.
\end{aligned}
\end{equation*}
The constants ${c}'$ and $\widetilde{c}'$ depend only on the initial data \textnormal{\eqref{data}}.
\end{proposition}

Now we distinguish special cases.

\begin{proposition}
Suppose that the assumptions of Theorem \textnormal{4.11} are satisfied.

\noindent
$1^\circ$. If $g^0=\overline{g}$, i.~e., relations \textnormal{(\ref{overline-g})} are satisfied,
then for $0< \eps \leqslant \eps_1$ and $\varepsilon ^2\leqslant t<T$ we have
$$
\Vert \mathbf{u}_\varepsilon (\cdot ,t)-\mathbf{u}_0(\cdot ,t)\Vert _{H^1(\mathcal{O})}\leqslant 2C_{13}\varepsilon^{1/2} t^{-3/4}
e^{-c_{*} t/2}
\Vert \boldsymbol{\varphi}\Vert _{L_2(\mathcal{O})}
+ \check{c}_p \rho(\eps,p) \| {\mathbf F}\|_{\mathfrak{H}_p(t)}.
$$

\noindent
$2^\circ$. If $g^0=\underline{g}$, i.~e., relations \textnormal{(\ref{underline-g})} are satisfied, then for
\hbox{$0< \eps \leqslant \eps_1$} and $\varepsilon ^2\leqslant t<T$ we have
\begin{equation}
\label{ppp}
\Vert \mathbf{p}_\varepsilon (\cdot ,t)-\mathbf{p}_0(\cdot ,t)\Vert _{L_2(\mathcal{O})}\leqslant 2\widetilde{C}_{15}\varepsilon^{1/2} t^{-3/4}
e^{-c_{*} t/2} \Vert \boldsymbol{\varphi}\Vert _{L_2(\mathcal{O})}
+ c_p''' {\rho}(\eps,p) \| {\mathbf F}\|_{\mathfrak{H}_p(t)},
\end{equation}
where $\mathbf{p}_0 = g^0 b(\D) \mathbf{u}_0$.
The constant $ {c}_p'''$ depends on the initial data \textnormal{\eqref{data}},  and also on $n$ and  $p$.
\end{proposition}

\begin{proof}
Assertion $1^\circ$ follows directly from Theorem 4.11, since $\Lambda=0$ if $g^0=\overline{g}$.

Let us prove assertion $2^\circ$. By (\ref{u_e=}) and (\ref{u^0=}),
\begin{equation}
\label{329}
\begin{aligned}
\|& {\mathbf p}_\eps(\cdot,t) - {\mathbf p}_0(\cdot,t)\|_{L_2({\mathcal O})}
\\
&\leqslant \left\| \left( g^\eps b(\D) e^{-\mathcal{A}_{D,\eps} t} - g^0 b(\D)e^{-\mathcal{A}_{D}^0 t} \right) {\boldsymbol{\varphi}}\right\|_{L_2({\mathcal O})}
\\
&+ \int_0^{t} \left\| \left( g^\eps b(\D) e^{-\mathcal{A}_{D,\eps} (t-\widetilde{t})}
- g^0 b(\D)e^{-\mathcal{A}_{D}^0 (t-\widetilde{t})} \right) {\mathbf F}(\cdot,\widetilde{t})\right\|_{L_2({\mathcal O})}\, d\widetilde{t}.
\end{aligned}
\end{equation}
The first term on the right is estimated with the help of Proposition 3.12.
The integral in the right-hand side of (\ref{329}) can be represented as the sum of two integrals, over the intervals $(0,t-\eps ^2)$
and $(t-\eps ^2,t)$. The first one is estimated by using Proposition 3.12 (cf. (\ref{J_2})--(\ref{J_2_p_infty})).
To estimate the second one, we apply (\ref{small_t_potok}) and (\ref{small_t_potok_0}) (cf. (\ref{J_3}), (\ref{J_3<=})).
This yields estimate (\ref{ppp}) with
$ {c}_p'''= 2 \widetilde{C}_{15} \nu_p^{1/p'} + 2 \widetilde{C}_{14} (1- p'/2)^{-1/p'}$.
\end{proof}

\subsection*{4.5. Approximation of the solutions of nonhomogeneous equation in a strictly interior subdomain.}
Similarly to the proof of Theorem 4.11, applying Theorem 3.13 and Proposition 3.5,
it is easy to obtain the following result.

\begin{theorem}
Suppose that the assumptions of Theorem \textnormal{4.11} are satisfied.
Let $\mathcal{O}'$ be a strictly interior subdomain of the domain $\mathcal{O}$. Denote $\delta =\textnormal{dist}\,\lbrace \mathcal{O}'; \partial \mathcal{O}\rbrace$. Then for $0<\varepsilon \leqslant \varepsilon _1$ and $ \varepsilon ^2\leqslant t<T$ we have
\begin{align*}
\begin{split}
&\Vert\mathbf{u}_\varepsilon (\cdot ,t)-\mathbf{u}_0(\cdot ,t)-\varepsilon \Lambda ^\varepsilon S_\varepsilon b(\mathbf{D})\widetilde{\mathbf{w}}_\varepsilon (\cdot ,t)\Vert _{H^1(\mathcal{O}')}\\
&\leqslant (C_{16}\delta ^{-1}+ C_{17})\varepsilon t^{-1}e^{-c_{*}t/2}\Vert \boldsymbol{\varphi}\Vert _{L_2(\mathcal{O})}
+(k_p \delta^{-1}+ k_p')\sigma (\varepsilon ,p)\Vert \mathbf{F}\Vert _{\mathfrak{H}_p(t)},
\end{split}
\\
\begin{split}
&\Vert \mathbf{p}_\varepsilon (\cdot,t)-\widetilde{g}^\varepsilon S_\varepsilon b(\mathbf{D})\widetilde{\mathbf{w}}_\varepsilon (\cdot ,t)\Vert _{L_2(\mathcal{O}')}\\
&\leqslant (\widetilde{C}_{16} \delta ^{-1}+ \widetilde{C}_{17})\varepsilon t^{-1}e^{-c_{*}t/2}\Vert \boldsymbol{\varphi}\Vert _{L_2(\mathcal{O})}+(\widetilde{k}_p \delta^{-1}+ \widetilde{k}_p'){\sigma} (\varepsilon ,p) \Vert \mathbf{F}\Vert _{\mathfrak{H}_p(t)}.
\end{split}
\end{align*}
Here $\sigma (\varepsilon ,p)$ is given by
\begin{equation}
\label{sigma (varepsilon ,p,delta)}
\sigma (\varepsilon ,p) =
\begin{cases}
\varepsilon ^{1-2/p} , &2<p<\infty ,
\\
\varepsilon\left(\vert \ln \varepsilon \vert +1\right) ,  &p=\infty .
\end{cases}
\end{equation}
The constants $k_p$, $k_p'$, $\widetilde{k}_p$, and $\widetilde{k}_p'$ depend only on the initial data \textnormal{\eqref{data}} and  $p$.
\end{theorem}

Finally, in the case where $\Lambda \in L_\infty$ Theorem 3.14 implies the following result.

\begin{theorem}
Suppose that the matrix-valued function $\Lambda (\mathbf{x})$ satisfies Condition \textnormal{2.5}.
Under the assumptions of Theorem \textnormal{4.15} for $0<\varepsilon\leqslant\varepsilon _1$ and $\varepsilon^2\leqslant t<T$ we have
\begin{align*}
\begin{split}
&\Vert\mathbf{u}_\varepsilon (\cdot ,t)-\mathbf{u}_0(\cdot ,t)-\varepsilon
\Lambda ^\varepsilon  b(\mathbf{D})\mathbf{w}_\varepsilon (\cdot ,t)\Vert _{H^1(\mathcal{O}')}\\
&\leqslant (C_{16}\delta ^{-1}+\check{C}_{17})\varepsilon t^{-1}e^{-c_{*}t/2}\Vert \boldsymbol{\varphi}\Vert _{L_2(\mathcal{O})}+
(k_p \delta^{-1} + \check{k}_p') \sigma (\varepsilon ,p)\Vert \mathbf{F}\Vert _{\mathfrak{H}_p(t)},
\end{split}
\\
\begin{split}
&\Vert \mathbf{p}_\varepsilon (\cdot,t)-\widetilde{g}^\varepsilon  b(\mathbf{D})\mathbf{w}_\varepsilon (\cdot ,t)\Vert _{L_2(\mathcal{O}')}\\
&\leqslant (\widetilde{C}_{16}\delta ^{-1}+\widehat{C}_{17}) \varepsilon t^{-1}e^{-c_{*}t/2}\Vert \boldsymbol{\varphi}\Vert _{L_2(\mathcal{O})}+
(\widetilde{k}_p \delta^{-1} + \widehat{k}_p')  \sigma (\varepsilon ,p)\Vert \mathbf{F}\Vert _{\mathfrak{H}_p(t)}.
\end{split}
\end{align*}
The constants $k_p$ and  $\widetilde{k}_p$ are the same as in Theorem \textnormal{4.15}.
The constants $\check{k}_p'$ and $\widehat{k}_p'$ depend only on the initial data
\textnormal{\eqref{data}}, $\|\Lambda\|_{L_\infty}$, and $p$.
\end{theorem}

\begin{remark}
\textnormal{1)
For $p \geqslant 4$ the coefficient at $\Vert \mathbf{F}\Vert _{\mathfrak{H}_p(t)}$
in estimates of Theorems 4.15 and 4.16 has the order (in $\eps$) sharper
than the similar coefficient in estimates of Theorems 4.11 and 4.12.
If $2< p <4$, these coefficients are of the same order $\eps^{1-2/p}$, so that there is no benefit in transition to
a strictly interior subdomain.
2) Application of Theorem 3.16 to approximation of the solutions of nonhomogeneous equation in the class $H^1({\mathcal O}'; {\mathbb C}^n)$
does not lead to interesting results, because of the strong growth of the factor $h_d(\delta;t)$ in estimates \eqref{3.26} and \eqref{3.27}
for small $t$.}
\end{remark}

\subsection*{4.6. Approximation of the solutions of nonhomogeneous equation in $L_p((0,T);H^1(\mathcal{O};\mathbb{C}^n))$.} Denote $\mathfrak{G}_p(T):=L_p((0,T);H^1(\mathcal{O};\mathbb{C}^n))$.

\begin{theorem}
Suppose that the assumptions of Theorem \textnormal{2.2} are satisfied.
Let $\mathbf{u}_\varepsilon$ and $\mathbf{u}_0$ be the solutions of problems \textnormal{(\ref{nonhomogeneous_problem})} and
 \textnormal{(\ref{nonhomogeneous_eff_problem})}, respectively, for $\boldsymbol{\varphi}\in L_2(\mathcal{O};\mathbb{C}^n)$ and $\mathbf{F}\in \mathfrak{H}_p(T)$, $0<T\leqslant \infty$.
For $\varepsilon ^2\leqslant t<T$ let ${\mathbf w}_\eps (\cdot ,t)$ be the function defined by \textnormal{(\ref{u_ptichka})}.
For $0\leqslant t<\varepsilon ^2$ we put $\mathbf{w}_\varepsilon (\cdot ,t)=0$.
Suppose that $\Lambda (\mathbf{x})$ is the $\Gamma$-periodic solution of problem \textnormal{(\ref{Lambda_problem})}.
Let $P_\mathcal{O}$ be the extension operator \textnormal{(\ref{P_O H^1, H^2})}, and let $S_\varepsilon$ be the operator
\textnormal{(\ref{S_e})}. Denote $\widetilde{\mathbf{w}}_\varepsilon =P_\mathcal{O}\mathbf{w}_\varepsilon $.
Let $\mathbf{p}_\varepsilon =g^\varepsilon b(\mathbf{D})\mathbf{u}_\varepsilon$, and let $\widetilde{g}$ be the matrix-valued
function \eqref{tilde g}.

\noindent
$1^\circ.$ Let $1\leqslant p < 2$. Then for $0<\varepsilon\leqslant \varepsilon _1$ we have
\begin{align}
\label{th4.14}
\begin{split}
&\Vert \mathbf{u}_\varepsilon -\mathbf{u}_0-\varepsilon \Lambda ^\varepsilon S_\varepsilon b(\mathbf{D})\widetilde{\mathbf{w}}_\varepsilon\Vert _{\mathfrak{G}_p(T)}
\leqslant \kappa_p \alpha (\varepsilon ,p)\Vert \boldsymbol{\varphi}\Vert _{L_2(\mathcal{O})}+C_{19}\varepsilon ^{1/2}\Vert \mathbf{F}\Vert _{\mathfrak{H}_p(T)},
\end{split}
\\
&\Vert \mathbf{p}_\varepsilon -\widetilde{g}^\varepsilon S_\varepsilon b(\mathbf{D})\widetilde{\mathbf{w}}_\varepsilon \Vert _{L_p((0,T);L_2(\mathcal{O}))}\leqslant \widetilde{\kappa}_p {\alpha}(\varepsilon ,p)\Vert \boldsymbol{\varphi}\Vert _{L_2(\mathcal{O})}+\widetilde{C}_{19}\varepsilon ^{1/2}\Vert \mathbf{F}\Vert _{\mathfrak{H}_p(T)}.\nonumber
\end{align}
Here $\alpha (\varepsilon ,p)= \rho (\varepsilon ,p')$ is defined by
\begin{equation}
\label{alpha(e,p)}
\alpha (\varepsilon ,p)=\begin{cases}
\varepsilon ^{1/2}, &1\leqslant p<4/3,\\
\varepsilon ^{1/2}(\vert \ln \varepsilon \vert +1)^{3/4}, & p=4/3,\\
\varepsilon ^{2/p-1}, &4/3<p<2.
\end{cases}
\end{equation}
The constants $C_{19}$ and $\widetilde{C}_{19}$
depend only on the initial data \textnormal{\eqref{data}}.
The constants ${\kappa} _p$ and $\widetilde{\kappa} _p$
depend on the initial data \textnormal{\eqref{data}} and $p$.

\noindent
$2^\circ.$ Let $\boldsymbol{\varphi}=0$ and $1\leqslant p \leqslant \infty$.
Then for $0<\varepsilon\leqslant \varepsilon _1$ we have
\begin{align}
\label{Prop.3.16}
&\Vert \mathbf{u}_\varepsilon -\mathbf{u}_0-\varepsilon \Lambda ^\varepsilon S_\varepsilon b(\mathbf{D})\widetilde{\mathbf{w}}_\varepsilon \Vert _{\mathfrak{G}_p(T)}\leqslant
C_{19}\varepsilon ^{1/2}\Vert \mathbf{F}\Vert _{\mathfrak{H}_p(T)},
\\
&\Vert \mathbf{p}_\varepsilon -\widetilde{g}^\varepsilon S_\varepsilon b(\mathbf{D})\widetilde{\mathbf{w}}_\varepsilon \Vert _{L_p((0,T);L_2(\mathcal{O}))}\leqslant \widetilde{C}_{19}\varepsilon ^{1/2}\Vert \mathbf{F}\Vert _{\mathfrak{H}_p(T)}.\nonumber
\end{align}
\end{theorem}

\begin{proof}
First, we assume that $T\geqslant \varepsilon^2$.
We have
\begin{equation}
\label{Th3.15_proof_1}
\begin{split}
&\Vert \mathbf{u}_\varepsilon -\mathbf{u}_0-\varepsilon \Lambda ^\varepsilon S_\varepsilon b(\mathbf{D})\widetilde{\mathbf{w}}_\varepsilon\Vert _{\mathfrak{G}_p(T)}
\leqslant
\mathcal{T}_1(\varepsilon;p;\boldsymbol{\varphi}) + \mathcal{T}_2(\varepsilon;p;\boldsymbol{\varphi})
\\
&+\mathcal{T}_3(\varepsilon;p;{\mathbf F})+\mathcal{T}_4(\varepsilon;p;{\mathbf F})
+\mathcal{T}_5(\varepsilon;p;{\mathbf F}),
\end{split}
\end{equation}
where
\begin{align}
\label{T1}
&\mathcal{T}_1(\varepsilon;p;\boldsymbol{\varphi})^p:= \int _0^{\varepsilon ^2}dt\left\Vert \bigl(e^{-\mathcal{A}_{D,\varepsilon}t}-e^{-\mathcal{A}_D^0 t}\bigr)\boldsymbol{\varphi}\right\Vert ^p _{H^1},
\\
&\mathcal{T}_2(\varepsilon;p;\boldsymbol{\varphi})^p:= \int _{\varepsilon ^2}^T dt \left\Vert \bigl(e^{-\mathcal{A}_{D,\varepsilon}t}-e^{-\mathcal{A}_D^0 t}-\varepsilon \mathcal{K}_D(t;\varepsilon)\bigr)\boldsymbol{\varphi}
\right\Vert ^p_{H^1},\nonumber
\\
&\mathcal{T}_3(\varepsilon;p;{\mathbf F})^p:= \int _0^{\varepsilon ^2}dt\biggl\Vert \int _0^t d\widetilde{t}\bigl(e^{-\mathcal{A}_{D,\varepsilon}(t-\widetilde{t})}-e^{-\mathcal{A}_D^0(t-\widetilde{t})}\bigr)\mathbf{F}(\cdot ,\widetilde{t})\biggr\Vert ^p_{H^1},\nonumber
\\
\begin{split}
&\mathcal{T}_4(\varepsilon;p;{\mathbf F})^p:=
\\
 &\int _{\varepsilon ^2}^{T}dt\biggl\Vert \int _0^{t-\varepsilon ^2} d\widetilde{t}\bigl(e^{-\mathcal{A}_{D,\varepsilon}(t-\widetilde{t})}-e^{-\mathcal{A}_D^0(t-\widetilde{t})}-\varepsilon \mathcal{K}_D(t-\widetilde{t};\varepsilon)\bigr)\mathbf{F}(\cdot ,\widetilde{t})\biggr\Vert ^p_{H^1},
\end{split}\nonumber
\\
&\mathcal{T}_5(\varepsilon;p;{\mathbf F})^p:=\int _{\varepsilon ^2}^Tdt\biggl \Vert \int _{t-\varepsilon ^2}^td\widetilde{t}\bigl( e^{-\mathcal{A}_{D,\varepsilon}(t-\widetilde{t})}-e^{-\mathcal{A}_D^0(t-\widetilde{t})}\bigr)\mathbf{F}(\cdot ,\widetilde{t})\biggr\Vert ^p_{H^1}.\nonumber
\end{align}

The integral in \eqref{T1} converges for $1\leqslant p<2$. Applying \eqref{small_t}, we find
\begin{equation}
\label{mathcal_J_1}
\mathcal{T}_1(\varepsilon ;p;\boldsymbol{\varphi})^p
\leqslant
C_{14}^p \Vert \boldsymbol{\varphi}\Vert ^p_{L_2({\mathcal O})}
\int _0^{\varepsilon ^2} t^{-p/2}\,dt = C_{14}^p(1-p/2)^{-1}\varepsilon ^{2-p}\Vert \boldsymbol{\varphi}\Vert ^p_{L_2({\mathcal O})}.
\end{equation}

Next, by \eqref{Th_exp_korrector},
\begin{equation}
\label{T2}
\mathcal{T}_2(\varepsilon ;p;\boldsymbol{\varphi})^p\leqslant
(2C_{13})^p\varepsilon ^{p/2}\Vert \boldsymbol{\varphi}\Vert ^p_{L_2({\mathcal O})}
\int _{\varepsilon ^2}^T t^{-3p/4}e^{-p c_{*}t/2} \,dt.
\end{equation}
The integral in the right-hand side of (\ref{T2}) has been already considered; this is  $I_p(T;\eps)$ (see (\ref{J_22})).
We have $I_p(T;\eps) \leqslant \nu_{p'}$ for $1 \leqslant p < 4/3$;
$I_{4/3}(T;\eps) \leqslant \nu_{4}(|\ln \eps| +1)$;
$I_p(T;\eps) \leqslant \nu_{p'} \eps^{2-3p/2}$ for $4/3 < p <2$.
Together with (\ref{mathcal_J_1}) and (\ref{T2}) this implies
 \begin{equation}
\label{J1+J2<=}
\mathcal{T}_1(\varepsilon ;p;\boldsymbol{\varphi})+
\mathcal{T}_2(\varepsilon ;p;\boldsymbol{\varphi})
\leqslant \kappa_p \alpha (\varepsilon ,p)\Vert \boldsymbol{\varphi}\Vert _{L_2},\quad 1\leqslant p<2,
\end{equation}
where $\kappa _p =C_{14}(1-p/2)^{-1/p}+2C_{13}\nu^{1/p}_{p'}$.

By the H\"older inequality and (\ref{small_t}),
\begin{equation*}
\begin{aligned}
&\mathcal{T}_3 (\varepsilon ;p;{\mathbf F})^p
\leqslant
 C_{14}^p \int _0^{\varepsilon ^2} dt \left(\int _0^t d\widetilde{t} ( t-\widetilde{t})^{-1/2}
 \Vert \mathbf{F}(\cdot ,\widetilde{t})\Vert_{L_2}\right)^p
 \\
&\leqslant
 C_{14}^p\int _0^{\varepsilon ^2} dt \left(\int _0^t ( t-\widetilde{t}) ^{-1/2}\Vert \mathbf{F}(\cdot ,\widetilde{t})\Vert ^p_{L_2}\,d\widetilde{t}\right)
 \left(\int _0 ^t (t-\widetilde{t})^{-1/2}\,d\widetilde{t}\right)^{p/p'}.
 \end{aligned}
\end{equation*}
For $0\leqslant t\leqslant \varepsilon ^2$ we have
$\int _0 ^t(t-\widetilde{t})^{-1/2}\,d\widetilde{t} \leqslant 2\varepsilon $. Whence, changing the order of integration in
the resulting integral, we obtain
\begin{equation}
\label{J3^p<=}
\begin{split}
\mathcal{T}_3(\varepsilon ;p;{\mathbf F})^p &\leqslant
C_{14}^p (2\varepsilon )^{p/p'}\int _0^{\varepsilon ^2} d\widetilde{t}\,\Vert \mathbf{F}(\cdot ,\widetilde{t})\Vert ^p_{L_2}\int _{\widetilde{t}}^{\varepsilon ^2}dt\,(t-\widetilde{t})^{-1/2}
\\
&\leqslant C_{14}^p (2\varepsilon )^{1+p/p'}  \Vert \mathbf{F}\Vert ^p_{\mathfrak{H}_p(T)}.
\end{split}
\end{equation}
The term $\mathcal{T}_5(\varepsilon ;p;{\mathbf F})$ is estimated in a similar way:
\begin{equation}
\label{J5<=}
\mathcal{T}_5(\varepsilon ;p;{\mathbf F})^p\leqslant C_{14}^p(2\varepsilon)^{1+p/p'}\Vert \mathbf{F}\Vert ^p_{\mathfrak{H}_p(T)}.
\end{equation}

It remains to estimate $\mathcal{T}_4(\varepsilon;p;{\mathbf F})$. By \eqref{Th_exp_korrector} and the H\"older inequality,
\begin{equation*}
\begin{split}
&\mathcal{T}_4(\varepsilon ;p;{\mathbf F})^p
\\
&\leqslant
(2C_{13})^p \varepsilon ^{p/2} \int _{\varepsilon ^2}^T dt
\biggl(\int _0^{t-\varepsilon ^2}d\widetilde{t}\,(t-\widetilde{t})^{-3/4}e^{-c_{*}(t-\widetilde{t})/2}
\Vert \mathbf{F}(\cdot ,\widetilde{t})\Vert_{L_2}\biggr)^p
\\
&\leqslant
(2C_{13})^p \varepsilon ^{p/2}\int _{\varepsilon ^2}^T dt
\biggl(\int _0^{t-\varepsilon ^2}d\widetilde{t}\,(t-\widetilde{t})^{-3/4}e^{-c_{*}(t-\widetilde{t})/2}\Vert \mathbf{F}(\cdot ,\widetilde{t})\Vert ^p_{L_2}\biggr)\\
&\times
\biggl(\int _0^{t-\varepsilon ^2}d\widetilde{t}\,(t-\widetilde{t})^{-3/4}e^{-c_{*}(t-\widetilde{t})/2}\biggr)^{p/p'}.
\end{split}
\end{equation*}
The last integral in the parenthesis does not exceed $(c_*/2)^{-1/4}\Gamma(1/4)$.
Changing the order of integration in the resulting integral, we see that
\begin{equation}
\label{J4<=}
\mathcal{T}_4(\varepsilon ;p; {\mathbf F})^p \leqslant
(2C_{13})^p \varepsilon ^{p/2} \left( (c_*/2)^{-1/4}\Gamma (1/4)\right)^{1+p/p'}\Vert \mathbf{F}\Vert ^p_{\mathfrak{H}_p(T)}.
\end{equation}

Combining \eqref{Th3.15_proof_1} and \eqref{J1+J2<=}--\eqref{J4<=}, we obtain estimate (\ref{th4.14})
(under the assumptions of assertion $1^\circ$)
for $T\geqslant \varepsilon ^2$.
Herewith, $C_{19}=4C_{14}+2C_{13}(c_* /2)^{-1/4}\Gamma (1/4)$.
Obviously, in the case where $0<T<\varepsilon ^2$ estimates simplify and rely only on Proposition 3.5.

Now we assume that $\boldsymbol{\varphi}=0$ and $1\leqslant p \leqslant \infty$.
For $1\leqslant p <\infty$ an analog of estimate \eqref{Th3.15_proof_1} holds with
$\mathcal{T}_1(\varepsilon ;p;\boldsymbol{\varphi})=\mathcal{T}_2(\varepsilon ;p;\boldsymbol{\varphi})=0$.
Note that estimates \eqref{J3^p<=}--\eqref{J4<=} remain true for all $1\leqslant p<\infty$, which implies
(\ref{Prop.3.16}) in this case.

For $p=\infty$ we have
\begin{equation}
\label{Prop3.16_proof_1}
\begin{split}
&\Vert \mathbf{u}_\varepsilon -\mathbf{u}_0-\varepsilon \Lambda ^\varepsilon S_\varepsilon b(\mathbf{D})\widetilde{\mathbf{w}}_\varepsilon \Vert _{\mathfrak{G}_\infty(T)} =
\max\{ \esssup\limits _{t \in (0,\eps^2)} \Vert \mathbf{u}_\varepsilon(\cdot,t) -\mathbf{u}_0(\cdot,t) \Vert _{H^1(\mathcal{O})};
\\
 &\esssup\limits _{t \in (\eps^2, T)} \Vert \mathbf{u}_\varepsilon(\cdot,t) -\mathbf{u}_0(\cdot,t)
-\varepsilon \Lambda ^\varepsilon S_\varepsilon b(\mathbf{D})\widetilde{\mathbf{w}}_\varepsilon(\cdot,t)
\Vert _{H^1(\mathcal{O})} \}.
\end{split}
\end{equation}
By Theorem 4.11 for $p=\infty$,
\begin{equation}
\label{Prop3.16_proof_2}
 \esssup\limits _{t \in (\eps^2, T)} \Vert \mathbf{u}_\varepsilon(\cdot,t) -\mathbf{u}_0(\cdot,t)
-\varepsilon \Lambda ^\varepsilon S_\varepsilon b(\mathbf{D})\widetilde{\mathbf{w}}_\varepsilon(\cdot,t)
\Vert _{H^1(\mathcal{O})} \leqslant \check{c}_\infty \eps^{1/2}\Vert \mathbf{F}\Vert _{\mathfrak{H}_\infty (T)},
\end{equation}
where $\check{c}_\infty = 2 C_{13} (c_*/2)^{-1/4}\Gamma(1/4) + 2 C_{14}$.
Taking \eqref{small_t} into account, we have
\begin{equation}
\label{Prop3.16_proof_3}
\begin{split}
&\esssup\limits _{t \in (0,\eps^2)} \Vert \mathbf{u}_\varepsilon(\cdot,t) -\mathbf{u}_0(\cdot,t) \Vert _{H^1(\mathcal{O})}
\\
&= \esssup\limits _{t \in (0,\eps^2)}
\biggl\Vert \int _0^t  \bigl( e^{-\mathcal{A}_{D,\varepsilon}(t-\widetilde{t})}-e^{-\mathcal{A}_D^0(t-\widetilde{t})}\bigr)
\mathbf{F}(\cdot,\widetilde{t}) d \widetilde{t} \biggr\Vert_{H^1(\mathcal{O})}
\\
&\leqslant
 C_{14}   \esssup\limits _{t \in (0,\eps^2)} \int _0^t (t-\widetilde{t})^{-1/2}\Vert \mathbf{F}(\cdot ,\widetilde{t})\Vert _{L_2}\,d\widetilde{t}
\leqslant 2\varepsilon C_{14}\Vert \mathbf{F}\Vert _{\mathfrak{H}_\infty (T)}.
\end{split}
\end{equation}

Relations (\ref{Prop3.16_proof_1})--(\ref{Prop3.16_proof_3}) imply (\ref{Prop.3.16}) for $p=\infty$.

The assertion concerning the fluxes $\mathbf{p}_\varepsilon$ can be obtained in a similar way
with the help of estimates \eqref{Th_exp_potok} and \eqref{small_t_potok}.
\end{proof}

Under the additional Condition 2.5 on the matrix-valued function $\Lambda (\mathbf{x})$
the following theorem is true. Its proof is completely analogous to that of Theorem
4.18; the difference is that instead of \eqref{Th_exp_korrector} and \eqref{Th_exp_potok}
one should use (\ref{th3.5}) and (\ref{th3.5_2}).

\begin{theorem}
Suppose that the assumptions of Theorem {\rm 4.18} and Condition {\rm 2.5} are satisfied.

\noindent
$1^\circ.$ Let $1\leqslant p < 2$. Then for $0<\varepsilon\leqslant \varepsilon _1$ we have
\begin{align*}
&\Vert \mathbf{u}_\varepsilon -\mathbf{u}_0-\varepsilon \Lambda ^\varepsilon b(\mathbf{D}){\mathbf{w}}_\varepsilon\Vert _{\mathfrak{G}_p(T)}
\leqslant \varkappa_p \alpha (\varepsilon ,p)\Vert \boldsymbol{\varphi}\Vert _{L_2(\mathcal{O})}+C_{20}\varepsilon ^{1/2}\Vert \mathbf{F}\Vert _{\mathfrak{H}_p(T)},
\\
&\Vert \mathbf{p}_\varepsilon -\widetilde{g}^\varepsilon  b(\mathbf{D}){\mathbf{w}}_\varepsilon \Vert _{L_p((0,T);L_2(\mathcal{O}))}\leqslant \widetilde{\varkappa}_p {\alpha}(\varepsilon ,p)\Vert \boldsymbol{\varphi}\Vert _{L_2(\mathcal{O})}+\widetilde{C}_{20}\varepsilon ^{1/2}\Vert \mathbf{F}\Vert _{\mathfrak{H}_p(T)}.
\end{align*}
The constants $C_{20}$ and $\widetilde{C}_{20}$ depend only on the initial data
\textnormal{\eqref{data}} and $\|\Lambda\|_{L_\infty}$.
The constants ${\varkappa} _p$ and $\widetilde{\varkappa} _p$ depend on the same parameters and $p$.

\noindent
$2^\circ.$ Let $\boldsymbol{\varphi}=0$ and $1\leqslant p \leqslant \infty$.
Then for $0<\varepsilon\leqslant \varepsilon _1$ we have
\begin{align*}
&\Vert \mathbf{u}_\varepsilon -\mathbf{u}_0-\varepsilon \Lambda ^\varepsilon b(\mathbf{D}){\mathbf{w}}_\varepsilon \Vert _{\mathfrak{G}_p(T)}\leqslant
C_{20}\varepsilon ^{1/2}\Vert \mathbf{F}\Vert _{\mathfrak{H}_p(T)},
\\
&\Vert \mathbf{p}_\varepsilon -\widetilde{g}^\varepsilon  b(\mathbf{D}){\mathbf{w}}_\varepsilon \Vert _{L_p((0,T);L_2(\mathcal{O}))}\leqslant \widetilde{C}_{20}\varepsilon ^{1/2}\Vert \mathbf{F}\Vert _{\mathfrak{H}_p(T)}.
\end{align*}
\end{theorem}

In the case of additional smoothness of the boundary one can apply Theorem 3.9.
As well as in Proposition 4.13, we obtain a meaningful result only in the threedimensional case.
Its proof is similar to that of Theorem  4.18.

\begin{proposition}
Suppose that the assumptions of Theorem \textnormal{4.18} are satisfied, and let $d=3$.
Suppose that $\partial {\mathcal O} \in C^{2,1}$.

\noindent
$1^\circ.$ Let $1\leqslant p < 4/3$. Then for $0<\varepsilon\leqslant \varepsilon _1$ we have
\begin{align*}
\begin{split}
&\Vert \mathbf{u}_\varepsilon -\mathbf{u}_0-\varepsilon \Lambda ^\varepsilon b(\mathbf{D}){\mathbf{w}}_\varepsilon\Vert _{\mathfrak{G}_p(T)}
\leqslant
\varkappa'_p \varepsilon^{2/p - 3/2} \Vert \boldsymbol{\varphi}\Vert _{L_2(\mathcal{O})}+
C_{21}\varepsilon ^{1/2}\Vert \mathbf{F}\Vert _{\mathfrak{H}_p(T)},
\end{split}
\\
&\Vert \mathbf{p}_\varepsilon -\widetilde{g}^\varepsilon  b(\mathbf{D}){\mathbf{w}}_\varepsilon \Vert _{L_p((0,T);L_2(\mathcal{O}))}\leqslant \widetilde{\varkappa}'_p  \varepsilon^{2/p - 3/2} \Vert \boldsymbol{\varphi}\Vert _{L_2(\mathcal{O})}+
\widetilde{C}_{21}\varepsilon ^{1/2}\Vert \mathbf{F}\Vert _{\mathfrak{H}_p(T)}.
\end{align*}
The constants $C_{21}$ and $\widetilde{C}_{21}$
depend only on the initial data \textnormal{\eqref{data}}.
The constants ${\varkappa} _p'$ and $\widetilde{\varkappa} _p'$
depend on the same parameters and $p$.

\noindent
$2^\circ.$ Let $\boldsymbol{\varphi}=0$, and let $1\leqslant p \leqslant \infty$.
Then for $0<\varepsilon\leqslant \varepsilon _1$ we have
\begin{align*}
&\Vert \mathbf{u}_\varepsilon -\mathbf{u}_0-\varepsilon \Lambda ^\varepsilon b(\mathbf{D}){\mathbf{w}}_\varepsilon \Vert _{\mathfrak{G}_p(T)}\leqslant
C_{21}\varepsilon ^{1/2}\Vert \mathbf{F}\Vert _{\mathfrak{H}_p(T)},
\\
&\Vert \mathbf{p}_\varepsilon -\widetilde{g}^\varepsilon  b(\mathbf{D}){\mathbf{w}}_\varepsilon \Vert _{L_p((0,T);L_2(\mathcal{O}))}\leqslant \widetilde{C}_{21}\varepsilon ^{1/2}\Vert \mathbf{F}\Vert _{\mathfrak{H}_p(T)}.
\end{align*}
\end{proposition}

\subsection*{4.7. Approximation of the solutions in $L_p((0,T);H^1(\mathcal{O}';\mathbb{C}^n))$.}
By analogy with Theorem 4.18, one can check next result with the help of Theorem 3.13 and Proposition 3.5.

\begin{theorem}
Suppose that the assumptions of Theorem {\rm 4.18} are satisfied.
Let $\mathcal{O}'$ be a strictly interior subdomain of the domain $\mathcal{O}$, and let $\delta =\textnormal{dist}\,\lbrace \mathcal{O}';\partial \mathcal{O}\rbrace $.

\noindent
$1^\circ .$ Let $1\leqslant p<2$. Then for $0<\varepsilon \leqslant \varepsilon _1$ we have
\begin{align*}
\begin{split}
\Vert &\mathbf{u}_\varepsilon -\mathbf{u}_0-\varepsilon \Lambda ^\varepsilon S_\varepsilon b(\mathbf{D})\widetilde{\mathbf{w}}_\varepsilon \Vert _{L_p((0,T);H^1(\mathcal{O}'))}\\
&\leqslant (k_{p'}\delta ^{-1}+k'_{p'}) \tau(\varepsilon,p)
\Vert \boldsymbol{\varphi}\Vert _{L_2(\mathcal{O})}+
(C_{22}' \delta ^{-1}+C''_{22} )\varepsilon (\vert \ln \varepsilon\vert +1)\Vert \mathbf{F}\Vert _{\mathfrak{H}_p(T)},
\end{split}
\\
\begin{split}
\Vert &\mathbf{p}_\varepsilon -\widetilde{g}^\varepsilon S_\varepsilon b(\mathbf{D})\widetilde{\mathbf{w}}_\varepsilon \Vert _{L_p((0,T);L_2(\mathcal{O}'))}\\
&\leqslant (\widetilde{k}_{p'}\delta ^{-1}+\widetilde{k}_{p'}') \tau(\varepsilon,p) \Vert \boldsymbol{\varphi}\Vert _{L_2(\mathcal{O})}
+ (\widetilde{C}_{22}' \delta ^{-1}+\widetilde{C}''_{22} )\varepsilon (\vert \ln \varepsilon \vert +1)\Vert \mathbf{F}\Vert _{\mathfrak{H}_p(T)}.
\end{split}
\end{align*}
Here $\tau(\varepsilon,p)=\sigma(\varepsilon,p')$ is defined by
\begin{equation}
\label{tau (varepsilon ,p)}
\tau (\varepsilon ,p) =
\begin{cases}
\varepsilon ^{2/p-1} , &1<p<2 ,
\\
\varepsilon\left(\vert \ln \varepsilon \vert +1\right) ,  &p=1.
\end{cases}
\end{equation}

\noindent
$2^\circ .$ Let $\boldsymbol{\varphi}=0$, and let $1\leqslant p\leqslant \infty$. Then for $0<\varepsilon \leqslant \varepsilon _1$ we have
\begin{align*}
\begin{split}
\Vert &\mathbf{u}_\varepsilon -\mathbf{u}_0-\varepsilon \Lambda ^\varepsilon S_\varepsilon b(\mathbf{D})\widetilde{\mathbf{w}}_\varepsilon \Vert _{L_p((0,T);H^1(\mathcal{O}'))}\\
&\leqslant (C_{22}' \delta ^{-1}+C''_{22} )
\varepsilon (\vert \ln \varepsilon\vert +1)\Vert \mathbf{F}\Vert _{\mathfrak{H}_p(T)},
\end{split}
\\
\Vert &\mathbf{p}_\varepsilon -\widetilde{g}^\varepsilon S_\varepsilon b(\mathbf{D})\widetilde{\mathbf{w}}_\varepsilon \Vert _{L_p((0,T);L_2(\mathcal{O}'))}
\leqslant
(\widetilde{C}_{22}' \delta ^{-1}+\widetilde{C}''_{22} )
\varepsilon (\vert \ln \varepsilon \vert +1)\Vert \mathbf{F}\Vert _{\mathfrak{H}_p(T)}.
\end{align*}
The constants ${C}_{22}'$, ${C}_{22}''$, $\widetilde{C}_{22}'$, and $\widetilde{C}_{22}''$
depend only on the initial data \textnormal{\eqref{data}}.
The constants  $k_{p'}$, $k_{p'}'$, $\widetilde{k}_{p'}$, and $\widetilde{k}'_{p'}$ are the same as in Theorem {\rm 4.15}
\textnormal{(}with $p$ replaced by $p'$\textnormal{)}.
\end{theorem}

Finally, in the case where $\Lambda \in L_\infty$ Theorem 3.14 and Proposition 3.5 imply the following result.

\begin{theorem}
Suppose that the assumptions of Theorem {\rm 4.21} are satisfied.
Suppose that the matrix-valued function $\Lambda (\mathbf{x})$ satisfies Condition {\rm 2.5}.

\noindent
$1^\circ .$ Let $1\leqslant p<2$. Then for $0<\varepsilon \leqslant \varepsilon _1$ we have
\begin{align*}
\begin{split}
\Vert &\mathbf{u}_\varepsilon -\mathbf{u}_0-\varepsilon \Lambda ^\varepsilon  b(\mathbf{D})\mathbf{w}_\varepsilon \Vert _{L_p((0,T);H^1(\mathcal{O}'))}\\
&\leqslant (k_{p'}\delta ^{-1}+\check{k}'_{p'}) \tau(\varepsilon,p) \Vert \boldsymbol{\varphi}\Vert _{L_2(\mathcal{O})}+
(C'_{22} \delta ^{-1}+\check{C}''_{22} )\varepsilon (\vert \ln \varepsilon\vert +1)\Vert \mathbf{F}\Vert _{\mathfrak{H}_p(T)},
\end{split}
\\
\begin{split}
\Vert &\mathbf{p}_\varepsilon -\widetilde{g}^\varepsilon  b(\mathbf{D})\mathbf{w}_\varepsilon \Vert _{L_p((0,T);L_2(\mathcal{O}'))}\\
&\leqslant (\widetilde{k}_{p'}\delta ^{-1}+\widehat{k}_{p'}')  \tau(\varepsilon,p) \Vert \boldsymbol{\varphi}\Vert _{L_2(\mathcal{O})}
+(\widetilde{C}'_{22} \delta ^{-1}+\widehat{C}''_{22} )\varepsilon (\vert \ln \varepsilon \vert +1)\Vert \mathbf{F}\Vert _{\mathfrak{H}_p(T)}.
\end{split}
\end{align*}

\noindent
$2^\circ .$ Let $\boldsymbol{\varphi}=0$, and let $1\leqslant p\leqslant \infty$. Then for $0<\varepsilon \leqslant \varepsilon _1$
we have
\begin{align*}
\begin{split}
\Vert &\mathbf{u}_\varepsilon -\mathbf{u}_0-\varepsilon \Lambda ^\varepsilon  b(\mathbf{D})\mathbf{w}_\varepsilon \Vert _{L_p((0,T);H^1(\mathcal{O}'))}\\
&\leqslant (C'_{22} \delta ^{-1}+\check{C}''_{22} )
\varepsilon (\vert \ln \varepsilon\vert +1)\Vert \mathbf{F}\Vert _{\mathfrak{H}_p(T)},
\end{split}
\\
\Vert &\mathbf{p}_\varepsilon -\widetilde{g}^\varepsilon  b(\mathbf{D})\mathbf{w}_\varepsilon \Vert _{L_p((0,T);L_2(\mathcal{O}'))}
\leqslant (\widetilde{C}'_{22} \delta ^{-1}+\widehat{C}''_{22} )
\varepsilon (\vert \ln \varepsilon \vert +1)\Vert \mathbf{F}\Vert _{\mathfrak{H}_p(T)}.
\end{align*}
The constants ${C}_{22}'$ and $\widetilde{C}_{22}'$ are the same as in Theorem {\rm 4.21}.
The constants $\check{C}_{22}''$ and $\widehat{C}_{22}''$ depend only on the initial data \textnormal{\eqref{data}} and $\|\Lambda\|_{L_\infty}$.
The constants  $k_{p'}$, $\check{k}_{p'}'$, $\widetilde{k}_{p'}$, and $\widehat{k}'_{p'}$ are the same as in Theorem {\rm 4.16}
\textnormal{(}with $p$ replaced by $p'$\textnormal{)}.
\end{theorem}

Note that application of Theorem 3.16 to approximation of the solutions of nonhomogeneous equation in the class
$L_p((0,T); H^1({\mathcal O}';{\mathbb C}^n))$ does not lead to interesting results because of the strong growth
of the factor $h_d(\delta;t)$ in estimates \eqref{3.26} and \eqref{3.27} for small $t$.

\section*{Chapter 2. Homogenization of the second initial boundary value problem for parabolic systems}

\section*{\S 5. The class of operators $\mathcal{A}_{N,\varepsilon}$. The effective operator}
\setcounter{section}{5}
\setcounter{equation}{0}
\setcounter{theorem}{0}

\subsection*{5.1. Coercivity.} We impose an additional condition on the symbol $b(\boldsymbol{\xi})=\sum _{l=1}^d b_l\xi _l$
for $\boldsymbol{\xi}\in \mathbb{C}^d$.

\begin{condition} The matrix-valued function $b(\boldsymbol{\xi})$, $\boldsymbol{\xi}\in\mathbb{C}^d$, is such that
\begin{equation}
\label{rank b in C^d}
\textnormal{rank}\,b(\boldsymbol{\xi})=n,\quad 0\neq \boldsymbol{\xi}\in\mathbb{C}^d.
\end{equation}
\end{condition}

Note that condition (\ref{rank b in C^d}) is more restrictive than (\ref{rank b in R^d}).
According to \cite{Ne} (see Theorem~7.8 in section 3.7), Condition 5.1 \textit{is equivalent to coercivity of the form
$\Vert b(\mathbf{D})\mathbf{u}\Vert ^2_{L_2(\mathcal{O})}$ in $H^1(\mathcal{O};\mathbb{C}^n)$}.
Moreover, the following statement holds under the only condition that $\partial {\mathcal O}$ is Lipschitz.

\begin{proposition}[\cite{Ne}]
Condition \textnormal{5.1} is necessary and sufficient for existence of constants ${\mathfrak C}_1>0$ and ${\mathfrak C}_2 \geqslant 0$
such that the G\"arding-type inequality
\begin{equation}
\label{Garding}
\Vert b(\mathbf{D})\mathbf{u}\Vert ^2_{L_2(\mathcal{O})}+{\mathfrak C}_2\Vert \mathbf{u}\Vert ^2_{L_2(\mathcal{O})}\geqslant
{\mathfrak C}_1\Vert \mathbf{D}\mathbf{u}\Vert ^2_{L_2(\mathcal{O})},\quad \mathbf{u}\in H^1(\mathcal{O};\mathbb{C}^n),
\end{equation}
is satisfied.
\end{proposition}

\begin{remark}
\textnormal{The constants ${\mathfrak C}_1$ and ${\mathfrak C}_2$ depend on the symbol $b(\boldsymbol{\xi})$
and the domain $\mathcal{O}$, but in the general case it is difficult to control these constants explicitly.
However, often for particular operators they can be found. Therefore, in what follows we refer to dependence of other constants on
${\mathfrak C}_1$ and ${\mathfrak C}_2$.}
\end{remark}

\subsection*{5.2. The class of operators.} In $L_2(\mathcal{O};\mathbb{C}^n)$, consider the operator
$\mathcal{A}_{N,\varepsilon}$ formally given by the differential expression
$b(\mathbf{D})^*g^\varepsilon(\mathbf{x})b(\mathbf{D})$ with the Neumann condition on $\partial \mathcal{O}$.
The precise definition of the operator $\mathcal{A}_{N,\varepsilon}$ is given in terms of the quadratic form
\begin{equation}
\label{a_N,e=}
a_{N,\varepsilon}[\mathbf{u},\mathbf{u}]:=\int _\mathcal{O}\langle g^\varepsilon (\mathbf{x})b(\mathbf{D})\mathbf{u},b(\mathbf{D})\mathbf{u}\rangle \,d\mathbf{x},\quad \mathbf{u}\in H^1(\mathcal{O};\mathbb{C}^n).
\end{equation}

By (\ref{b(D)=}) and (\ref{|b_l|<=}),
\begin{equation}
\label{a_N,e<}
a_{N,\varepsilon}[\mathbf{u},\mathbf{u}]\leqslant d\alpha _1\Vert g\Vert _{L_\infty}\Vert \mathbf{D}\mathbf{u}\Vert ^2_{L_2(\mathcal{O})},\quad \mathbf{u}\in H^1(\mathcal{O};\mathbb{C}^n).
\end{equation}
From (\ref{Garding}) it follows that
\begin{equation}
\label{a_N,e>}
a_{N,\varepsilon}[\mathbf{u},\mathbf{u}]\geqslant \Vert g^{-1}\Vert _{L_\infty}^{-1}\left( {\mathfrak C}_1\Vert \mathbf{D}\mathbf{u}\Vert ^2_{L_2(\mathcal{O})}- {\mathfrak C}_2\Vert \mathbf{u}\Vert ^2_{L_2(\mathcal{O})}\right),\quad \mathbf{u}\in H^1(\mathcal{O};\mathbb{C}^n).
\end{equation}
Relations (\ref{a_N,e=})--(\ref{a_N,e>}) show that the form (\ref{a_N,e=}) is closed and nonnegative.

\subsection*{5.3. The effective operator.} In $L_2(\mathcal{O};\mathbb{C}^n)$, consider the operator $\mathcal{A}_N^0$
generated by the quadratic form
\begin{equation}
\label{a_n^0=}
a_N^0[\mathbf{u},\mathbf{u}]=\int _\mathcal{O}\langle g^0 b(\mathbf{D})\mathbf{u},b(\mathbf{D})\mathbf{u}\rangle \,d\mathbf{x},\quad \mathbf{u}\in H^1(\mathcal{O};\mathbb{C}^n).
\end{equation}
Here $g^0$ is the constant effective matrix defined by (\ref{g^0}).
By (\ref{|g^0|, |g^0^_1|<=}), the form (\ref{a_n^0=}) satisfies estimates of the form (\ref{a_N,e<}) and (\ref{a_N,e>})
with the same constants.

Let $\boldsymbol{\nu} ({\mathbf x})$ be the unit external normal vector to $\partial {\mathcal O}$ at the point
${\mathbf x} \in \partial {\mathcal O}$. Let $\partial^0_{\boldsymbol{\nu}}$ be the conormal derivative, i.~e.,
$\partial^0_{\boldsymbol{\nu}} \mathbf{u}(\mathbf{x})= b(\boldsymbol{\nu}({\mathbf x}))^* g^0 b(\nabla) \mathbf{u}(\mathbf{x})$.
Since $\partial \mathcal{O}\in C^{1,1}$ and $b(\mathbf{D})^*g^0b(\mathbf{D})$ is a strongly elliptic operator (with constant coefficients),
the operator $\mathcal{A}_N^0$ can be given by the differential expression $b(\mathbf{D})^* g^0 b(\mathbf{D})$
on the domain $\{ \mathbf{u}\in H^2(\mathcal{O};\mathbb{C}^n):\ \partial^0_{\boldsymbol{\nu}} {\mathbf u}\vert_{\partial
\mathcal{O}}=0 \}$. Herewith,
\begin{equation}
\label{5.6a}
\Vert (\mathcal{A}_N^0+ I)^{-1}\Vert _{L_2(\mathcal{O})\rightarrow H^2(\mathcal{O})}\leqslant c^\circ.
\end{equation}
The constant $c^\circ$ depends only on the constants ${\mathfrak C}_1$ and ${\mathfrak C}_2$ from (\ref{Garding}),
on $\alpha _0$, $\alpha _1$, $\Vert g\Vert_{L_\infty}$, $\Vert g^{-1}\Vert _{L_\infty}$, and the domain $\mathcal{O}$.
To justify this fact, we refer to the results about regularity of the solutions of strongly elliptic systems (see, e.~g., \cite[Chapter 4]{McL}).

The following remark is similar to Remark 1.4.

\begin{remark}
\textnormal{Instead of condition $\partial \mathcal{O}\in C^{1,1}$, one could impose the following implicit
condition on the domain: a bounded domain $\mathcal{O}$ with Lipschitz boundary is such that
estimate \eqref{5.6a} is satisfied. For such domain the results of Chapter 2 remain valid (except for the results of Subsection 7.6, Theorem 8.3, and Propositions 8.12 and 8.18, where an additional smoothness of the boundary is required).
In the case of scalar elliptic operators wide conditions on  $\partial \mathcal{O}$ sufficient for estimate \eqref{5.6a} can be found in
[KoE] and [MaSh, Chapter 7] (in particular, it suffices that $\partial \mathcal{O}\in C^{\alpha}$, $\alpha >3/2$).
}
\end{remark}

\subsection*{5.4.} Denote
\begin{equation*}
Z:=\textnormal{Ker}\,b(\mathbf{D})=\lbrace \mathbf{z}\in H^1(\mathcal{O};\mathbb{C}^n) : b(\mathbf{D})\mathbf{z}=0 \rbrace .
\end{equation*}
From (\ref{Garding}) and the compactness of the embedding of $H^1(\mathcal{O};\mathbb{C}^n)$ into $L_2(\mathcal{O};\mathbb{C}^n)$
it follows that the dimension of $Z$ is finite. Obviously, $Z$ contains the $n$-dimensional subspace of constant vector-valued functions.
Denote $q=\textnormal{dim}\,Z$. We put
\begin{equation}
\label{H(O)}
\mathcal{H}(\mathcal{O}):=L_2(\mathcal{O};\mathbb{C}^n)\ominus Z.
\end{equation}
Let $\mathcal{P}$ be the orthogonal projection of $L_2(\mathcal{O};\mathbb{C}^n)$ onto $\mathcal{H}(\mathcal{O})$.
Next, let $ H^1_\perp (\mathcal{O};\mathbb{C}^n):=H^1(\mathcal{O};\mathbb{C}^n)\cap \mathcal{H}(\mathcal{O})$.
In other words,
\begin{equation*}
H^1_\perp (\mathcal{O};\mathbb{C}^n)=\lbrace \mathbf{u}\in H^1(\mathcal{O};\mathbb{C}^n) : (\mathbf{u},\mathbf{z})_{L_2(\mathcal{O})}=0, \; \forall\, \mathbf{z}\in Z\rbrace .
\end{equation*}
As shown in \cite[Proposition 9.1]{Su_SIAM},
the form $\Vert b(\mathbf{D})\mathbf{u}\Vert ^2 _{L_2(\mathcal{O})}$ defines a norm in the space
$H^1_\perp (\mathcal{O};\mathbb{C}^n)$ equivalent to the standard $H^1$-norm,
i.~e., there exists a constant $\widetilde{\mathfrak C}_1>0$ such that
\begin{equation}
\label{b(D)u>= in H_perp}
\widetilde{\mathfrak C}_1 \Vert \mathbf{u}\Vert ^2_{H^1(\mathcal{O})}\leqslant \Vert b(\mathbf{D})\mathbf{u}\Vert ^2_{L_2(\mathcal{O})},\quad \mathbf{u}\in H^1_\perp (\mathcal{O};\mathbb{C}^n).
\end{equation}

Obviously,
\begin{equation}
\label{Ker=Z}
\textnormal{Ker}\,\mathcal{A}_{N,\varepsilon}=\textnormal{Ker}\,\mathcal{A}_N^0=Z.
\end{equation}
Therefore, the orthogonal decomposition $L_2(\mathcal{O};\mathbb{C}^n)=Z\oplus \mathcal{H}(\mathcal{O})$
reduces the operators $\mathcal{A}_{N,\varepsilon}$ and $\mathcal{A}_N^0$.
 By \eqref{b(D)u>= in H_perp},
\begin{equation}
\label{<b_N,e<}
\begin{split}
a_{N,\varepsilon}[\mathbf{u},\mathbf{u}]\geqslant \Vert g^{-1}\Vert _{L_\infty}^{-1}\widetilde{\mathfrak C}_1 \Vert \mathbf{u}\Vert ^2_{H^1(\mathcal{O})},\quad \mathbf{u}\in H^1_\perp (\mathcal{O};\mathbb{C}^n),\\
a_N^0[\mathbf{u},\mathbf{u}]\geqslant \Vert g^{-1}\Vert _{L_\infty}^{-1}\widetilde{\mathfrak C}_1 \Vert \mathbf{u}\Vert ^2_{H^1(\mathcal{O})},\quad
\mathbf{u}\in H^1_\perp (\mathcal{O};\mathbb{C}^n).
\end{split}
\end{equation}

\section*{\S 6. Preliminaries. \\ Approximation of the resolvent of the operator $\mathcal{A}_{N,\varepsilon}$}
\setcounter{section}{6}
\setcounter{equation}{0}
\setcounter{theorem}{0}

\subsection*{6.1. Approximation of the resolvent of the operator $\mathcal{A}_{N,\varepsilon}$
in the $L_2(\mathcal{O};{\mathbb C}^n)$-operator norm}
 In  [Su7,8], approximation of the resolvent of the operator $(\mathcal{A}_{N,\varepsilon}-\zeta I)^{-1}$
was obtained for $\zeta \in \mathbb{C}\setminus [c_\flat ,\infty)$, $\zeta \neq 0$. Here
 $0<c_\flat \leqslant \min \lbrace \mu _{2,\varepsilon} ;\mu _2^0\rbrace$, where $\mu_{2,\varepsilon}$ and $\mu _2^0$
are the first nonzero eigenvalues of the operators $\mathcal{A}_{N,\varepsilon}$ and $\mathcal{A}_N^0$, respectively.
(If multiplicities are taken into account, they are $(q+1)$-th eigenvalues.)
By (\ref{<b_N,e<}), one can take $c_\flat = \Vert g^{-1}\Vert _{L_\infty}^{-1}\widetilde{\mathfrak C}_1 $.

For convenience of further references, the following set of parameters is called the \textit{initial data}:
\begin{equation}
\label{data1}
\begin{aligned}
&m,\ d,\ \alpha _0,\ \alpha _1,\ \Vert g\Vert _{L_\infty},\ \Vert g^{-1}\Vert _{L_\infty};
\ \
 \textnormal{the constants}\  \mathfrak{C}_1, \mathfrak{C}_2 \ \textnormal{from} \ \eqref{Garding};
\\
&\textnormal{the parameters of the lattice}\ \Gamma; \ \  \textnormal{the domain} \ \mathcal{O}.
\end{aligned}
\end{equation}
Also, we define the \textit{extended data}:
\begin{equation}
\label{data2}
\begin{aligned}
&m,\ n,\ d,\ q,\ \alpha _0,\ \alpha _1,\ \Vert g\Vert _{L_\infty},\ \Vert g^{-1}\Vert _{L_\infty};
\\
 &\textnormal{the constants}\  \mathfrak{C}_1, \mathfrak{C}_2 \ \textnormal{from} \ \eqref{Garding};
\ \
\textnormal{the constant}\  \widetilde{\mathfrak{C}}_1 \ \textnormal{from} \ \eqref{b(D)u>= in H_perp};
\\
&\textnormal{the parameters of the lattice}\ \Gamma; \ \  \textnormal{the domain} \ \mathcal{O}.
\end{aligned}
\end{equation}

The following result was obtained in Theorems 10.1 and 14.8 of \cite{Su14-2}.

\begin{theorem}[\cite{Su14-2}] Suppose that $\mathcal{O}\subset \mathbb{R}^d$ is a bounded domain with the boundary of class $C^{1,1}$. Suppose that the matrix-valued function $g(\mathbf{x})$ and DO $b(\mathbf{D})$ satisfy the assumptions of Subsection~\textnormal{1.1}.
Suppose that Condition~\textnormal{5.1} is satisfied.
Suppose that the number $\varepsilon _1\in (0,1]$ is subject to Condition~\textnormal{2.1}.

\noindent $1^\circ$. Let $\zeta \in \mathbb{C}\setminus \mathbb{R}_+$, $\zeta =\vert \zeta \vert  e^{i\phi}$, and $\vert \zeta \vert \geqslant 1$.
Then for $0<\varepsilon \leqslant \varepsilon _1$ we have
\begin{equation}
\label{th6.1_1}
\Vert (\mathcal{A}_{N,\varepsilon}-\zeta I)^{-1}-(\mathcal{A}_N^0-\zeta I)^{-1}\Vert _{L_2(\mathcal{O})\rightarrow L_2(\mathcal{O})}\leqslant \mathcal{C}_1  c(\phi)^5\left( \vert \zeta \vert ^{-1/2}\varepsilon +\varepsilon ^2\right).
\end{equation}
Here $c(\phi)$ is defined by \textnormal{(\ref{c(phi)})}. The constant $\mathcal{C}_1$ depends only on the initial data \textnormal{\eqref{data1}}.

\noindent $2^\circ$. Now, let $\zeta \in \mathbb{C}\setminus [c_\flat ,\infty)$, $\zeta \neq 0$, where $c_\flat =\Vert g^{-1}\Vert _{L_\infty}^{-1}\widetilde{\mathfrak{C}}_1$. We put $\zeta -c_\flat =\vert \zeta -c_\flat \vert e^{i\vartheta}$,
$\vartheta \in (0,2\pi)$, and denote
\begin{equation*}
\rho _\flat (\zeta)=\begin{cases}
c(\vartheta)^2\vert \zeta -c_\flat\vert ^{-2}, &\vert \zeta -c_\flat \vert <1,\\
c(\vartheta)^2, &\vert \zeta -c_\flat\vert \geqslant 1.
\end{cases}
\end{equation*}
Then for $0<\varepsilon\leqslant \varepsilon _1$ we have
\begin{equation}
\label{6.2a}
\Vert (\mathcal{A}_{N,\varepsilon}-\zeta I)^{-1}-(\mathcal{A}_N^0-\zeta I)^{-1}\Vert _{L_2(\mathcal{O})\rightarrow L_2(\mathcal{O})}\leqslant \mathcal{C}_2 \rho _\flat (\zeta)\varepsilon .
\end{equation}
The constant $\mathcal{C}_2$ depends only on the data \textnormal{\eqref{data2}}.
\end{theorem}

\subsection*{6.2. Approximation of the resolvent of the operator $\mathcal{A}_{N,\varepsilon}$ in the norm of operators
acting from $L_2(\mathcal{O};{\mathbb C}^n)$ to $H^1(\mathcal{O};{\mathbb C}^n)$}

Approximation of the operator $(\mathcal{A}_{N,\varepsilon}-\zeta I)^{-1}$ with corrector taken into account was obtained in
\cite[Theorems 10.2 and 14.8]{Su14-2}. The corrector is defined similarly to \eqref{K_D(e)}:
\begin{equation}
\label{K_N(e)}
K_N(\varepsilon ;\zeta )= R_\mathcal{O}[\Lambda ^\varepsilon ] S_\varepsilon b(\mathbf{D})P_\mathcal{O}(\mathcal{A}_N^0-\zeta I)^{-1} .
\end{equation}
The operator $K_N(\varepsilon ;\zeta)$ is a continuous mapping of $L_2(\mathcal{O};\mathbb{C}^n)$ to $H^1(\mathcal{O};\mathbb{C}^n)$.

\begin{theorem}[\cite{Su14-2}]
Suppose that the assumptions of Theorem \textnormal{6.1} are satisfied. Let $K_N(\varepsilon ;\zeta )$ be the operator \eqref{K_N(e)},
where $P_\mathcal{O}$ is the extension operator \eqref{P_O H^1, H^2} and $S_\varepsilon$ is the Steklov smoothing operator \eqref{S_e}.
Let $\widetilde{g}$ be the matrix-valued function \eqref{tilde g}.

\noindent $1^\circ$. For $\zeta \in \mathbb{C}\setminus \mathbb{R}_+$, $\vert \zeta \vert \geqslant 1$, and $0<\varepsilon \leqslant \varepsilon _1$
we have
\begin{align}
\label{4.11a}
\begin{split}
\Vert &(\mathcal{A}_{N,\varepsilon}-\zeta I)^{-1}-(\mathcal{A}_N^0-\zeta I)^{-1}-\varepsilon K_N(\varepsilon ;\zeta )\Vert _{L_2(\mathcal{O})\rightarrow H^1(\mathcal{O})}\\
&\leqslant\mathcal{C}_3 c(\phi)^2\vert \zeta \vert ^{-1/4}\varepsilon ^{1/2}+\mathcal{C}_4 c(\phi)^4\varepsilon ,
\end{split}\\
\label{4.11b}
\begin{split}
\Vert &g^\varepsilon b(\mathbf{D})(\mathcal{A}_{N,\varepsilon}-\zeta I)^{-1}-\widetilde{g}^\varepsilon S_\varepsilon b(\mathbf{D})P_\mathcal{O}(\mathcal{A}_N^0-\zeta I)^{-1}\Vert _{L_2(\mathcal{O})\rightarrow L_2(\mathcal{O})}\\
&\leqslant \widetilde{\mathcal{C}}_3 c(\phi)^2\vert \zeta \vert ^{-1/4}\varepsilon ^{1/2}+\widetilde{\mathcal{C}}_4 c(\phi)^4\varepsilon .
\end{split}
\end{align}
The constants $\mathcal{C}_3$, $\mathcal{C}_4$, $\widetilde{\mathcal{C}}_3$, and $\widetilde{\mathcal{C}}_4$ depend only on the initial
data \textnormal{\eqref{data1}}.

\noindent $2^\circ$. Now, let $\zeta \in \mathbb{C}\setminus [c_\flat ,\infty)$, $\zeta \neq 0$.
Let $\mathcal{P}$ be the orthogonal projection of $L_2(\mathcal{O};\mathbb{C}^n)$ onto the subspace $\mathcal{H}(\mathcal{O})$
defined by \eqref{H(O)}. Then for $0<\varepsilon\leqslant\varepsilon _1$ we have
\begin{align}
\label{4.11c}
\begin{split}
\Vert &(\mathcal{A}_{N,\varepsilon}-\zeta I)^{-1}-(\mathcal{A}_N^0-\zeta I)^{-1}-\varepsilon K_N(\varepsilon ;\zeta )\mathcal{P}\Vert _{L_2(\mathcal{O})\rightarrow H^1(\mathcal{O})}
\leqslant\mathcal{C}_5\rho _\flat (\zeta)\varepsilon ^{1/2} ,
\end{split}\\
\label{4.11d}
\begin{split}
\Vert &g^\varepsilon b(\mathbf{D})(\mathcal{A}_{N,\varepsilon}-\zeta I)^{-1}-\widetilde{g}^\varepsilon S_\varepsilon b(\mathbf{D})P_\mathcal{O}(\mathcal{A}_N^0-\zeta I)^{-1}\mathcal{P}\Vert _{L_2(\mathcal{O})\rightarrow L_2(\mathcal{O})}\\
&\leqslant \widetilde{\mathcal{C}}_5\rho _\flat (\zeta)\varepsilon ^{1/2} .
\end{split}
\end{align}
The constants $\mathcal{C}_5$ and $\widetilde{\mathcal{C}}_5$ depend only on the data \textnormal{\eqref{data2}}.
\end{theorem}

\begin{remark}
\textnormal{
Instead of $\Vert g^{-1}\Vert _{L_\infty}^{-1}\widetilde{\mathfrak{C}}_1$,
the number $c_\flat$ in Theorems 6.1 and 6.2
can be taken equal to any number not exceeding $\min \lbrace \mu _{2,\varepsilon}; \mu _2^0\rbrace $. Let $\kappa >0$ be arbitrarily small number. If 
$\varepsilon$ is sufficiently small, we can take $c_\flat =\mu _2^0 -\kappa$.
Then the constants in estimates will depend on $\kappa$.
}
\end{remark}

\subsection*{6.3. The case where $\Lambda\in L_\infty$.} Suppose that Condition 2.5 is satisfied.
Then the smoothing operator in the corrector can be removed. We put
\begin{equation}
\label{K_n^0(e)}
K_N^0(\varepsilon ;\zeta)=[\Lambda ^\varepsilon ]b(\mathbf{D})(\mathcal{A}_N^0-\zeta I)^{-1}.
\end{equation}
The continuity of the operator $K_N^0(\varepsilon ;\zeta)$ from $L_2(\mathcal{O};\mathbb{C}^n)$ to $H^1(\mathcal{O};\mathbb{C}^n)$ is checked similarly to the continuity of the operator \eqref{K_D^0(e)}.

The following result was proved in Theorems 12.1 and 14.9 of \cite{Su14-2}.

\begin{theorem}[\cite{Su14-2}]
Suppose that the assumptions of Theorem~\textnormal{6.1} are satisfied. Suppose that the matrix-valued function $\Lambda (\mathbf{x})$
satisfies Condition~\textnormal{2.5}. Let $K_N^0(\varepsilon ;\zeta )$ be the operator~\eqref{K_n^0(e)}.
Let $\widetilde{g}$ be the matrix-valued function \eqref{tilde g}.

\noindent $1^\circ$. Let $\zeta \in \mathbb{C}\setminus \mathbb{R}_+$, $\vert \zeta \vert \geqslant 1$.
Then for $0<\varepsilon \leqslant \varepsilon _1$ we have
\begin{align*}
\begin{split}
\Vert &(\mathcal{A}_{N,\varepsilon}-\zeta I)^{-1}-(\mathcal{A}_N^0-\zeta I)^{-1}-\varepsilon K_N^0(\varepsilon ;\zeta )\Vert _{L_2(\mathcal{O})\rightarrow H^1(\mathcal{O})}\\
&\leqslant \mathcal{C}_3 c(\phi )^2\vert \zeta \vert ^{-1/4}\varepsilon ^{1/2}+\mathcal{C}_6 c(\phi)^4\varepsilon ,
\end{split}\\
\begin{split}
\Vert &g^\varepsilon b(\mathbf{D})(\mathcal{A}_{N,\varepsilon }-\zeta I)^{-1}-\widetilde{g}^\varepsilon b(\mathbf{D})(\mathcal{A}_N^0-\zeta I)^{-1}\Vert _{L_2(\mathcal{O})\rightarrow L_2(\mathcal{O})}\\
&\leqslant \widetilde{\mathcal{C}}_3 c(\phi )^2\vert \zeta \vert ^{-1/4}\varepsilon ^{1/2}+\widetilde{\mathcal{C}}_6 c(\phi)^4\varepsilon .
\end{split}
\end{align*}
The constants $\mathcal{C}_3$ and $\widetilde{\mathcal{C}}_3$ are the same as in Theorem~\textnormal{6.2}.
The constants $\mathcal{C}_6$ and $\widetilde{\mathcal{C}}_6$ depend only on the initial data
\textnormal{\eqref{data1}} and $\Vert \Lambda\Vert_{L_\infty}$.

\noindent $2^\circ$. Now, let $\zeta \in \mathbb{C}\setminus [c_\flat ,\infty)$, $\zeta \neq 0$.
Then for $0<\varepsilon \leqslant \varepsilon _1$ we have
\begin{align*}
\begin{split}
\Vert (\mathcal{A}_{N,\varepsilon}-\zeta I)^{-1}-(\mathcal{A}_N^0-\zeta I)^{-1}-\varepsilon K_N^0(\varepsilon ;\zeta )\Vert _{L_2(\mathcal{O})\rightarrow H^1(\mathcal{O})}
\leqslant \mathcal{C}_7 \rho _\flat (\zeta )\varepsilon ^{1/2} ,
\end{split}\\
\begin{split}
\Vert g^\varepsilon b(\mathbf{D})(\mathcal{A}_{N,\varepsilon }-\zeta I)^{-1}-\widetilde{g}^\varepsilon b(\mathbf{D})(\mathcal{A}_N^0-\zeta I)^{-1}\Vert _{L_2(\mathcal{O})\rightarrow L_2(\mathcal{O})}
\leqslant \widetilde{\mathcal{C}}_7 \rho _\flat (\zeta )\varepsilon ^{1/2}.
\end{split}
\end{align*}
The constants $\mathcal{C}_7$ and $\widetilde{\mathcal{C}}_7$ depend only on the data
\textnormal{\eqref{data2}} and $\Vert \Lambda\Vert_{L_\infty}$.
\end{theorem}

\subsection*{6.4. Estimates in a strictly interior subdomain.} In a strictly interior subdomain $\mathcal{O}'$ of the domain $\mathcal{O}$
it is possible to obtain error estimates in $H^1$ of sharp order with respect to $\varepsilon$.
The following result was proved in Theorems 12.4 and 14.10 of \cite{Su14-2}.

\begin{theorem}[\cite{Su14-2}]
Suppose that the assumptions of Theorem~\textnormal{6.2} are satisfied.
Let $\mathcal{O}'$ be a strictly interior subdomain of the domain $\mathcal{O}$, and let $\delta =\textnormal{dist}\,\lbrace \mathcal{O}';\partial  \mathcal{O}\rbrace$.

\noindent $1^\circ$. Let $\zeta \in \mathbb{C}\setminus \mathbb{R}_+$ and $\vert \zeta \vert \geqslant 1$. Then for
$0<\varepsilon \leqslant \varepsilon _1$ we have
\begin{align*}
\begin{split}
\Vert &(\mathcal{A}_{N,\varepsilon}-\zeta I)^{-1}-(\mathcal{A}_N^0-\zeta I)^{-1}-\varepsilon K_N(\varepsilon ;\zeta)\Vert _{L_2(\mathcal{O})\rightarrow H^1(\mathcal{O}')}\\
&\leqslant (\mathcal{C}_{8}\delta ^{-1}+\mathcal{C}_{9})c(\phi)^6\varepsilon ,
\end{split}
\\
\begin{split}
\Vert &g^\varepsilon b(\mathbf{D})(\mathcal{A}_{N,\varepsilon}-\zeta I)^{-1}-\widetilde{g}^\varepsilon S_\varepsilon b(\mathbf{D})P_\mathcal{O}(\mathcal{A}_N^0-\zeta I)^{-1}\Vert _{L_2(\mathcal{O})\rightarrow L_2(\mathcal{O}')}\\
&\leqslant (\widetilde{\mathcal{C}}_{8} \delta ^{-1}+\widetilde{\mathcal{C}}_{9})c(\phi)^6\varepsilon .
\end{split}
\end{align*}
The constants $\mathcal{C}_{8}$, $\mathcal{C}_{9}$,  $\widetilde{\mathcal{C}}_{8}$, and
$\widetilde{\mathcal{C}}_{9}$ depend only on the initial data \textnormal{\eqref{data1}}.

\noindent $2^\circ$. Now, let $\zeta \in \mathbb{C}\setminus [c_\flat ,\infty)$, $\zeta \neq 0$.
Denote $\widehat{\rho} _\flat (\zeta):= c(\vartheta )\rho _\flat (\zeta)+c(\vartheta )^{5/2}\rho _\flat (\zeta )^{3/4}$.
For $0<\varepsilon\leqslant \varepsilon_1$ we have
\begin{align*}
\begin{split}
\Vert &(\mathcal{A}_{N,\varepsilon}-\zeta I)^{-1}-(\mathcal{A}_N^0-\zeta I)^{-1}-\varepsilon K_N(\varepsilon ;\zeta )\mathcal{P}\Vert _{L_2(\mathcal{O})\rightarrow H^1(\mathcal{O}')}\\
&\leqslant \left(\mathcal{C}_{10} \delta ^{-1} \widehat{\rho} _\flat (\zeta)
+\mathcal{C}_{11}c(\vartheta )^{1/2}\rho _\flat (\zeta)^{5/4}\right)\varepsilon ,
\end{split}
\\
\begin{split}
\Vert &g^\varepsilon b(\mathbf{D})(\mathcal{A}_{N,\varepsilon}-\zeta I)^{-1}-\widetilde{g}^\varepsilon S_\varepsilon b(\mathbf{D})P_\mathcal{O}(\mathcal{A}_N^0-\zeta I)^{-1}\mathcal{P}\Vert _{L_2(\mathcal{O})\rightarrow L_2(\mathcal{O}')}\\
&\leqslant \left(\widetilde{\mathcal{C}}_{10}\delta ^{-1}\widehat{\rho} _\flat (\zeta)
+\widetilde{\mathcal{C}}_{11}c(\vartheta )^{1/2}\rho _\flat (\zeta)^{5/4}\right)\varepsilon .
\end{split}
\end{align*}
The constants $\mathcal{C}_{10}$, $\mathcal{C}_{11}$, $\widetilde{\mathcal{C}}_{10}$, and $\widetilde{\mathcal{C}}_{11}$ depend only on
the data \textnormal{\eqref{data2}}.
\end{theorem}

The following result was proved  in Theorems 12.5 and 14.11 of \cite{Su14-2}; it concerns the case
where $\Lambda \in L_\infty$.

\begin{theorem}[\cite{Su14-2}]
Suppose that the assumptions of Theorem \textnormal{6.5} are satisfied.
Suppose that the matrix-valued function $\Lambda (\mathbf{x})$ satisfies Condition~\textnormal{2.5}.
Let $K_N^0(\varepsilon ;\zeta)$ be the operator \eqref{K_n^0(e)}.

\noindent
$1^\circ$. Let $\zeta \in \mathbb{C}\setminus \mathbb{R}_+$, $\vert \zeta \vert \geqslant 1$.
For $0<\varepsilon \leqslant \varepsilon _1$ we have
\begin{align*}
\begin{split}
\Vert &(\mathcal{A}_{N,\varepsilon}-\zeta I)^{-1}-(\mathcal{A}_N^0-\zeta I)^{-1}-\varepsilon K_N^0(\varepsilon ;\zeta)\Vert _{L_2(\mathcal{O})\rightarrow H^1(\mathcal{O}')}\\
&\leqslant ( \mathcal{C}_{8}\delta ^{-1}+\check{\mathcal C}_{9})c(\phi)^6\varepsilon ,
\end{split}
\\
\begin{split}
\Vert &g^\varepsilon b(\mathbf{D})(\mathcal{A}_{N,\varepsilon}-\zeta I)^{-1}-\widetilde{g}^\varepsilon b(\mathbf{D})(\mathcal{A}_N^0-\zeta I)^{-1}\Vert _{L_2(\mathcal{O})\rightarrow L_2(\mathcal{O}')}\\
&\leqslant (\widetilde{\mathcal{C}}_{8}\delta ^{-1}+\widehat{\mathcal C}_{9})c(\phi)^6\varepsilon .
\end{split}
\end{align*}
The constants ${\mathcal{C}}_{8}$ and $\widetilde{\mathcal{C}}_{8}$ are the same as in Theorem~\textnormal{6.5}.
The constants  $\check{\mathcal{C}}_{9}$ and $\widehat{\mathcal{C}}_{9}$ depend only on the initial data
\textnormal{\eqref{data1}} and $\Vert \Lambda \Vert _{L_\infty}$.

\noindent
$2^\circ$. Let $\zeta \in \mathbb{C}\setminus [c_\flat ,\infty)$, $\zeta \neq 0$. For $0<\varepsilon\leqslant \varepsilon _1$
we have
\begin{align*}
\begin{split}
\Vert &(\mathcal{A}_{N,\varepsilon}-\zeta I)^{-1}-(\mathcal{A}_N^0-\zeta I)^{-1}-\varepsilon K_N^0(\varepsilon ;\zeta )\Vert _{L_2(\mathcal{O})\rightarrow H^1(\mathcal{O}')}\\
&\leqslant \left(\mathcal{C}_{10}\delta ^{-1}   \widehat{\rho} _\flat (\zeta)
+\check{\mathcal{C}}_{11}c(\vartheta )^{1/2}\rho _\flat (\zeta)^{5/4}\right)\varepsilon ,
\end{split}
\\
\begin{split}
\Vert &g^\varepsilon b(\mathbf{D})(\mathcal{A}_{N,\varepsilon}-\zeta I)^{-1}-\widetilde{g}^\varepsilon  b(\mathbf{D})(\mathcal{A}_N^0-\zeta I)^{-1}\Vert _{L_2(\mathcal{O})\rightarrow L_2(\mathcal{O}')}\\
&\leqslant \left(\widetilde{\mathcal{C}}_{10}\delta ^{-1} \widehat{\rho} _\flat (\zeta)
+\widehat{\mathcal{C}}_{11}c(\vartheta )^{1/2}\rho _\flat (\zeta)^{5/4}\right)\varepsilon .
\end{split}
\end{align*}
The constants $\mathcal{C}_{10}$ and $\widetilde{\mathcal{C}}_{10}$ are the same as in Theorem~\textnormal{6.5}.
The constants $\check{\mathcal{C}}_{11}$ and $\widehat{\mathcal{C}}_{11}$ depend only on the data
\textnormal{\eqref{data2}} and $\Vert \Lambda \Vert _{L_\infty}$.
\end{theorem}

\section*{\S 7. Homogenization of the operator exponential $e^{-\mathcal{A}_{N,\varepsilon}t}$}
\setcounter{section}{7}
\setcounter{equation}{0}
\setcounter{theorem}{0}

\subsection*{7.1. The properties of the operator exponential.}
We start with the following simple statement about estimates for the operators
$e^{-\mathcal{A}_{N,\varepsilon}t} {\mathcal P}$ and $e^{-\mathcal{A}_{N}^0 t} {\mathcal P}$ in various operator norms.

\begin{lemma}
Suppose that the assumptions of Theorem~\textnormal{6.1} are satisfied. Let ${\mathcal P}$ be the orthogonal
projection of $L_2(\mathcal{O};{\mathbb C}^n)$ onto the subspace \textnormal{(\ref{H(O)})}.
Then for $t>0$ and $\varepsilon >0$ we have
\begin{align}
\label{7.1}
&\Vert e^{-\mathcal{A}_{N,\varepsilon}t} {\mathcal P} \Vert _{L_2(\mathcal{O})\rightarrow L_2(\mathcal{O})}
\leqslant e^{-c_\flat t},
\\
\label{7.2}
&\Vert e^{-\mathcal{A}_{N,\varepsilon}t} {\mathcal P}\Vert _{L_2(\mathcal{O})\rightarrow H^1(\mathcal{O})}
\leqslant c_\flat^{-1/2} t^{-1/2} e^{-c_\flat t/2},
\\
\label{7.3}
&\Vert e^{-\mathcal{A}_{N}^0 t} {\mathcal P} \Vert _{L_2(\mathcal{O})\rightarrow L_2(\mathcal{O})}
\leqslant e^{-c_\flat t},
\\
\label{7.4}
&\Vert e^{-\mathcal{A}_{N}^0 t} {\mathcal P} \Vert _{L_2(\mathcal{O})\rightarrow H^1(\mathcal{O})}
\leqslant  c_\flat^{-1/2} t^{-1/2} e^{-c_\flat t/2},
\\
\label{7.5}
&\Vert e^{-\mathcal{A}_{N}^0 t} {\mathcal P} \Vert _{L_2(\mathcal{O})\rightarrow H^2(\mathcal{O})}
\leqslant \check{c}^\circ \,t^{-1} e^{-c_\flat t/2},
\end{align}
where $\check{c}^\circ= {c}^\circ (1+ c_\flat^{-1})$.
\end{lemma}

\begin{proof}
Estimates \eqref{7.1} and \eqref{7.3} follow directly from \eqref{<b_N,e<}.

By \eqref{<b_N,e<},
\begin{equation}
\label{7.6}
\Vert e^{-\mathcal{A}_{N,\varepsilon}t}  {\mathcal P} \Vert _{L_2(\mathcal{O})\rightarrow H^1(\mathcal{O})}\leqslant
c_\flat^{-1/2}
\Vert \mathcal{A}_{N,\varepsilon}^{1/2} e^{-\mathcal{A}_{N,\varepsilon}t} {\mathcal P} \Vert _{L_2(\mathcal{O})\rightarrow L_2(\mathcal{O})}.
\end{equation}
Similarly to \eqref{3.7},
\begin{equation}
\label{7.7}
\Vert \mathcal{A}_{N,\varepsilon}^{1/2} e^{-\mathcal{A}_{N,\varepsilon}t} {\mathcal P} \Vert _{L_2(\mathcal{O})\rightarrow L_2(\mathcal{O})}
\leqslant \sup_{\mu \geqslant c_\flat} \mu^{1/2} e^{-\mu t}
\leqslant  t^{-1/2}e^{-c_{\flat}t/2}.
\end{equation}
Relations (\ref{7.6}) and (\ref{7.7}) imply \eqref{7.2}.
Similarly, from the inequality
\begin{equation}
\label{7.8}
\Vert (\mathcal{A}_{N}^0)^{1/2} e^{-\mathcal{A}_{N}^0 t} {\mathcal P}\Vert _{L_2(\mathcal{O})\rightarrow L_2(\mathcal{O})}
\leqslant t^{-1/2}e^{-c_{\flat}t/2}
\end{equation}
and \eqref{<b_N,e<} we deduce \eqref{7.4}.

Next, by \eqref{5.6a},
\begin{equation*}
\|(\mathcal{A}_{N}^0)^{-1}{\mathcal P}\Vert _{L_2(\mathcal{O})\rightarrow H^2(\mathcal{O})}
\leqslant c^\circ \sup_{\mu \geqslant c_\flat} (\mu +1) \mu^{-1}
\leqslant c^\circ (1+ c_\flat^{-1}) = \check{c}^\circ.
\end{equation*}
Similarly to \eqref{3.8a}, this yields
\begin{equation*}
\Vert  e^{-\mathcal{A}_{N}^0 t} {\mathcal P} \Vert _{L_2(\mathcal{O})\rightarrow H^2(\mathcal{O})}\leqslant
\check{c}^\circ
\Vert \mathcal{A}_{N}^0 e^{-\mathcal{A}_{N}^0 t} {\mathcal P} \Vert _{L_2(\mathcal{O})\rightarrow L_2(\mathcal{O})}
\leqslant \check{c}^\circ\, t^{-1}e^{-c_{\flat}t/2},
\end{equation*}
which proves (\ref{7.5}).
\end{proof}

\subsection*{7.2. Approximation of the operator $e^{-\mathcal{A}_{N,\varepsilon}t}$ in the $L_2(\mathcal{O};\mathbb{C}^n)$-operator norm.} Now, we prove the following result.

\begin{theorem}
Suppose that the assumptions of Theorem \textnormal{6.1} are satisfied.
Then for \hbox{$0<\varepsilon \leqslant \varepsilon _1$} we have
\begin{equation}
\label{Th exp_N L2->L2}
\Vert e^{-\mathcal{A}_{N,\varepsilon}t}-e^{-\mathcal{A}_N^0t}\Vert _{L_2(\mathcal{O})\rightarrow L_2(\mathcal{O})}\leqslant \mathcal{C}_{12} \varepsilon (t+\varepsilon ^2)^{-1/2}e^{- c_\flat t/2},\quad t\geqslant 0.
\end{equation}
The constant $\mathcal{C}_{12}$ depends only on the data \textnormal{\eqref{data2}}.
\end{theorem}

\begin{proof}
Since $I-\mathcal{P}$ is the orthogonal projection of $L_2(\mathcal{O};\mathbb{C}^n)$ onto $Z$, by (\ref{Ker=Z}), we have
\begin{equation*}
e^{-\mathcal{A}_{N,\varepsilon}t}(I-\mathcal{P})=e^{-\mathcal{A}_N^0 t}(I-\mathcal{P}) =I-\mathcal{P}.
\end{equation*}
Hence,
\begin{equation}
\label{exp-exp=(exp-exp)P}
e^{-\mathcal{A}_{N,\varepsilon}t}-e^{-\mathcal{A}_N^0t}=\left(e^{-\mathcal{A}_{N,\varepsilon}t}-e^{-\mathcal{A}_N^0t}\right)\mathcal{P}.
\end{equation}
Together with \eqref{<b_N,e<} this implies
\begin{equation}
\label{tozd exp Neumann}
e^{-\mathcal{A}_{N,\varepsilon}t}-e^{-\mathcal{A}_N^0t} =-\frac{1}{2\pi i}\int _{\widetilde{\gamma}} e^{-\zeta t}\left( (\mathcal{A}_{N,\varepsilon}-\zeta I)^{-1}-(\mathcal{A}_N^0-\zeta I)^{-1}\right)\,d\zeta .
\end{equation}
Here $\widetilde{\gamma}\subset \mathbb{C}$ is the positively oriented contour consisting of two rays:
\begin{equation*}
\begin{split}
\widetilde{\gamma}&=\lbrace \zeta\in \mathbb{C} : \textnormal{Im}\,\zeta\geqslant 0,\, \textnormal{Re}\,\zeta=\textnormal{Im}\,\zeta +c_\flat /2\rbrace \\
&\cup \lbrace \zeta\in \mathbb{C} : \textnormal{Im}\,\zeta <0,\, \textnormal{Re}\,\zeta =-\textnormal{Im}\,\zeta +c_\flat /2\rbrace .
\end{split}
\end{equation*}

Denote $\check{c}_\flat :=\max\lbrace 1;\sqrt{5} c_\flat /2\rbrace$.
For $\zeta \in \widetilde{\gamma}$ and $|\zeta| \leqslant \check{c}_\flat$ we apply estimate
(\ref{6.2a}). For $\zeta \in \widetilde{\gamma}$ and $|\zeta| > \check{c}_\flat$ we use (\ref{th6.1_1}).
Similarly to the proof of Theorem 3.2, this implies (\ref{Th exp_N L2->L2}).
\end{proof}

\subsection*{7.3. Approximation of the operator $e^{-\mathcal{A}_{N,\varepsilon}t}$ in the $(L_2 \to H^1)$-norm.}
We introduce the operator
\begin{equation}
\label{K_N(t;eps)}
\mathcal{K}_N(t;\eps)=R_\mathcal{O}[\Lambda ^\varepsilon]S_\varepsilon b(\mathbf{D})P_\mathcal{O}e^{-\mathcal{A}_N^0 t}\mathcal{P}.
\end{equation}
The continuity of the operator $\mathcal{K}_N(t;\eps)$ from $L_2(\mathcal{O};{\mathbb C}^n)$ to $H^1(\mathcal{O};{\mathbb C}^n)$
for $t>0$ is checked similarly to the continuity of the operator (\ref{K_D(t,e)}).

Using Theorem 6.2, we obtain the following result.

\begin{theorem}
Suppose that the assumptions of Theorem \textnormal{6.2} are satisfied.
Let $\mathcal{K}_N(t;\eps)$ be the operator \textnormal{(\ref{K_N(t;eps)})}.
Then for $0<\varepsilon \leqslant \varepsilon _1$ and $t>0$ we have
\begin{align}
\label{Th exn_N L2->H1}
\begin{split}
\Vert &e^{-\mathcal{A}_{N,\varepsilon}t}-e^{-\mathcal{A}_N^0 t}-\varepsilon \mathcal{K}_N(t;\eps)\Vert _{L_2(\mathcal{O})\rightarrow H^1(\mathcal{O})}\\
&\leqslant \mathcal{C}_{13}(\varepsilon ^{1/2}t^{-3/4}+\varepsilon t^{-1})e^{- c_\flat t/2},
\end{split}\\
\label{Th potoki_N}
\begin{split}
\Vert &g^\varepsilon b(\mathbf{D})e^{-\mathcal{A}_{N,\varepsilon}t}-\widetilde{g}^\varepsilon S_\varepsilon b(\mathbf{D})P_\mathcal{O}e^{-\mathcal{A}_N^0 t}\mathcal{P}\Vert _{L_2(\mathcal{O})\rightarrow L_2(\mathcal{O})}\\
&\leqslant \widetilde{\mathcal{C}}_{13}(\varepsilon ^{1/2}t^{-3/4}+\varepsilon t^{-1})e^{- c_\flat t/2}.
\end{split}
\end{align}
The constants $\mathcal{C}_{13}$ and $\widetilde{\mathcal{C}}_{13}$
depend only on the data \textnormal{\eqref{data2}}.
\end{theorem}

\begin{proof}
Similarly to \eqref{tozd exp Neumann}, we have
\begin{equation}
\label{7.6a}
\begin{split}
&e^{-\mathcal{A}_{N,\varepsilon}t}-e^{-\mathcal{A}_N^0t} - \varepsilon \mathcal{K}_N(t;\eps)
\\
&=-\frac{1}{2\pi i}\int _{\widetilde{\gamma}} e^{-\zeta t}\left( (\mathcal{A}_{N,\varepsilon}-\zeta I)^{-1}-(\mathcal{A}_N^0-\zeta I)^{-1}
- \varepsilon {K}_N(\eps;\zeta) {\mathcal P}\right)\,d\zeta .
\end{split}
\end{equation}
Here ${K}_N(\eps;\zeta)$ is the operator (\ref{K_N(e)}). Note that, by (\ref{Ker=Z}),
\begin{equation}
\label{7.7b}
\begin{split}
&(\mathcal{A}_{N,\varepsilon}-\zeta I)^{-1}-(\mathcal{A}_N^0-\zeta I)^{-1}
- \varepsilon {K}_N(\eps;\zeta) {\mathcal P}
\\
&=
\left((\mathcal{A}_{N,\varepsilon}-\zeta I)^{-1}-(\mathcal{A}_N^0-\zeta I)^{-1}
- \varepsilon {K}_N(\eps;\zeta)\right) {\mathcal P}.
\end{split}
\end{equation}

For $\zeta \in \widetilde{\gamma}$ and $\vert \zeta \vert \leqslant \check{c}_\flat $
we apply estimate \eqref{4.11c}. For $\zeta \in \widetilde{\gamma}$ and $\vert \zeta \vert >  \check{c}_\flat $
we use \eqref{4.11a} and (\ref{7.7b}). Similarly to the proof of Theorem~3.3 this implies (\ref{Th exn_N L2->H1}).

Let us discuss the proof of estimate \eqref{Th potoki_N}. Since $I-\mathcal{P}$ is the orthogonal projection of
$L_2(\mathcal{O};\mathbb{C}^n)$ onto $Z=\textnormal{Ker}\,b(\mathbf{D})=\textnormal{Ker}\,\mathcal{A}_{N,\varepsilon}$, then
\begin{equation}
\label{potok_N-potot_N*P}
g^\varepsilon b(\mathbf{D})e^{-\mathcal{A}_{N,\varepsilon}t}=g^\varepsilon b(\mathbf{D})e^{-\mathcal{A}_{N,\varepsilon}t}\mathcal{P}.
\end{equation}
Hence,
\begin{equation}
\label{7.8_0}
\begin{split}
&g^\varepsilon b(\mathbf{D})e^{-\mathcal{A}_{N,\varepsilon}t}-\widetilde{g}^\varepsilon S_\varepsilon b(\mathbf{D})P_\mathcal{O} e^{-\mathcal{A}_N^0 t}\mathcal{P}\\
&=-\frac{1}{2\pi i}\int _{\widetilde{\gamma}}e^{-\zeta t}\left(g^\varepsilon b(\mathbf{D})(\mathcal{A}_{N,\varepsilon}-\zeta I)^{-1}-\widetilde{g}^\varepsilon S_\varepsilon b(\mathbf{D})P_\mathcal{O}(\mathcal{A}_N^0-\zeta I)^{-1}\right) \mathcal{P}\,d\zeta .
\end{split}
\end{equation}
Besides, we have
\begin{equation}
\label{7.8a}
\begin{split}
&g^\varepsilon b(\mathbf{D})(\mathcal{A}_{N,\varepsilon}-\zeta I)^{-1}-\widetilde{g}^\varepsilon S_\varepsilon b(\mathbf{D})P_\mathcal{O}(\mathcal{A}_N^0-\zeta I)^{-1}\mathcal{P}
\\
&=
\left(g^\varepsilon b(\mathbf{D})(\mathcal{A}_{N,\varepsilon}-\zeta I)^{-1}-\widetilde{g}^\varepsilon S_\varepsilon b(\mathbf{D})P_\mathcal{O}(\mathcal{A}_N^0-\zeta I)^{-1}\right)\mathcal{P}.
\end{split}
\end{equation}
For $\zeta \in \widetilde{\gamma}$ and $\vert \zeta \vert \leqslant \check{c}_\flat $ we apply \eqref{4.11d} and \eqref{7.8a}.
For $\zeta \in \widetilde{\gamma}$ and $\vert \zeta \vert > \check{c}_\flat $ we rely on \eqref{4.11b}.
This leads to  \eqref{Th potoki_N}.
\end{proof}

\begin{remark}\textnormal{Taking Remark 6.3 into account, it is possible to prove estimates of the form
\eqref{Th exp_N L2->L2}, \eqref{Th exn_N L2->H1}, and \eqref{Th potoki_N} with $e^{- c_\flat t/2}$ replaced by
$e^{- (\mu_2^0 - \kappa) t}$, where $\kappa >0$ is arbitrarily small number.
Indeed, for sufficiently small $\varepsilon$ one can transform the contour of integration so that
it will intersect the real axis at the point $\mu _2^0-\kappa$.
Herewith, all constants in estimates of the form \eqref{Th exp_N L2->L2}, \eqref{Th exn_N L2->H1}, and \eqref{Th potoki_N}
will depend on $\kappa$.
}
\end{remark}

\subsection*{7.4. Estimates for small $t$.}
Note that for $0<t<\varepsilon ^2$ instead of Theorem 7.3 it is better
to use the following simple statement (which is valid, however, for all $t>0$).

\begin{proposition}
Suppose that the assumptions of Theorem \textnormal{6.1} are satisfied. Then for $t>0$ and $\varepsilon > 0$ we have
\begin{align}
\label{proposition 5.4-1}
&\Vert e^{-\mathcal{A}_{N,\varepsilon}t}-e^{-\mathcal{A}_N^0 t}\Vert _{L_2(\mathcal{O})\rightarrow H^1(\mathcal{O})}\leqslant \mathcal{C}_{14}t^{-1/2}e^{-c_\flat t/2},\\
\label{proposition 5.4-2}
&\Vert g^\varepsilon b(\mathbf{D})e^{-\mathcal{A}_{N,\varepsilon}t}\Vert _{L_2(\mathcal{O})\rightarrow L_2(\mathcal{O})}\leqslant \widetilde{\mathcal{C}}_{14}t^{-1/2}e^{-c_\flat t/2},\\
\label{proposition 5.4-3}
&\Vert g^0 b(\mathbf{D})e^{-\mathcal{A}_N^0 t}\Vert _{L_2(\mathcal{O})\rightarrow L_2(\mathcal{O})}\leqslant \widetilde{\mathcal{C}}_{14}t^{-1/2}e^{- c_\flat t/2} .
\end{align}
Here $\mathcal{C}_{14}=2\widetilde{\mathfrak C}_1^{\,- 1/2}\Vert g^{-1}\Vert ^{1/2}_{L_\infty}$ and $\widetilde{\mathcal{C}}_{14}=\Vert g\Vert ^{1/2}_{L_\infty}$.
\end{proposition}

\begin{proof}
Estimate \eqref{proposition 5.4-1} follows from \eqref{7.2}, \eqref{7.4},
\eqref{exp-exp=(exp-exp)P}, and expression for the constant $c_\flat$.

By \eqref{potok_N-potot_N*P},
\begin{equation*}
\Vert g^\varepsilon b(\mathbf{D})e^{-\mathcal{A}_{N,\varepsilon}t}\Vert _{L_2(\mathcal{O})\rightarrow L_2(\mathcal{O})} \leqslant \Vert g\Vert ^{1/2}_{L_\infty}\Vert \mathcal{A}_{N,\varepsilon}^{1/2}e^{-\mathcal{A}_{N,\varepsilon}t}\mathcal{P}\Vert _{L_2(\mathcal{O})\rightarrow L_2(\mathcal{O})}.
\end{equation*}
Together with \eqref{7.7} this implies  (\ref{proposition 5.4-2}).
In a similar way, we check estimate (\ref{proposition 5.4-3}), taking \eqref{|g^0|, |g^0^_1|<=} and  (\ref{7.8}) into account.
\end{proof}

\subsection*{7.5. The case where $\Lambda \in L_\infty$.}
Suppose that the matrix-valued function $\Lambda (\mathbf{x})$ satisfies Condition 2.5.
Then the operator
\begin{equation}
\label{K_N^0(t;e)}
\mathcal{K}_N^0(t;\eps)=[\Lambda ^\varepsilon ] b(\mathbf{D})e^{-\mathcal{A}_N^0 t}
\end{equation}
is continuous from $L_2(\mathcal{O};\mathbb{C}^n)$ to $H^1(\mathcal{O};\mathbb{C}^n)$ for $t>0$.

\begin{theorem}
Suppose that the assumptions of Theorem \textnormal{6.1} are satisfied.
Suppose that Condition \textnormal{2.5} is satisfied.
Let $\mathcal{K}_N^0(t;\eps)$ be the operator \textnormal{(\ref{K_N^0(t;e)})}.
Let $\widetilde{g}$ be the matrix-valued function \eqref{tilde g}.
Then for $0<\varepsilon \leqslant \varepsilon _1$ and $t>0$ we have
\begin{align*}
&\Vert e^{-\mathcal{A}_{N,\varepsilon}t}-e^{-\mathcal{A}_N^0 t}-\varepsilon \mathcal{K}_N^0 (t;\eps)\Vert _{L_2(\mathcal{O})\rightarrow H^1(\mathcal{O})}\leqslant \mathcal{C}_{15}(\varepsilon ^{1/2}t^{-3/4}+\varepsilon t^{-1})e^{-c_\flat t/2},\\
&\Vert g^\varepsilon b(\mathbf{D})e^{-\mathcal{A}_{N,\varepsilon}t}-\widetilde{g}^\varepsilon b(\mathbf{D})e^{-\mathcal{A}_N^0 t}\Vert _{L_2(\mathcal{O})\rightarrow L_2(\mathcal{O})}\leqslant \widetilde{\mathcal{C}}_{15}(\varepsilon ^{1/2}t^{-3/4}+\varepsilon t^{-1})e^{-c_\flat t/2}.
\end{align*}
The constants $\mathcal{C}_{15}$ and $\widetilde{\mathcal{C}}_{15}$
depend only on the data \textnormal{\eqref{data2}} and $\Vert \Lambda \Vert _{L_\infty}$.
\end{theorem}

\begin{proof}
The proof is based on application of Theorem 6.4 and repeats the proof of Theorem~7.3 with some simplifications.
\end{proof}

\subsection*{7.6. The case of the smooth boundary.}
It is also possible to remove the smoothing operator $S_\eps$ in the corrector under the assumption of
additional smoothness of the boundary. In this subsection we consider the case where $d \geqslant 3$,
because for $d \leqslant 2$ Theorem 7.6 is applicable (see Proposition~2.7($1^\circ$)).

\begin{lemma}
Suppose that $k \geqslant 2$ is integer. Let ${\mathcal O} \subset {\mathbb R}^d$ be a bounded domain with the boundary
$\partial {\mathcal O}$ of class $C^{k-1,1}$.
Then for $t>0$ the operator $e^{-\mathcal{A}^0_{N}t}{\mathcal P}$ is continuous from $L_2({\mathcal O};{\mathbb C}^n)$
to $H^s({\mathcal O};{\mathbb C}^n)$, $0\leqslant s \leqslant k$, and the following estimate is valid:
\begin{equation}
\label{7.*1}
\|e^{-\mathcal{A}^0_{N}t} {\mathcal P} \|_{L_2({\mathcal O}) \to H^s({\mathcal O})}
\leqslant \widehat{\mathscr{C}}_s  t^{-s/2} e^{-c_{\flat}t/2} ,\quad t>0.
\end{equation}
The constant $\widehat{\mathscr{C}}_s$ depends only on
$s$,  $\alpha _0$, $\alpha _1$, $\Vert g\Vert _{L_\infty}$, $\Vert g^{-1}\Vert _{L_\infty}$,
the constants ${\mathfrak C}_1$ and ${\mathfrak C}_2$ from the inequality \textnormal{(\ref{Garding})},
the constant $\widetilde{\mathfrak C}_1$ from \textnormal{(\ref{b(D)u>= in H_perp})}, and the domain $\mathcal{O}$.
\end{lemma}

\begin{proof}
The proof is similar to that of Lemma 3.7.
It suffices to check estimate \eqref{7.*1} for integer $s\in [0,k]$.
For $s=0,1,2$ estimate \eqref{7.*1} has been already proved (see Lemma 7.1).

So, let $s$ be integer and let $2 \leqslant s \leqslant k$.
By  [McL, Chapter 4], we conclude that the operator $({\mathcal A}_N^0)^{-s/2} {\mathcal P}$
is continuous from $L_2({\mathcal O};{\mathbb C}^n)$ to
$H^{s}({\mathcal O};{\mathbb C}^n)$. Herewith,
\begin{equation}
\label{7.*2}
\| (\mathcal{A}^0_{N})^{-s/2} {\mathcal P} \|_{L_2({\mathcal O}) \to H^s({\mathcal O})} \leqslant \widetilde{\mathscr{C}}_s,
\end{equation}
where the constant $\widetilde{\mathscr{C}}_s$ depends only on
$s$,  $\alpha _0$, $\alpha _1$, $\Vert g\Vert _{L_\infty}$, $\Vert g^{-1}\Vert _{L_\infty}$,
the constants ${\mathfrak C}_1$ and ${\mathfrak C}_2$ from the inequality \textnormal{(\ref{Garding})},
the constant $\widetilde{\mathfrak C}_1$ from \textnormal{(\ref{b(D)u>= in H_perp})}, and the domain $\mathcal{O}$.
Similarly to \eqref{3.32a}, this implies
\eqref{7.*1} with $\widehat{\mathscr{C}}_s=\widetilde{\mathscr{C}}_s  (s/e)^{s/2}$.
\end{proof}

Using the properties of the matrix-valued function $\Lambda({\mathbf x})$ and the operator $S_\eps$,
it is possible to deduce the estimate for the difference of the correctors
\eqref{K_N(t;eps)} and \eqref{K_N^0(t;e)} from Lemma 7.7 under the assumption of additional smoothness
of the boundary.

\begin{lemma}
Let $d \geqslant 3$.
Suppose that ${\mathcal O} \subset {\mathbb R}^d$ is a bounded domain with the boundary
$\partial {\mathcal O}$ of class $C^{d/2,1}$ if $d$ is even and of class $C^{(d+1)/2,1}$ if $d$ is odd.
Let $\mathcal{K}_N(t;\varepsilon)$ be the operator \eqref{K_N(t;eps)}, and let
$\mathcal{K}_N^0(t;\varepsilon)$ be the operator \eqref{K_N^0(t;e)}.
Then
\begin{equation}
\label{7.*3}
\| \mathcal{K}_N(t;\varepsilon) - \mathcal{K}_N^0(t;\varepsilon)
\|_{L_2({\mathcal O}) \to H^1({\mathcal O})}
\leqslant \check{\mathscr C}_d   (t^{-1} +  t^{-d/4 -1/2}) e^{-c_{\flat}t/2},\quad t>0.
\end{equation}
The constant $\check{\mathscr C}_d$
depends only on $d$,  $m$,
$\alpha _0$, $\alpha _1$, $\Vert g\Vert _{L_\infty}$, $\Vert g^{-1}\Vert _{L_\infty}$,
the constants ${\mathfrak C}_1$ and ${\mathfrak C}_2$ from the inequality \textnormal{(\ref{Garding})},
the estimate $\widetilde{\mathfrak C}_1$ from \textnormal{(\ref{b(D)u>= in H_perp})}, and the domain $\mathcal{O}$.
\end{lemma}

Lemma 7.8 and Theorem 7.3 imply the following result.

\begin{theorem}
Suppose that the assumptions of Theorem \textnormal{6.1} are satisfied,
and let $d \geqslant 3$. Suppose that the domain ${\mathcal O}$ satisfies the assumptions of Lemma~\textnormal{7.8}.
Let $\mathcal{K}_N^0(t;\varepsilon)$ be the operator \textnormal{(\ref{K_N^0(t;e)})}.
Let $\widetilde{g}$ be the matrix-valued function \eqref{tilde g}.
Then for \hbox{$0<\varepsilon \leqslant \varepsilon _1$} and $t>0$ we have
\begin{equation}
\label{7.*4}
\begin{split}
\Vert &e^{-\mathcal{A}_{N,\varepsilon}t}-e^{-\mathcal{A}_N^0t}-\varepsilon \mathcal{K}_N^0(t;\varepsilon)\Vert _{L_2(\mathcal{O})\rightarrow H^1(\mathcal{O})}
\\
&\leqslant
{\mathscr C}^*_{d}(\varepsilon ^{1/2}t^{-3/4}+\varepsilon t^{-1} + \varepsilon t^{-d/4 -1/2})e^{-c_{\flat}t/2},
\end{split}
\end{equation}
\begin{equation}
\label{7.*5}
\begin{split}
\Vert &g^\varepsilon b(\mathbf{D})e^{-\mathcal{A}_{N,\varepsilon}t}-\widetilde{g}^\varepsilon b(\mathbf{D})e^{-\mathcal{A}_N^0t}\Vert _{L_2(\mathcal{O})\rightarrow L_2(\mathcal{O})}
\\
&\leqslant
\widetilde{\mathscr C}^*_{d}(\varepsilon ^{1/2}t^{-3/4}+\varepsilon t^{-1} + \varepsilon t^{-d/4 -1/2})
e^{-c_{\flat}t/2}.
\end{split}
\end{equation}
The constants ${\mathscr C}^*_{d}$ and $\widetilde{\mathscr C}^*_{d}$ depend only on the data  \textnormal{\eqref{data2}}.
\end{theorem}

The proofs of Lemma 7.8 and Theorem 7.9 are postponed until Appendix (see \S 9).
As in the case of the Dirichlet condition, it is convenient to apply Theorem 7.9, if $t$ is separated from zero.

\begin{remark}
\textnormal{In Lemma 7.8, instead of the smoothness condition on $\partial \mathcal{O}$, one could impose an implicit
condition on the domain:  suppose that a bounded domain $\mathcal{O}$ with Lipschitz boundary is such that
estimate \eqref{7.*2} with $s=d/2+1$ is valid.
For such domain all the assertions of Lemma 7.8 and Theorem 7.9 are true.}
\end{remark}

\subsection*{7.7. Special cases.} Special cases are distinguished
by the arguments similar to the arguments of Subsection~3.7.
Theorems 7.3, 7.6 and Propositions 1.2, 1.3, 2.7($3^\circ$)  imply the following result.

\begin{proposition}
Suppose that the assumptions of Theorem \textnormal{7.3} are satisfied.

\noindent
$1^\circ$. Let $g^0=\overline{g}$, i.~e., relations \textnormal{(\ref{overline-g})} are satisfied.
Then $\mathcal{K}_N(t;\eps)=0$, and for \hbox{$0<\varepsilon \leqslant \varepsilon _1$} we have
\begin{equation*}
\Vert e^{-\mathcal{A}_{N,\varepsilon}t}-e^{-\mathcal{A}_N^0 t}\Vert _{L_2(\mathcal{O})\rightarrow H^1(\mathcal{O})}\leqslant \mathcal{C}_{13}(\varepsilon ^{1/2}t^{-3/4}+\varepsilon t^{-1})e^{-c_\flat t/2},\quad t>0.
\end{equation*}

\noindent
$2^\circ$. Let $g^0=\underline{g}$, i.~e., relations \textnormal{(\ref{underline-g})} are satisfied.
Then for \hbox{$0<\varepsilon\leqslant \varepsilon _1$} and $t>0$ we have
\begin{equation*}
\Vert g^\varepsilon b(\mathbf{D})e^{-\mathcal{A}_{N,\varepsilon}t}-g^0 b(\mathbf{D})e^{-\mathcal{A}_N^0t}\Vert _{L_2(\mathcal{O})\rightarrow L_2(\mathcal{O})}\leqslant \widetilde{\mathcal{C}}_{15}(\varepsilon ^{1/2}t^{-3/4}+\varepsilon t^{-1})
e^{-c_\flat t/2}.
\end{equation*}
\end{proposition}

\subsection*{7.8. Estimates in a strictly interior subdomain.}
In a strictly interior subdomain $\mathcal{O}'$ of the domain $\mathcal{O}$,
using Theorem 6.5, we obtain estimates of sharp order $O(\varepsilon)$ (for a fixed $t>0$).

\begin{theorem}
Suppose that the assumptions of Theorem \textnormal{7.3} are satisfied. Let $\mathcal{O}'$ be a strictly interior subdomain of the domain
$\mathcal{O}$, and let $\delta =\textnormal{dist}\,\lbrace \mathcal{O}';\partial \mathcal{O}\rbrace$.
Then for $0<\varepsilon\leqslant\varepsilon _1$ and $t>0$ we have
\begin{align}
\label{7.29a}
\begin{split}
\Vert
&e^{-\mathcal{A}_{N,\varepsilon}t}-e^{-\mathcal{A}_N^0 t}-\varepsilon\mathcal{K}_N(t;\varepsilon)\Vert _{L_2(\mathcal{O})\rightarrow H^1(\mathcal{O}')}
\\
&\leqslant (\mathcal{C}_{16}\delta ^{-1}+\mathcal{C}_{17})\varepsilon t^{-1}e^{-c_\flat t/2},
\end{split}
\\
\begin{split}
\Vert
&g^\varepsilon b(\mathbf{D})e^{-\mathcal{A}_{N,\varepsilon}t}-\widetilde{g}^\varepsilon S_\varepsilon b(\mathbf{D})P_\mathcal{O}e^{-\mathcal{A}_N^0 t}\mathcal{P}\Vert _{L_2(\mathcal{O})\rightarrow L_2(\mathcal{O}')}
\\
&\leqslant (\widetilde{\mathcal{C}}_{16}\delta ^{-1}+\widetilde{\mathcal{C}}_{17})\varepsilon t^{-1}e^{-c_\flat t/2}.
\end{split}\nonumber
\end{align}
The constants $\mathcal{C}_{16}$, $\mathcal{C}_{17}$, $\widetilde{\mathcal{C}}_{16}$, and $\widetilde{\mathcal{C}}_{17}$
depend only on the data \textnormal{\eqref{data2}}.
\end{theorem}

\begin{proof}
The proof is similar to that of Theorem 7.3 and relies on Theorem 6.5 and relations \eqref{7.6a}, \eqref{7.7b}, \eqref{7.8_0},
and \eqref{7.8a}.
\end{proof}

In the case where $\Lambda \in L_\infty$, Theorem 6.6 implies the following result.

\begin{theorem}
Suppose that the assumptions of Theorem \textnormal{7.12} are satisfied.
Suppose that the matrix-valued function $\Lambda (\mathbf{x})$ satisfies Condition \textnormal{2.5}.
Let $\mathcal{K}_N^0(t;\varepsilon)$ be the operator \eqref{K_N^0(t;e)}. Then for $0<\varepsilon \leqslant \varepsilon_1$ and $t>0$
we have
\begin{align*}
\begin{split}
&\Vert
e^{-\mathcal{A}_{N,\varepsilon}t}-e^{-\mathcal{A}_N^0 t}-\varepsilon\mathcal{K}_N^0(t;\varepsilon)\Vert _{L_2(\mathcal{O})\rightarrow H^1(\mathcal{O}')}
\leqslant (\mathcal{C}_{16}\delta ^{-1}+\check{\mathcal{C}}_{17})\varepsilon t^{-1}e^{-c_\flat t/2},
\end{split}
\\
\begin{split}
&\Vert g^\varepsilon b(\mathbf{D})e^{-\mathcal{A}_{N,\varepsilon}t}-\widetilde{g}^\varepsilon  b(\mathbf{D})e^{-\mathcal{A}_N^0 t}\Vert _{L_2(\mathcal{O})\rightarrow L_2(\mathcal{O}')}
\leqslant (\widetilde{\mathcal{C}}_{16}\delta ^{-1}+\widehat{\mathcal{C}}_{17})\varepsilon t^{-1}e^{-c_\flat t/2}.
\end{split}
\end{align*}
The constants $\mathcal{C}_{16}$ and $\widetilde{\mathcal{C}}_{16}$ are the same as in Theorem
\textnormal{7.12}. The constants $\check{\mathcal{C}}_{17}$ and $\widehat{\mathcal{C}}_{17}$ depend only on the data
\textnormal{\eqref{data2}} and $\|\Lambda\|_{L_\infty}$.
\end{theorem}

It is possible to remove the smoothing operator $S_\eps$ from the corrector in estimates of Theorem 7.12 without Condition 2.5.
We consider the case $d \geqslant 3$ (in the opposite case, by Proposition 2.7($1^\circ$), Theorem 7.13 is applicable).
For $t>0$ the operator $e^{-\mathcal{A}^0_{N}t}{\mathcal P}$ is continuous from $L_2({\mathcal O};{\mathbb C}^n)$
to $H^2({\mathcal O};{\mathbb C}^n)$, and estimate \eqref{7.5} holds.
Moreover, the following property of ,,interior regularity'' is true:
for $t>0$ the operator $e^{-\mathcal{A}^0_{N}t} {\mathcal P}$ is continuous from $L_2({\mathcal O};{\mathbb C}^n)$
to $H^s({\mathcal O}';{\mathbb C}^n)$ for any integer $s \geqslant 3$. We have
\begin{equation}
\label{apriori_N}
\|e^{-\mathcal{A}^0_{N}t} {\mathcal P}\|_{L_2({\mathcal O}) \to H^s({\mathcal O}')}
\leqslant {\mathscr C}'_s  t^{-1/2}(\delta^{-2} + t^{-1})^{(s-1)/2} e^{-c_{\flat}t/2} ,\quad t>0.
\end{equation}
The constant ${\mathscr C}'_s$ depends on $s$, $d$, $\alpha _0$, $\alpha _1$, $\Vert g\Vert _{L_\infty}$, $\Vert g^{-1}\Vert _{L_\infty}$,
the constants ${\mathfrak C}_1$ and ${\mathfrak C}_2$ from the inequality \eqref{Garding}, the constant
$\widetilde{\mathfrak C}_1$ from \eqref{b(D)u>= in H_perp}, and the domain $\mathcal{O}$.

The way of the proof of estimates \eqref{apriori_N} is the same as in the case of the operator ${\mathcal A}_D^0$; see comments in Subsection 3.8.
The difference is that instead of Lemma 3.1 one should use Lemma 7.1.

The following statement is deduced from  \eqref{apriori_N}
by using the properties of the matrix-valued function $\Lambda({\mathbf x})$ and the operator $S_\eps$.

\begin{lemma}
Suppose that the assumptions of Theorem \textnormal{7.12} are satisfied, and let $d\geqslant 3$.
Let $\mathcal{K}_N^0(t;\varepsilon)$ be the operator \eqref{K_N^0(t;e)}.
Let $r_1$ be defined by  \textnormal{(\ref{r})}. Then for $0<\varepsilon \leqslant (4r_1)^{-1}\delta$ and
$t>0$ we have
\begin{equation}
\label{K_minus_K_N}
\| \mathcal{K}_N(t;\varepsilon) - \mathcal{K}_N^0(t;\varepsilon)
\|_{L_2({\mathcal O}) \to H^1({\mathcal O}')}
\leqslant {\mathscr C}''_d   (t^{-1} +  t^{-1/2}(\delta^{-2} + t^{-1})^{d/4})e^{-c_{\flat}t/2} .
\end{equation}
The constant ${{\mathscr C}}''_d$ depends only on $d$,
$\alpha _0$, $\alpha _1$, $\Vert g\Vert _{L_\infty}$, $\Vert g^{-1}\Vert _{L_\infty}$,
the constants ${\mathfrak C}_1$ and ${\mathfrak C}_2$ from the inequality \eqref{Garding}, the constant
$\widetilde{\mathfrak C}_1$ from \eqref{b(D)u>= in H_perp}, and the domain $\mathcal{O}$.
\end{lemma}

Lemma 7.14 and Theorem 7.12 imply the following result.

\begin{theorem}
Suppose that the assumptions of Theorem \textnormal{7.12} are satisfied, and let $d \geqslant 3$.
Let $\mathcal{K}_N^0(t;\varepsilon)$ be the operator \eqref{K_N^0(t;e)}.
Then for $0<\varepsilon\leqslant \min\{\varepsilon _1;(4r_1)^{-1}\delta\}$ and $t>0$ we have
\begin{align}
\label{7.100}
&\Vert  e^{-\mathcal{A}_{N,\varepsilon}t}-e^{-\mathcal{A}_N^0t}-\varepsilon \mathcal{K}_N^0(t;\varepsilon)\Vert _{L_2(\mathcal{O})\rightarrow H^1(\mathcal{O}')}\leqslant
{\mathscr C}_d \eps h_d(\delta;t) e^{-c_{\flat}t/2},
\\
\label{7.33}
&\Vert g^\varepsilon b(\mathbf{D})e^{-\mathcal{A}_{N,\varepsilon}t}-\widetilde{g}^\varepsilon
b(\mathbf{D}) e^{-\mathcal{A}_N^0 t}\Vert _{L_2(\mathcal{O})\rightarrow L_2(\mathcal{O}')}
\leqslant
\widetilde{{\mathscr C}}_d \eps h_d(\delta;t) e^{-c_{\flat}t/2}.
\end{align}
Here $h_d(\delta;t)$ is defined by \eqref{h_d(delta;t)}.
The constants ${{\mathscr C}}_d$ and $\widetilde{{\mathscr C}}_d$ depend on the data \textnormal{\eqref{data2}}.
\end{theorem}

The proofs of Lemma 7.14 and Theorem 7.15 are postponed until Appendix (see \S 9).
It is convenient to apply Theorem 7.15 if $t$ is separated from zero.

\section*{\S 8. Homogenization of the solutions of the second initial boundary value problem}
\setcounter{section}{8}
\setcounter{equation}{0}
\setcounter{theorem}{0}

In this section, we apply the results of \S 7 to homogenization of the solutions of the second initial boundary value problem for
parabolic systems with rapidly oscillating coefficients.

\subsection*{8.1. The problem for the homogeneous parabolic equation.}
We start with the case of homogeneous equation and nonhomogeneous initial condition.
We are interested in the behaviour for small $\varepsilon$ of the generalized solution of the problem
\begin{equation}
\label{second problem}
\begin{cases} \frac{\partial \mathbf{u}_\varepsilon (\mathbf{x},t)}{\partial t}=-b(\mathbf{D})^*g^\varepsilon (\mathbf{x})b(\mathbf{D})
\mathbf{u}_\varepsilon (\mathbf{x}, t),\quad\mathbf{x}\in \mathcal{O},\;t>0,\\
\mathbf{u}_\varepsilon(\mathbf{x},0)=\boldsymbol{\varphi}(\mathbf{x}),\quad \mathbf{x}\in\mathcal{O},\\
\partial _{\boldsymbol{\nu}}^\varepsilon \mathbf{u}_\varepsilon(\cdot ,t)\vert _{\partial \mathcal{O}}=0,\quad t>0.
\end{cases}
\end{equation}
Here $\boldsymbol{\varphi}\in L_2(\mathcal{O};\mathbb{C}^n)$. The symbol  $\partial _{\boldsymbol{\nu}}^\varepsilon$ stands for the conormal
derivative corresponding to the operator $b(\mathbf{D})^*g^\varepsilon (\mathbf{x})b(\mathbf{D})$.
The formal differential expression for $\partial _{\boldsymbol{\nu}}^\varepsilon$ is given by
$\partial _{\boldsymbol{\nu}}^\varepsilon \mathbf{u}(\mathbf{x}):=b(\boldsymbol{\nu}(\mathbf{x}))^*g^\varepsilon (\mathbf{x})b(\nabla)\mathbf{u}(\mathbf{x})$.

The solution of problem \eqref{second problem} is represented as $\mathbf{u}_\varepsilon(\cdot ,t)=e^{-\mathcal{A}_{N,\varepsilon}t}\boldsymbol{\varphi}$. Thus, the question about homogenization of solutions
is reduced to the question about approximation of the operator $e^{-\mathcal{A}_{N,\varepsilon}t}$ studied
in \S 7.

The corresponding effective problem has the form
\begin{equation}
\label{eff. second problem}
\begin{cases}
\frac{\partial \mathbf{u}_0 (\mathbf{x},t)}{\partial t}=-b(\mathbf{D})^*g^0b(\mathbf{D})
\mathbf{u}_0(\mathbf{x},t),\quad \mathbf{x}\in \mathcal{O},\; t>0,\\
\mathbf{u}_0(\mathbf{x},0)=\boldsymbol{\varphi}(\mathbf{x}),\quad \mathbf{x}\in \mathcal{O},\\
\partial _{\boldsymbol{\nu}}^0\mathbf{u}_0 (\cdot ,t)\vert _{\partial \mathcal{O}}=0,\quad t>0.
\end{cases}
\end{equation}

We have $\mathbf{u}_\varepsilon (\cdot ,t)=e^{-\mathcal{A}_{N,\varepsilon}t}\boldsymbol{\varphi}$, $\mathbf{u}_0(\cdot ,t)=e^{-\mathcal{A}_N^0 t}\boldsymbol{\varphi}$. Therefore, Theorems 7.2 and 7.3 together with the identity $e^{-\mathcal{A}_N^0t}\mathcal{P}\boldsymbol{\varphi}=\mathcal{P}e^{-\mathcal{A}_N^0 t}\boldsymbol{\varphi}=\mathcal{P}\mathbf{u}_0(\cdot ,t)$ imply the following result.

\begin{theorem}
Suppose that the assumptions of Theorem \textnormal{6.2} are satisfied.
Let $\mathbf{u}_\varepsilon$ and $\mathbf{u}_0$ be the solutions of problems \textnormal{(\ref{second problem}) and (\ref{eff. second problem})},
respectively, where $\boldsymbol{\varphi}\in L_2(\mathcal{O};\mathbb{C}^n)$. Then for $0<\varepsilon \leqslant \varepsilon _1$
and $t\geqslant 0$ we have
\begin{equation*}
\Vert \mathbf{u}_\varepsilon (\cdot ,t)-\mathbf{u}_0(\cdot ,t)\Vert _{L_2(\mathcal{O})}\leqslant \mathcal{C}_{12}\varepsilon (t+\varepsilon^2)^{-1/2}e^{-c_\flat  t/2}\Vert \boldsymbol{\varphi}\Vert _{L_2(\mathcal{O})}.
\end{equation*}
Denote $\widehat{\mathbf{u}}_0:=P_\mathcal{O}\mathcal{P}\mathbf{u}_0$,
$\mathbf{p}_\varepsilon :=g^\varepsilon b(\mathbf{D})\mathbf{u}_\varepsilon$.
Then for $0<\varepsilon \leqslant \varepsilon  _1$ and $t>0$ we have
\begin{align*}
\begin{split}
\Vert &\mathbf{u}_\varepsilon (\cdot ,t)-\mathbf{u}_0(\cdot ,t)-\varepsilon \Lambda ^\varepsilon S_\varepsilon b(\mathbf{D})\widehat{\mathbf{u}}_0(\cdot ,t)\Vert _{H^1(\mathcal{O})}\\
&\leqslant \mathcal{C}_{13}(\varepsilon ^{1/2}t^{-3/4}+\varepsilon t^{-1})e^{-c_\flat  t/2}\Vert \boldsymbol{\varphi}\Vert _{L_2(\mathcal{O})},
\end{split}
\\
\begin{split}
\Vert &\mathbf{p}_\varepsilon (\cdot ,t)-\widetilde{g}^\varepsilon S_\varepsilon b(\mathbf{D})\widehat{\mathbf{u}}_0(\cdot ,t)\Vert _{L_2(\mathcal{O})}
\leqslant \widetilde{\mathcal{C}}_{13}(\varepsilon ^{1/2}t^{-3/4}+\varepsilon t^{-1})e^{-c_\flat  t/2}\Vert \boldsymbol{\varphi}\Vert _{L_2(\mathcal{O})}.
\end{split}
\end{align*}
\end{theorem}

In the case where $\Lambda \in L_\infty$ Theorem 7.6 leads to the following result.

\begin{theorem}
Suppose that the assumptions of Theorem \textnormal{8.1} are satisfied. Suppose that Condition \textnormal{2.5} is satisfied.
Then for $0<\varepsilon \leqslant \varepsilon_1$ and $t>0$ we have
\begin{align*}
\begin{split}
\Vert &\mathbf{u}_\varepsilon (\cdot ,t)-\mathbf{u}_0 (\cdot ,t)-\varepsilon \Lambda ^\varepsilon b(\mathbf{D})\mathbf{u}_0 (\cdot ,t)\Vert _{H^1(\mathcal{O})}\\
&\leqslant\mathcal{C}_{15}(\varepsilon ^{1/2}t^{-3/4}+\varepsilon t^{-1})
e^{-c_\flat  t/2}\Vert \boldsymbol{\varphi}\Vert _{L_2(\mathcal{O})},
\end{split}
\\
\begin{split}
\Vert  &\mathbf{p}_\varepsilon (\cdot ,t)-\widetilde{g}^\varepsilon b(\mathbf{D})\mathbf{u}_0(\cdot ,t)\Vert _{L_2(\mathcal{O})}
\leqslant \widetilde{\mathcal{C}}_{15} (\varepsilon ^{1/2}t^{-3/4}+\varepsilon t^{-1})e^{-c_\flat  t/2}\Vert \boldsymbol{\varphi}\Vert _{L_2(\mathcal{O})}.
\end{split}
\end{align*}
\end{theorem}

In the case of sufficiently smooth boundary from Theorem 7.9 we obtain the following result.

\begin{theorem}
Suppose that the assumptions of Theorem~{\rm 8.1} are satisfied, and let $d\geqslant 3$.
Suppose that $\partial \mathcal{O}\in C^{d/2,1}$ if $d$ is even and  $\partial\mathcal{O}\in C^{(d+1)/2,1}$ if $d$ is odd.
Then for $0<\varepsilon \leqslant \varepsilon _1$ and $t>0$ we have
\begin{align*}
\begin{split}
\Vert &\mathbf{u}_\varepsilon (\cdot ,t)-\mathbf{u}_0(\cdot ,t)-\varepsilon \Lambda ^\varepsilon b(\mathbf{D})\mathbf{u}_0(\cdot ,t)\Vert _{H^1(\mathcal{O})}\\
&\leqslant \mathscr{C}_d^*(\varepsilon ^{1/2}t^{-3/4}+\varepsilon t^{-1}+\varepsilon t^{-d/4-1/2})e^{-c_\flat t/2}\Vert \boldsymbol{\varphi}\Vert _{L_2(\mathcal{O})},
\end{split}
\\
\begin{split}
\Vert &\mathbf{p}_\varepsilon (\cdot ,t)-\widetilde{g}^\varepsilon b(\mathbf{D})\mathbf{u}_0(\cdot ,t)\Vert _{L_2(\mathcal{O})}\\
&\leqslant \widetilde{\mathscr{C}}_d^*(\varepsilon ^{1/2}t^{-3/4}+\varepsilon t^{-1}+\varepsilon t^{-d/4-1/2})e^{-c_\flat t/2}\Vert \boldsymbol{\varphi}\Vert _{L_2(\mathcal{O})}.
\end{split}
\end{align*}
\end{theorem}

Special cases are distinguished with the help of Proposition 7.11.

\begin{proposition}
Suppose that the assumptions of Theorem \textnormal{8.1} are satisfied.

\noindent
$1^\circ$. If $g^0=\overline{g}$, i.~e., relations \textnormal{(\ref{overline-g})} are satisfied, then
for $0< \eps \leqslant \eps_1$ and $t>0$ we have
$$
\Vert \mathbf{u}_\varepsilon (\cdot ,t)-\mathbf{u}_0(\cdot ,t)\Vert _{H^1(\mathcal{O})}\leqslant
\mathcal{C}_{13}(\varepsilon^{1/2} t^{-3/4}+\varepsilon t^{-1}) e^{-c_\flat  t/2}
\Vert \boldsymbol{\varphi}\Vert _{L_2(\mathcal{O})}.
$$

\noindent
$2^\circ$. If $g^0=\underline{g}$, i.~e., relations \textnormal{(\ref{underline-g})} are satisfied, then for
\hbox{$0< \eps \leqslant \eps_1$} and $t>0$ we have
$$
\Vert \mathbf{p}_\varepsilon (\cdot ,t)-\mathbf{p}_0(\cdot ,t)\Vert _{L_2(\mathcal{O})}\leqslant
\widetilde{\mathcal{C}}_{15}(\varepsilon^{1/2} t^{-3/4}+\varepsilon t^{-1})e^{-c_\flat  t/2}
\Vert \boldsymbol{\varphi}\Vert _{L_2(\mathcal{O})},
$$
where $\mathbf{p}_0 = g^0 b(\D) \mathbf{u}_0$.
\end{proposition}

Approximation of the solutions in a strictly interior subdomain is deduced from Theorem 7.12.

\begin{theorem}
Suppose that the assumptions of Theorem \textnormal{8.1} are satisfied.
Let $\mathcal{O}'$ be a strictly interior subdomain of the domain $\mathcal{O}$, and let $\delta =\textnormal{dist}\,\lbrace \mathcal{O}';\partial \mathcal{O}\rbrace $. Then for $0<\varepsilon\leqslant \varepsilon _1$ and $t>0$ we have
\begin{align*}
\begin{split}
\Vert &\mathbf{u}_\varepsilon (\cdot ,t)-\mathbf{u}_0(\cdot ,t)-\varepsilon \Lambda ^\varepsilon S_\varepsilon b(\mathbf{D})\widehat{\mathbf{u}}_0(\cdot ,t)\Vert _{H^1(\mathcal{O}')}\\
&\leqslant (\mathcal{C}_{16}\delta ^{-1}+\mathcal{C}_{17})\varepsilon t^{-1}e^{-c_\flat  t/2}\Vert \boldsymbol{\varphi}\Vert _{L_2(\mathcal{O})} ,
\end{split}\\
\begin{split}
\Vert &\mathbf{p}_\varepsilon (\cdot ,t)-\widetilde{g}^\varepsilon S_\varepsilon b(\mathbf{D})\widehat{\mathbf{u}}_0(\cdot ,t)\Vert _{L_2(\mathcal{O}')}
\leqslant (\widetilde{\mathcal{C}}_{16}\delta ^{-1}+\widetilde{\mathcal{C}}_{17})\varepsilon
t^{-1}e^{-c_\flat  t/2}\Vert \boldsymbol{\varphi}\Vert _{L_2(\mathcal{O})}.
\end{split}
\end{align*}
\end{theorem}

In the case where $\Lambda \in L_\infty$ Theorem 7.13 implies the following result.

\begin{theorem}
Suppose that the assumptions of Theorem \textnormal{8.5} are satisfied. Suppose that the matrix-valued
function $\Lambda (\mathbf{x})$ satisfies Condition \textnormal{2.5}. Then for \hbox{$0<\varepsilon \leqslant \varepsilon _1$} and $t>0$ we have
\begin{align*}
\begin{split}
\Vert &\mathbf{u}_\varepsilon (\cdot ,t)-\mathbf{u}_0(\cdot ,t)-\varepsilon \Lambda ^\varepsilon b(\mathbf{D})\mathbf{u}_0(\cdot ,t)\Vert _{H^1(\mathcal{O}')}\\
&\leqslant (\mathcal{C}_{16}\delta ^{-1}+\check{\mathcal{C}}_{17})\varepsilon t^{-1}e^{-c_\flat  t/2}\Vert \boldsymbol{\varphi}\Vert _{L_2(\mathcal{O})} ,
\end{split}\\
\begin{split}
\Vert &\mathbf{p}_\varepsilon (\cdot ,t)-\widetilde{g}^\varepsilon  b(\mathbf{D})\mathbf{u}_0(\cdot ,t)\Vert _{L_2(\mathcal{O}')}
\leqslant (\widetilde{\mathcal{C}}_{16}\delta ^{-1}+\widehat{\mathcal{C}}_{17})\varepsilon t^{-1}e^{-c_\flat  t/2}\Vert \boldsymbol{\varphi}\Vert _{L_2(\mathcal{O})}.
\end{split}
\end{align*}
\end{theorem}

Finally, for $d\geqslant 3$ Theorem 7.15 implies the following result.

\begin{theorem}
Let $d\geqslant 3$. Under the assumptions of Theorem~{\rm 8.5} for $0<\varepsilon\leqslant \min \lbrace \varepsilon _1;(4r_1)^{-1}\delta \rbrace $ and $t>0$ we have
\begin{align*}
\begin{split}
\Vert &\mathbf{u}_\varepsilon (\cdot ,t)-\mathbf{u}_0(\cdot ,t)-\varepsilon \Lambda ^\varepsilon b(\mathbf{D})\mathbf{u}_0(\cdot ,t)\Vert _{H^1(\mathcal{O}')}\leqslant \mathscr{C}_d\varepsilon h_d(\delta ;t )e^{-c_\flat t/2}\Vert  \boldsymbol{\varphi}\Vert _{L_2(\mathcal{O})},
\end{split}
\\
\begin{split}
\Vert &\mathbf{p}_\varepsilon (\cdot ,t)-\widetilde{g}^\varepsilon b(\mathbf{D})\mathbf{u}_0(\cdot ,t)\Vert _{L_2(\mathcal{O}')}\leqslant \widetilde{\mathscr{C}}_d\varepsilon h_d(\delta ;t )e^{-c_\flat t/2}\Vert  \boldsymbol{\varphi}\Vert _{L_2(\mathcal{O})}.
\end{split}
\end{align*}
Here $h_d(\delta ;t )$ is defined by \textnormal{\eqref{h_d(delta;t)}}.
\end{theorem}

\subsection*{8.2. The problem for nonhomogeneous parabolic equation. Approximation of the solutions in $L_2({\mathcal O};{\mathbb C}^n)$ and
${\mathfrak H}_p(T)$ }
Now, we consider the problem
\begin{equation}
\label{non-homogeneous Neumann}
\begin{cases}
\frac{\partial \mathbf{u}_\varepsilon (\mathbf{x},t)}{\partial t}=-b(\mathbf{D})^*g^\varepsilon (\mathbf{x})b(\mathbf{D})
\mathbf{u}_\varepsilon (\mathbf{x},t)+\mathbf{F}(\mathbf{x},t),\quad \mathbf{x}\in \mathcal{O},\quad 0<t<T,\\
\mathbf{u}_\varepsilon (\mathbf{x},0)=\boldsymbol{\varphi}(\mathbf{x}),\quad \mathbf{x}\in \mathcal{O},\\
\partial _{\boldsymbol{\nu}}^\varepsilon \mathbf{u}_\varepsilon (\cdot ,t)\vert _{\partial \mathcal{O}}=0,\quad 0<t<T.
\end{cases}
\end{equation}
Here $\boldsymbol{\varphi}\in L_2(\mathcal{O};\mathbb{C}^n)$ and $\mathbf{F}\in \mathfrak{H}_p (T)=L_p((0,T);L_2(\mathcal{O};\mathbb{C}^n))$, $0<T\leqslant \infty$, for some $1\leqslant p\leqslant \infty$.

The corresponding effective problem is given by
\begin{equation}
\label{non-homogeneous eff. Neumann}
\begin{cases}
\frac{\partial \mathbf{u}_0(\mathbf{x},t)}{\partial t}=-b(\mathbf{D})^*g^0b(\mathbf{D})
\mathbf{u}_0(\mathbf{x},t)+\mathbf{F}(\mathbf{x},t),\quad \mathbf{x}\in \mathcal{O},\quad 0<t<T,\\
\mathbf{u}_0(\mathbf{x},0)=\boldsymbol{\varphi}(\mathbf{x}),\quad \mathbf{x}\in \mathcal{O},\\
\partial _{\boldsymbol{\nu}}^0 \mathbf{u}_0(\cdot ,t)\vert _{\partial \mathcal{O}}=0,\quad 0<t<T.
\end{cases}
\end{equation}

We have $\mathbf{u}_\varepsilon (\cdot ,t)=e^{-\mathcal{A}_{N,\varepsilon}t}\boldsymbol{\varphi}(\cdot )+\int _0^t e^{-\mathcal{A}_{N,\varepsilon}(t-\widetilde{t})}\mathbf{F}(\cdot ,\widetilde{t})\,d\widetilde{t}$. Similarly, $\mathbf{u}_0(\cdot ,t)=e^{-\mathcal{A}_N^0 t}\boldsymbol{\varphi}(\cdot)+\int _0^t e^{-\mathcal{A}_N^0(t-\widetilde{t})}\mathbf{F}(\cdot ,\widetilde{t})\,d\widetilde{t}$.
By analogy with the proof of Theorem 4.8, we deduce the following result from Theorem 7.2.

\begin{theorem}
Suppose that the assumptions of Theorem \textnormal{6.1} are satisfied.
Let $\mathbf{u}_\varepsilon$ be the solution of problem \textnormal{(\ref{non-homogeneous Neumann})},
where $\boldsymbol{\varphi}\in L_2(\mathcal{O};\mathbb{C}^n)$ and $\mathbf{F}\in \mathfrak{H}_p(T)$, $0<T\leqslant \infty$, with some
$1<p\leqslant \infty$. Let $\mathbf{u}_0$ be the solution of the effective problem \textnormal{(\ref{non-homogeneous eff. Neumann})}.
Then for $0<\varepsilon \leqslant \varepsilon _1$ and $0<t<T$ we have
\begin{equation*}
\Vert \mathbf{u}_\varepsilon (\cdot ,t)-\mathbf{u}_0(\cdot ,t)\Vert _{L_2(\mathcal{O})}
\leqslant \mathcal{C}_{12}\varepsilon (t+\varepsilon ^2)^{-1/2}e^{-c_\flat  t/2}\Vert \boldsymbol{\varphi}\Vert _{L_2(\mathcal{O})}
+ {\mathfrak c}_p {\theta} (\varepsilon ,p)\Vert \mathbf{F}\Vert _{\mathfrak{H}_p(t)}.
\end{equation*}
Here ${\theta}(\eps ,p)$ is defined by \eqref{theta (varepsilon ,p)}. The constant ${\mathfrak c}_p$ depends only on the data
\textnormal{\eqref{data2}} and $p$.
\end{theorem}

The following result can be proved with the help of Theorem 7.2 by analogy with the proof of Theorem 4.9.

\begin{theorem}
Suppose that the assumptions of Theorem \textnormal{6.1} are satisfied. Let $\mathbf{u}_\varepsilon$ be the solution
of problem \eqref{non-homogeneous Neumann}, and let $\mathbf{u}_0$ be the solution of problem \eqref{non-homogeneous eff. Neumann}
with $\boldsymbol{\varphi}\in L_2(\mathcal{O};\mathbb{C}^n)$ and $\mathbf{F}\in \mathfrak{H}_p(T)$, $0<T\leqslant \infty$, for some
$1\leqslant p <\infty$. Then for $0<\varepsilon \leqslant \varepsilon _1$ we have
\begin{equation*}
\Vert \mathbf{u}_\varepsilon -\mathbf{u}_0\Vert _{\mathfrak{H}_p(T)}
\leqslant {\mathfrak c}_{p'} {\Theta}(\varepsilon ,p)\Vert \boldsymbol{\varphi}\Vert _{L_2(\mathcal{O})}+
\mathcal{C}_{18} \varepsilon\Vert \mathbf{F}\Vert _{\mathfrak{H}_p(T)}.
\end{equation*}
Here ${\Theta}(\varepsilon ,p)$ is defined by \eqref{vartheta (varepsilon,p)}, and $\mathcal{C}_{18}=\mathcal{C}_{12} (2\pi)^{1/2} c_\flat^{-1/2}$.
\end{theorem}

\subsection*{8.3. Approximation of the solutions in $H^1({\mathcal O};{\mathbb C}^n)$}
Now we obtain approximation of the solution $\mathbf{u}_\varepsilon $ of problem (\ref{non-homogeneous Neumann})
in $H^1(\mathcal{O};\mathbb{C}^n)$.
By analogy with the proof of Theorem 4.11, we deduce the following result from Theorem 7.3 and Proposition 7.5.

\begin{theorem}
Suppose that the assumptions of Theorem \textnormal{6.2} are satisfied. Let $\mathbf{u}_\varepsilon$ be the solution
of problem \textnormal{(\ref{non-homogeneous Neumann})}, where $\boldsymbol{\varphi}\in L_2(\mathcal{O};\mathbb{C}^n)$ and
$\mathbf{F}\in \mathfrak{H}_p(T)$, $0<T\leqslant \infty$, for some $2<p\leqslant \infty$. Let $\mathbf{u}_0$ be the
solution of the effective problem \textnormal{(\ref{non-homogeneous eff. Neumann})}.
Assuming that $t\geqslant \varepsilon ^2$, denote
\begin{equation}
\label{w_eps neumann}
\mathbf{w}_\varepsilon (\cdot ,t):=e^{-\mathcal{A}_N^0 \varepsilon ^2}\mathbf{u}_0(\cdot ,t-\varepsilon ^2)=e^{-\mathcal{A}_N^0 t}\boldsymbol{\varphi}(\cdot)+\int _0 ^{t-\varepsilon ^2}e^{-\mathcal{A}_N^0 (t-\widetilde{t})}\mathbf{F}(\cdot ,\widetilde{t})\,d\widetilde{t}.
\end{equation}
We put $\widehat{\mathbf{w}}_\varepsilon(\cdot ,t) :=P_\mathcal{O}\mathcal{P}\mathbf{w}_\varepsilon (\cdot ,t) $,
$\mathbf{p}_\varepsilon =g^\varepsilon b(\mathbf{D})\mathbf{u}_\varepsilon$.
Then for $0<\varepsilon \leqslant \varepsilon _1$ and $\varepsilon ^2\leqslant t<T$ we have
\begin{equation*}
\begin{split}
\Vert &\mathbf{u}_\varepsilon (\cdot ,t)-\mathbf{u}_0(\cdot ,t)-\varepsilon \Lambda ^\varepsilon S_\varepsilon b(\mathbf{D})\widehat{\mathbf{w}}_\varepsilon (\cdot ,t)\Vert _{H^1(\mathcal{O})}\\
&\leqslant 2\mathcal{C}_{13}\varepsilon ^{1/2}t^{-3/4}e^{-c_\flat t/2}\Vert \boldsymbol{\varphi}\Vert _{L_2(\mathcal{O})}+
\check{\mathfrak c}_p {\rho}(\varepsilon ,p)\Vert \mathbf{F}\Vert_{\mathfrak{H}_p(t)},
\end{split}
\end{equation*}
\begin{equation*}
\begin{split}
\Vert &\mathbf{p}_\varepsilon (\cdot ,t)-\widetilde{g}^\varepsilon S_\varepsilon b(\mathbf{D})\widehat{\mathbf{w}}_\varepsilon (\cdot ,t)\Vert _{L_2(\mathcal{O})}\\
&\leqslant 2\widetilde{\mathcal{C}}_{13}\varepsilon ^{1/2}t^{-3/4}e^{-c_\flat t/2}\Vert \boldsymbol{\varphi}\Vert _{L_2(\mathcal{O})}+
\widetilde{\mathfrak c}_p {\rho}(\varepsilon ,p)\Vert \mathbf{F}\Vert _{\mathfrak{H}_p(t)}.
\end{split}
\end{equation*}
Here ${\rho}(\eps ,p)$ is defined by \eqref{rho}.
The constants $\check{\mathfrak c}_p$ and $\widetilde{\mathfrak c}_p$
depend only on the data \textnormal{\eqref{data2}} and $p$.
\end{theorem}

In the case where $\Lambda \in L_\infty$ Theorem 7.6 and Proposition 7.5 imply the following result similar to Theorem 4.12.

\begin{theorem}
Suppose that the assumptions of Theorem \textnormal{8.10} are satisfied. Suppose that the matrix-valued function $\Lambda (\mathbf{x})$ satisfies Condition \textnormal{2.5}. Then for \hbox{$0<\varepsilon \leqslant \varepsilon _1$} and $\varepsilon ^2\leqslant t<T$ we have
\begin{align*}
\begin{split}
\Vert &\mathbf{u}_\varepsilon (\cdot ,t)-\mathbf{u}_0(\cdot ,t)-\varepsilon \Lambda ^\varepsilon b(\mathbf{D})\mathbf{w}_\varepsilon (\cdot ,t)\Vert _{H^1(\mathcal{O})}\\
&\leqslant 2\mathcal{C}_{15}\varepsilon ^{1/2}t^{-3/4}e^{-c_\flat t/2}\Vert \boldsymbol{\varphi}\Vert _{L_2(\mathcal{O})}
+{\mathfrak c}_p' \rho(\varepsilon ,p)\Vert \mathbf{F}\Vert _{\mathfrak{H}_p(t)},
\end{split}\\
\begin{split}
\Vert &\mathbf{p}_\varepsilon (\cdot ,t)-\widetilde{g}^\varepsilon b(\mathbf{D})\mathbf{w}_\varepsilon (\cdot ,t)\Vert _{L_2(\mathcal{O})}\\
&\leqslant 2\widetilde{\mathcal{C}}_{15}\varepsilon ^{1/2}t^{-3/4}e^{-c_\flat t/2}\Vert \boldsymbol{\varphi}\Vert _{L_2(\mathcal{O})} +
{\mathfrak c}_p'' {\rho}(\eps ,p)\Vert \mathbf{F}\Vert _{\mathfrak{H}_p(t)} .
\end{split}
\end{align*}
The constants ${\mathfrak c}_p'$ and ${\mathfrak c}''_p$ depend only on the data \textnormal{\eqref{data2}},  $\|\Lambda\|_{L_\infty}$, and $p$.
\end{theorem}

In the case of additional smoothness of the boundary Theorem 7.9 implies the following result (cf. Proposition 4.13).

\begin{proposition}
Suppose that the assumptions of Theorem~{\rm 8.10} are satisfied. Let $d=3$ and $p=\infty$.
Suppose that $\partial \mathcal{O}\in C^{2,1}$. Then for $0<\varepsilon \leqslant \varepsilon _1$ and $\varepsilon ^2\leqslant t <T$
we have
\begin{align*}
\begin{split}
\Vert &\mathbf{u}_\varepsilon (\cdot ,t)-\mathbf{u}_0(\cdot ,t)-\varepsilon \Lambda ^\varepsilon b(\mathbf{D})\mathbf{w}_\varepsilon (\cdot ,t)\Vert _{H^1(\mathcal{O})}\\
&\leqslant 2\mathscr{C}_3^*(\varepsilon ^{1/2}t^{-3/4}+\varepsilon t^{-5/4})e^{-c_\flat t/2}\Vert \boldsymbol{\varphi}\Vert _{L_2(\mathcal{O})}+\mathfrak{c}'\varepsilon ^{1/2}\Vert \mathbf{F}\Vert _{\mathfrak{H}_\infty(t)},
\end{split}
\\
\begin{split}
\Vert &\mathbf{p}_\varepsilon (\cdot ,t)-\widetilde{g}^\varepsilon b(\mathbf{D})\mathbf{w}_\varepsilon (\cdot ,t)\Vert _{L_2(\mathcal{O})}\\
&\leqslant 2 \widetilde{\mathscr{C}}_3^*(\varepsilon ^{1/2}t^{-3/4}+\varepsilon t^{-5/4})e^{-c_\flat t/2}\Vert \boldsymbol{\varphi}\Vert _{L_2(\mathcal{O})}+\widetilde{\mathfrak{c}}'\varepsilon ^{1/2}\Vert \mathbf{F}\Vert _{\mathfrak{H}_\infty(t)}.
\end{split}
\end{align*}
The constants $\mathfrak{c}'$ and $\widetilde{\mathfrak{c}}'$ depend only on the data \textnormal{\eqref{data2}}.
\end{proposition}

Now we distinguish special cases. The following statement is proved with the help of
Propositions 7.5, 7.11 and Theorem 8.10, by analogy with the proof of Proposition 4.14.

\begin{proposition}
Suppose that the assumptions of Theorem \textnormal{8.10} are satisfied.

\noindent
$1^\circ$. If $g^0=\overline{g}$, i.~e., relations \textnormal{(\ref{overline-g})} are satisfied, then for
$0<\varepsilon\leqslant \varepsilon _1$ and $\varepsilon ^2\leqslant t<T$ we have
\begin{equation*}
\begin{split}
\Vert \mathbf{u}_\varepsilon (\cdot ,t)-\mathbf{u}_0(\cdot ,t)\Vert _{H^1(\mathcal{O})}
&\leqslant 2\mathcal{C}_{13}\varepsilon ^{1/2}t^{-3/4}
e^{-c_\flat t/2}\Vert \boldsymbol{\varphi}\Vert _{L_2(\mathcal{O})}\\
&+ \check{\mathfrak c}_p {\rho}(\eps ,p)\Vert \mathbf{F}(\cdot ,t)\Vert _{\mathfrak{H}_p(t)}.
\end{split}
\end{equation*}

\noindent
$2^\circ$. If $g^0=\underline{g}$, i.~e., relations \textnormal{(\ref{underline-g})} are satisfied, then for
\hbox{$0<\varepsilon \leqslant \varepsilon _1$} and $\varepsilon ^2\leqslant t<T$ we have
\begin{equation*}
\Vert \mathbf{p}_\varepsilon (\cdot ,t)-\mathbf{p}_0(\cdot ,t)\Vert _{L_2(\mathcal{O})}\leqslant 2\widetilde{\mathcal{C}}_{15}\varepsilon ^{1/2}t^{-3/4}e^{-c_\flat t/2}\Vert \boldsymbol{\varphi}\Vert _{L_2(\mathcal{O})}+{\mathfrak c}_p'''{\rho}(\varepsilon ,p)\Vert \mathbf{F}\Vert _{\mathfrak{H}_p(t)},
\end{equation*}
where $\mathbf{p}_0=g^0b(\mathbf{D})\mathbf{u}_0$. The constant ${\mathfrak c}_p'''$ depends only on the data
\textnormal{\eqref{data2}} and $p$.
\end{proposition}

\subsection*{8.4. Approximation of the solutions in a strictly interior subdomain}
Approximation of the solutions in a strictly interior subdomain can be derived from Theorem 7.12 (cf. Theorem 4.15).

\begin{theorem}
Suppose that the assumptions of Theorem \textnormal{8.10} are satisfied.
Let $\mathcal{O}'$ be a strictly interior subdomain of the domain $\mathcal{O}$, and let $\delta =\textnormal{dist}\,\lbrace\mathcal{O}';\partial \mathcal{O}\rbrace $. Then for $0<\varepsilon \leqslant\varepsilon _1$ and $\varepsilon ^2\leqslant t<T$ we have
\begin{align*}
\begin{split}
\Vert &\mathbf{u}_\varepsilon (\cdot ,t)-\mathbf{u}_0(\cdot ,t)-\varepsilon \Lambda ^\varepsilon S_\varepsilon b(\mathbf{D})\widehat{\mathbf{w}}_\varepsilon (\cdot ,t)\Vert _{H^1(\mathcal{O}')}\\
&\leqslant (\mathcal{C}_{16}\delta ^{-1}+\mathcal{C}_{17})\varepsilon t^{-1}e^{-c_\flat t/2}\Vert \boldsymbol{\varphi}\Vert _{L_2(\mathcal{O})}+({\rm k}_p\delta^{-1}+ {\rm k}_p')\sigma(\varepsilon ,p) \Vert \mathbf{F}\Vert _{\mathfrak{H}_p(t)},
\end{split}
\\
\begin{split}
\Vert &\mathbf{p}_\varepsilon (\cdot ,t)-\widetilde{g}^\varepsilon S_\varepsilon b(\mathbf{D})\widehat{\mathbf{w}}_\varepsilon (\cdot ,t)\Vert _{L_2(\mathcal{O}')}\\
&\leqslant (\widetilde{\mathcal{C}}_{16}\delta ^{-1}+ \widetilde{\mathcal{C}}_{17})\varepsilon t^{-1}e^{-c_\flat t/2}
\Vert \boldsymbol{\varphi}\Vert _{L_2(\mathcal{O})}+ (\widetilde{\rm k}_p\delta^{-1}+ \widetilde{\rm k}_p'){\sigma}(\varepsilon ,p  )\Vert \mathbf{F}\Vert _{\mathfrak{H}_p(t)}.
\end{split}
\end{align*}
Here $\sigma(\varepsilon ,p )$ is defined by \eqref{sigma (varepsilon ,p,delta)}.
The constants ${\rm k}_p$,  ${\rm k}_p'$, $\widetilde{\rm k}_p$, and $\widetilde{\rm k}_p'$
depend only on the data \textnormal{\eqref{data2}} and $p$.
\end{theorem}

In the case where $\Lambda \in L_\infty$ from Theorem 7.13 we deduce the following result similar to Theorem 4.16.

\begin{theorem}
Suppose that the assumptions of Theorem \textnormal{8.14} are satisfied. Suppose that the matrix-valued function $\Lambda (\mathbf{x})$
satisfies Condition \textnormal{2.5}. For $0<\varepsilon\leqslant \varepsilon _1$ and $\varepsilon ^2\leqslant t<T$
we have
\begin{align*}
\begin{split}
\Vert &\mathbf{u}_\varepsilon (\cdot ,t)-\mathbf{u}_0(\cdot ,t)-\varepsilon \Lambda ^\varepsilon  b(\mathbf{D})\mathbf{w}_\varepsilon (\cdot ,t)\Vert _{H^1(\mathcal{O}')}\\
&\leqslant (\mathcal{C}_{16} \delta ^{-1}+\check{\mathcal{C}}_{17})\varepsilon t^{-1}
e^{-c_\flat t/2} \Vert \boldsymbol{\varphi}\Vert _{L_2(\mathcal{O})}+({\rm k}_p \delta^{-1}+ \check{\rm k}_p')
\sigma(\varepsilon ,p) \Vert \mathbf{F}\Vert _{\mathfrak{H}_p(t)},
\end{split}
\\
\begin{split}
\Vert &\mathbf{p}_\varepsilon (\cdot ,t)-\widetilde{g}^\varepsilon  b(\mathbf{D})\mathbf{w}_\varepsilon (\cdot ,t)\Vert _{L_2(\mathcal{O}')}\\
&\leqslant (\widetilde{\mathcal{C}}_{16}\delta ^{-1}+\widehat{\mathcal{C}}_{17})
\varepsilon t^{-1}e^{-c_\flat t/2}\Vert \boldsymbol{\varphi}\Vert _{L_2(\mathcal{O})}+
(\widetilde{\rm k}_p\delta^{-1}+ \widehat{\rm k}_p') \sigma(\varepsilon ,p )\Vert \mathbf{F}\Vert _{\mathfrak{H}_p(t)}.
\end{split}
\end{align*}
The constants ${\rm k}_p$ and $\widetilde{\rm k}_p$ are the same as in Theorem \textnormal{8.14}.
The constants $\check{\rm k}_p'$ and $\widehat{\rm k}_p'$ depend only on the data \textnormal{\eqref{data2}}, $\|\Lambda\|_{L_\infty}$, and $p$.
\end{theorem}

\subsection*{8.5. Approximation of the solutions in ${\mathfrak G}_p(T)=L_p((0,T);H^1(\mathcal{O};\mathbb{C}^n))$}

By analogy with the proof of Theorem 4.18, it is easy to check the following result.

\begin{theorem}
Suppose that the assumptions of Theorem \textnormal{6.2} are satisfied.
Let $\mathbf{u}_\varepsilon$ and $\mathbf{u}_0$ be the solutions of problems \textnormal{(\ref{non-homogeneous Neumann})} and \textnormal{(\ref{non-homogeneous eff. Neumann})}, respectively, where $\boldsymbol{\varphi}\in L_2(\mathcal{O};\mathbb{C}^n)$ and
$\mathbf{F}\in \mathfrak{H}_p(T)$, $0<T\leqslant \infty$.
Suppose that for $\varepsilon ^2\leqslant t<T$ the function $\mathbf{w}_\varepsilon (\cdot ,t)$ is defined by \eqref{w_eps neumann}.
For $0\leqslant t<\varepsilon ^2$ we put $\mathbf{w}_\varepsilon (\cdot ,t)=0$.
Denote $\widehat{\mathbf{w}}_\varepsilon := P_\mathcal{O}\mathcal{P}\mathbf{w}_\varepsilon (\cdot ,t)$. Let
$\mathbf{p}_\varepsilon = {g}^\varepsilon b(\mathbf{D}) {\mathbf{u}}_\varepsilon$.

\noindent $1^\circ$. Let $1\leqslant p<2$. Then for $0<\varepsilon \leqslant \varepsilon _1$ we have
\begin{align*}
&\Vert \mathbf{u}_\varepsilon -\mathbf{u}_0-\varepsilon \Lambda ^\varepsilon S_\varepsilon b(\mathbf{D})\widehat{\mathbf{w}}_\varepsilon \Vert _{\mathfrak{G}_p(T)}
\leqslant \kappa_p' {\alpha}(\varepsilon ,p)\Vert \boldsymbol{\varphi}\Vert _{L_2(\mathcal{O})}+\mathcal{C}_{19}\varepsilon ^{1/2}\Vert \mathbf{F}\Vert _{\mathfrak{H}_p(T)},\\
&\Vert \mathbf{p}_\varepsilon -\widetilde{g}^\varepsilon S_\varepsilon b(\mathbf{D})\widehat{\mathbf{w}}_\varepsilon \Vert _{L_p((0,T);L_2(\mathcal{O}))}
\leqslant  \kappa_p'' {\alpha}(\varepsilon ,p)\Vert \boldsymbol{\varphi}\Vert _{L_2(\mathcal{O})}+
\widetilde{\mathcal{C}}_{19}\varepsilon ^{1/2}\Vert \mathbf{F}\Vert _{\mathfrak{H}_p(T)}.
\end{align*}
Here ${\alpha}(\varepsilon ,p)$ is defined by \textnormal{\eqref{alpha(e,p)}}.
The constants $\mathcal{C}_{19}$ and $\widetilde{\mathcal{C}}_{19}$ depend only on the data \textnormal{\eqref{data2}}.
The constants $\kappa_p'$ and $\kappa_p''$ depend on the same parameters and on $p$.

\noindent $2^\circ$. Let $\boldsymbol{\varphi}=0$, and let $1\leqslant p \leqslant \infty$.
Then for $0<\varepsilon \leqslant \varepsilon _1$ we have
\begin{align*}
&\Vert \mathbf{u}_\varepsilon -\mathbf{u}_0-\varepsilon \Lambda ^\varepsilon S_\varepsilon b(\mathbf{D})\widehat{\mathbf{w}}_\varepsilon \Vert _{\mathfrak{G}_p(T)}
\leqslant\mathcal{C}_{19}\varepsilon ^{1/2}\Vert \mathbf{F}\Vert _{\mathfrak{H}_p(T)},\\
&\Vert \mathbf{p}_\varepsilon -\widetilde{g}^\varepsilon S_\varepsilon b(\mathbf{D})\widehat{\mathbf{w}}_\varepsilon
\Vert _{L_p((0,T);L_2(\mathcal{O}))}
\leqslant  \widetilde{\mathcal{C}}_{19}\varepsilon ^{1/2}\Vert \mathbf{F}\Vert _{\mathfrak{H}_p(T)}.
\end{align*}
\end{theorem}

Under the additional assumption that $\Lambda \in L_\infty$ we obtain the following theorem similar to Theorem 4.19.

\begin{theorem}
Suppose that the assumptions of Theorem \textnormal{8.16} and Condition \textnormal{2.5} are satisfied.

\noindent $1^\circ$. Let $1\leqslant p<2$. Then for $0<\varepsilon \leqslant \varepsilon _1$ we have
\begin{align*}
&\Vert \mathbf{u}_\varepsilon -\mathbf{u}_0-\varepsilon \Lambda ^\varepsilon  b(\mathbf{D}) {\mathbf{w}}_\varepsilon \Vert _{\mathfrak{G}_p(T)}
\leqslant \mathfrak{k}_p {\alpha}(\varepsilon ,p)\Vert \boldsymbol{\varphi}\Vert _{L_2(\mathcal{O})}+
\mathcal{C}_{20}\varepsilon ^{1/2}\Vert \mathbf{F}\Vert _{\mathfrak{H}_p(T)},
\\
&\Vert \mathbf{p}_\varepsilon -\widetilde{g}^\varepsilon b(\mathbf{D}) {\mathbf{w}}_\varepsilon \Vert _{L_p((0,T);L_2(\mathcal{O}))}
\leqslant  \widetilde{\mathfrak{k}}_p {\alpha}(\varepsilon ,p)\Vert \boldsymbol{\varphi}\Vert _{L_2(\mathcal{O})}+
\widetilde{\mathcal{C}}_{20}\varepsilon ^{1/2}\Vert \mathbf{F}\Vert _{\mathfrak{H}_p(T)}.
\end{align*}
The constants $\mathcal{C}_{20}$ and $\widetilde{\mathcal{C}}_{20}$ depend only on the data \textnormal{\eqref{data2}}
and $\|\Lambda \|_{L_\infty}$.
The constants $\mathfrak{k}_p$ and $\widetilde{\mathfrak{k}}_p$ depend on the same parameters and on $p$.

\noindent $2^\circ$. Let $\boldsymbol{\varphi}=0$, and let $1\leqslant p \leqslant \infty$.
Then for $0<\varepsilon \leqslant \varepsilon _1$ we have
\begin{align*}
&\Vert \mathbf{u}_\varepsilon -\mathbf{u}_0-\varepsilon \Lambda ^\varepsilon b(\mathbf{D}) {\mathbf{w}}_\varepsilon \Vert _{\mathfrak{G}_p(T)}
\leqslant\mathcal{C}_{20}\varepsilon ^{1/2}\Vert \mathbf{F}\Vert _{\mathfrak{H}_p(T)},
\\
&\Vert \mathbf{p}_\varepsilon -\widetilde{g}^\varepsilon b(\mathbf{D}) {\mathbf{w}}_\varepsilon
\Vert _{L_p((0,T);L_2(\mathcal{O}))}
\leqslant  \widetilde{\mathcal{C}}_{20}\varepsilon ^{1/2}\Vert \mathbf{F}\Vert _{\mathfrak{H}_p(T)}.
\end{align*}
\end{theorem}

The following result is similar to Proposition 4.20 and is deduced from Theorem 7.9.

\begin{proposition}
Suppose that the assumptions of Theorem~{\rm 8.16} are satisfied, and let $d=3$. Suppose that $\partial\mathcal{O}\in C^{2,1}$.

\noindent $1^\circ$. Let $1\leqslant p <4/3$. Then for $0<\varepsilon \leqslant\varepsilon _1$ we have
\begin{align*}
\begin{split}
\Vert &\mathbf{u}_\varepsilon -\mathbf{u}_0 -\varepsilon \Lambda ^\varepsilon b(\mathbf{D})\mathbf{w}_\varepsilon \Vert _{\mathfrak{G}_p(T)}
\leqslant \mathfrak{k}_p'\varepsilon ^{2/p-3/2}\Vert \boldsymbol{\varphi}\Vert _{L_2(\mathcal{O})}
+\mathcal{C}_{21}\varepsilon ^{1/2}\Vert \mathbf{F}\Vert _{\mathfrak{H}_p(T)},
\end{split}
\\
\begin{split}
\Vert &\mathbf{p}_\varepsilon -\widetilde{g}^\varepsilon b(\mathbf{D})\mathbf{w}_\varepsilon \Vert _{L_p((0,T);L_2(\mathcal{O}))}
\leqslant \widetilde{\mathfrak{k}}_p'\varepsilon ^{2/p-3/2}\Vert \boldsymbol{\varphi}\Vert _{L_2(\mathcal{O})}
+\widetilde{\mathcal{C}}_{21}\varepsilon ^{1/2}\Vert \mathbf{F}\Vert _{\mathfrak{H}_p(T)}.
\end{split}
\end{align*}
The constants $\mathcal{C}_{21}$ and $\widetilde{\mathcal{C}}_{21}$ depend only on the data \textnormal{\eqref{data2}}.
The constants $\mathfrak{k}_p'$ and $\widetilde{\mathfrak{k}}_p'$ depend on the same parameters and $p$.

\noindent $2^\circ$. Let $\boldsymbol{\varphi }=0$, and let $1\leqslant p\leqslant \infty$.
Then for $0<\varepsilon \leqslant \varepsilon _1$ we have
\begin{align*}
\begin{split}
\Vert &\mathbf{u}_\varepsilon -\mathbf{u}_0 -\varepsilon \Lambda ^\varepsilon b(\mathbf{D})\mathbf{w}_\varepsilon \Vert _{\mathfrak{G}_p(T)}
\leqslant \mathcal{C}_{21}\varepsilon ^{1/2}\Vert \mathbf{F}\Vert _{\mathfrak{H}_p(T)},
\end{split}
\\
\begin{split}
\Vert &\mathbf{p}_\varepsilon -\widetilde{g}^\varepsilon b(\mathbf{D})\mathbf{w}_\varepsilon \Vert _{L_p((0,T);L_2(\mathcal{O}))}
\leqslant
\widetilde{\mathcal{C}}_{21}\varepsilon ^{1/2}\Vert \mathbf{F}\Vert _{\mathfrak{H}_p(T)}.
\end{split}
\end{align*}
\end{proposition}

\subsection*{8.6. Approximation of the solutions in $L_p((0,T);H^1(\mathcal{O}';\mathbb{C}^n))$.}
The following result is similar to Theorem 4.21 and can be checked by using Theorem 7.12 and Proposition 7.5.

\begin{theorem}
Suppose that the assumptions of Theorem {\rm 8.16} are satisfied. Let $\mathcal{O}'$ be a strictly interior subdomain of the domain $\mathcal{O}$,
and let $\delta =\textnormal{dist}\,\lbrace \mathcal{O}';\partial \mathcal{O}\rbrace $.

\noindent
$1^\circ .$ Let $1\leqslant p<2$. Then for $0<\varepsilon \leqslant \varepsilon _1$ we have
\begin{align*}
\begin{split}
\Vert &\mathbf{u}_\varepsilon -\mathbf{u}_0-\varepsilon \Lambda ^\varepsilon S_\varepsilon b(\mathbf{D})\widehat{\mathbf{w}}_\varepsilon \Vert _{L_p((0,T);H^1(\mathcal{O}'))}\\
&\leqslant ({\rm k}_{p'}\delta ^{-1}+{\rm k}'_{p'}) \tau(\varepsilon,p)
\Vert \boldsymbol{\varphi}\Vert _{L_2(\mathcal{O})}+({\mathcal C}_{22}' \delta ^{-1}+{\mathcal C}_{22}'' )\varepsilon (\vert \ln \varepsilon\vert +1)\Vert \mathbf{F}\Vert _{\mathfrak{H}_p(T)},
\end{split}
\\
\begin{split}
\Vert &\mathbf{p}_\varepsilon -\widetilde{g}^\varepsilon S_\varepsilon b(\mathbf{D})\widehat{\mathbf{w}}_\varepsilon \Vert _{L_p((0,T);L_2(\mathcal{O}'))}\\
&\leqslant (\widetilde{{\rm k}}_{p'}\delta ^{-1}+\widetilde{{\rm k}}_{p'}') \tau(\varepsilon,p) \Vert \boldsymbol{\varphi}\Vert _{L_2(\mathcal{O})}
+(\widetilde{\mathcal C}_{22}' \delta ^{-1}+\widetilde{\mathcal C}_{22}'' )\varepsilon (\vert \ln \varepsilon \vert +1)\Vert \mathbf{F}\Vert _{\mathfrak{H}_p(T)}.
\end{split}
\end{align*}
Here $\tau(\varepsilon,p)$ is defined by \textnormal{(\ref{tau (varepsilon ,p)})}.

\noindent
$2^\circ .$ Let $\boldsymbol{\varphi}=0$, and let $1\leqslant p\leqslant \infty$. Then for $0<\varepsilon \leqslant \varepsilon _1$ we have
\begin{align*}
\begin{split}
\Vert &\mathbf{u}_\varepsilon -\mathbf{u}_0-\varepsilon \Lambda ^\varepsilon S_\varepsilon b(\mathbf{D})\widehat{\mathbf{w}}_\varepsilon \Vert _{L_p((0,T);H^1(\mathcal{O}'))}\\
&\leqslant ({\mathcal C}_{22}' \delta ^{-1}+{\mathcal C}_{22}'' )\varepsilon (\vert \ln \varepsilon\vert +1)
\Vert \mathbf{F}\Vert _{\mathfrak{H}_p(T)},
\end{split}
\\
\Vert &\mathbf{p}_\varepsilon -\widetilde{g}^\varepsilon S_\varepsilon b(\mathbf{D})\widehat{\mathbf{w}}_\varepsilon \Vert _{L_p((0,T);L_2(\mathcal{O}'))}
\leqslant (\widetilde{\mathcal C}_{22}' \delta ^{-1}+\widetilde{\mathcal C}_{22}'' )
\varepsilon (\vert \ln \varepsilon \vert +1)\Vert \mathbf{F}\Vert _{\mathfrak{H}_p(T)}.
\end{align*}
The constants ${\mathcal C}_{22}'$, ${\mathcal C}_{22}''$, $\widetilde{\mathcal C}_{22}'$, and $\widetilde{\mathcal C}_{22}''$
depend only on the data \textnormal{\eqref{data2}}.
The constants  ${\rm k}_{p'}$, ${\rm k}_{p'}'$, $\widetilde{{\rm k}}_{p'}$, and $\widetilde{{\rm k}}'_{p'}$ are the same as in Theorem~{\rm 8.14}
\textnormal{(}with $p$ replaced by $p'$\textnormal{)}.
\end{theorem}

Finally, in the case where $\Lambda \in L_\infty$ Theorem 7.13 and Proposition 7.5 imply the following result
similar to Theorem 4.22.

\begin{theorem}
Suppose that the assumptions of Theorem {\rm 8.19} are satisfied.
Suppose that the matrix-valued function $\Lambda (\mathbf{x})$ satisfies Condition {\rm 2.5}.

\noindent
$1^\circ .$ Let $1\leqslant p<2$. Then for $0<\varepsilon \leqslant \varepsilon _1$ we have
\begin{align*}
\begin{split}
\Vert &\mathbf{u}_\varepsilon -\mathbf{u}_0-\varepsilon \Lambda ^\varepsilon  b(\mathbf{D})\mathbf{w}_\varepsilon \Vert _{L_p((0,T);H^1(\mathcal{O}'))}\\
&\leqslant ({\rm k}_{p'}\delta ^{-1}+\check{{\rm k}}'_{p'}) \tau(\varepsilon,p) \Vert \boldsymbol{\varphi}\Vert _{L_2(\mathcal{O})}
+({\mathcal C}_{22}' \delta ^{-1}+ \check{\mathcal C}_{22}'' )\varepsilon (\vert \ln \varepsilon\vert +1)\Vert \mathbf{F}\Vert _{\mathfrak{H}_p(T)},
\end{split}
\\
\begin{split}
\Vert &\mathbf{p}_\varepsilon -\widetilde{g}^\varepsilon  b(\mathbf{D})\mathbf{w}_\varepsilon \Vert _{L_p((0,T);L_2(\mathcal{O}'))}\\
&\leqslant (\widetilde{{\rm k}}_{p'}\delta ^{-1}+\widehat{{\rm k}}_{p'}')  \tau(\varepsilon,p) \Vert \boldsymbol{\varphi}\Vert _{L_2(\mathcal{O})}
+(\widetilde{\mathcal C}_{22}' \delta ^{-1}+\widehat{\mathcal C}_{22}'' )\varepsilon (\vert \ln \varepsilon \vert +1)
\Vert \mathbf{F}\Vert _{\mathfrak{H}_p(T)}.
\end{split}
\end{align*}

\noindent
$2^\circ .$ Let $\boldsymbol{\varphi}=0$, and let $1\leqslant p\leqslant \infty$. Then for $0<\varepsilon \leqslant \varepsilon _1$ we have
\begin{align*}
\begin{split}
\Vert &\mathbf{u}_\varepsilon -\mathbf{u}_0-\varepsilon \Lambda ^\varepsilon  b(\mathbf{D})\mathbf{w}_\varepsilon \Vert _{L_p((0,T);H^1(\mathcal{O}'))}\\
&\leqslant ({\mathcal C}_{22}' \delta ^{-1}+ \check{\mathcal C}_{22}'' )\varepsilon (\vert \ln \varepsilon\vert +1)\Vert \mathbf{F}\Vert _{\mathfrak{H}_p(T)},
\end{split}
\\
\Vert &\mathbf{p}_\varepsilon -\widetilde{g}^\varepsilon  b(\mathbf{D})\mathbf{w}_\varepsilon \Vert _{L_p((0,T);L_2(\mathcal{O}'))}
\leqslant (\widetilde{\mathcal C}_{22}' \delta ^{-1}+\widehat{\mathcal C}_{22}'' )
\varepsilon (\vert \ln \varepsilon \vert +1)\Vert \mathbf{F}\Vert _{\mathfrak{H}_p(T)}.
\end{align*}
The constants ${\mathcal C}_{22}'$ and $\widetilde{\mathcal C}_{22}'$
are the same as in Theorem {\rm 8.19}.
The constants $\check{\mathcal C}_{22}''$ and $\widehat{\mathcal C}_{22}''$
depend only on the initial data \textnormal{\eqref{data2}} and $\|\Lambda\|_{L_\infty}$.
The constants  ${\rm k}_{p'}$, $\check{{\rm k}}_{p'}'$, $\widetilde{{\rm k}}_{p'}$, and $\widehat{{\rm k}}'_{p'}$ are the same as in
Theorem {\rm 8.15} \textnormal{(}with $p$ replaced by $p'$\textnormal{)}.
\end{theorem}

\section*{Appendix}
\section*{\S 9. Removal of the smoothing operator}
\setcounter{section}{9}
\setcounter{equation}{0}
\setcounter{theorem}{0}

In \S 9, we assume that $d \geqslant 3$ and prove the statements about removing the smoothing operator $S_\eps$
in the corrector in the case of sufficiently smooth boundary (Lemma 3.8 and Theorem 3.9 in the case of the Dirichlet condition,
Lemma 7.8 and Theorem 7.9 in the case of the Neumann condition) and  in a strictly interior subdomain (Lemma 3.15 and Theorem 3.16
in the case of the Dirichlet condition, Lemma 7.14 and Theorem 7.15 in the case of the Neumann condition).

\subsection*{9.1.  The multiplicative property of the matrix-valued function $\Lambda({\mathbf x})$.}

\begin{lemma}
Suppose that the matrix-valued function $\Lambda({\mathbf x})$ is the $\Gamma$-periodic solution of problem
\textnormal{(\ref{Lambda_problem})}. Let $d\geqslant 3$, and let $l=d/2$.

\noindent $1^\circ$. For $0< \varepsilon \leqslant 1$
and for any ${\mathbf v} \in H^{l-1}({\mathbb R}^d; {\mathbb C}^m)$ we have
$\Lambda^\varepsilon {\mathbf v} \in L_2({\mathbb R}^d;{\mathbb C}^n)$, and
\begin{equation}
\label{9.0}
\| \Lambda^\varepsilon {\mathbf v}\|^2_{L_2({\mathbb R}^d)}
 \leqslant C^{(0)} \| {\mathbf v}\|_{H^{l-1}({\mathbb R}^d)}^2.
\end{equation}

\noindent $2^\circ$. For $0< \varepsilon \leqslant 1$
and for any ${\mathbf v} \in H^{l}({\mathbb R}^d; {\mathbb C}^m)$ we have
$\Lambda^\varepsilon {\mathbf v} \in H^1({\mathbb R}^d;{\mathbb C}^n)$, and
\begin{equation}
\label{9.1}
\| \Lambda^\varepsilon {\mathbf v}\|_{H^1({\mathbb R}^d)} \leqslant  C^{(1)} \varepsilon^{-1}
\| {\mathbf v}\|_{L_2({\mathbb R}^d)} + C^{(2)}  \|{\mathbf v}\|_{H^l({\mathbb R}^d)}.
\end{equation}
 The constants $C^{(0)}$, $C^{(1)}$, and $C^{(2)}$ depend on  $m$, $d$, $\alpha_0$, $\alpha_1$,
$\|g\|_{L_\infty}$, $\|g^{-1}\|_{L_\infty}$, and the parameters of the lattice $\Gamma$.
\end{lemma}

\begin{proof}
It suffices to check \eqref{9.0} and \eqref{9.1} for  ${\mathbf v} \in C_0^\infty({\mathbb R}^d; {\mathbb C}^m)$.
Substituting ${\mathbf x}= \varepsilon {\mathbf y}$, $\varepsilon^{d/2}{\mathbf v}({\mathbf x})={\mathbf V}({\mathbf y})$,
we have
\begin{equation}
\label{9.2}
\begin{aligned}
\| \Lambda^\varepsilon {\mathbf v}\|^2_{L_2({\mathbb R}^d)} \leqslant
\int_{{\mathbb R}^d} |\Lambda(\varepsilon^{-1} {\mathbf x})|^2
|{\mathbf v}({\mathbf x})|^2 \,d{\mathbf x} =
\int_{{\mathbb R}^d} |\Lambda({\mathbf y})|^2
|{\mathbf V}({\mathbf y})|^2 \,d{\mathbf y}
\\
=
\sum_{{\mathbf a}\in \Gamma} \int_{\Omega + {\mathbf a}}|\Lambda({\mathbf y})|^2
|{\mathbf V}({\mathbf y})|^2 \,d{\mathbf y} \leqslant
\sum_{{\mathbf a}\in \Gamma}   \|\Lambda\|^2_{L_{2 r}(\Omega)} \| {\mathbf V}\|_{L_{2 r'}(\Omega + {\mathbf a})}^2,
\end{aligned}
\end{equation}
where $r^{-1}+(r')^{-1}=1$. The number $r$ is chosen so that the continuous embedding
$H^1(\Omega) \subset L_{2r}(\Omega)$ is valid, i.~e., $r= d (d-2)^{-1}$. Then
\begin{equation}
\label{9.3}
\|\Lambda\|_{L_{2r}(\Omega)}^2 \leqslant c_\Omega \|\Lambda\|^2_{H^1(\Omega)},
\end{equation}
where the constant $c_\Omega$ depends only on the dimension $d$ and the lattice $\Gamma$.
For $r$ chosen above we have $2r'=d$. From the continuity of the embedding
$H^{l-1}(\Omega) \subset L_{d}(\Omega)$ it follows that
\begin{equation}
\label{9.4}
\| {\mathbf V}\|_{L_{d}(\Omega + {\mathbf a})}^2
 \leqslant c_\Omega'
\| {\mathbf V}\|_{H^{l-1}(\Omega + {\mathbf a})}^2,
\end{equation}
where the constant $c_\Omega'$ depends only on the dimension $d$ and the lattice $\Gamma$.
Now, relations \eqref{9.2}--\eqref{9.4} imply that
\begin{equation}
\label{9.5}
\int_{{\mathbb R}^d} |\Lambda(\varepsilon^{-1} {\mathbf x})|^2
|{\mathbf v}({\mathbf x})|^2 \,d{\mathbf x}
 \leqslant c_\Omega c_\Omega' \|\Lambda\|^2_{H^1(\Omega)}
\| {\mathbf V}\|_{H^{l-1}({\mathbb R}^d)}^2.
\end{equation}

Obviously, for $0< \varepsilon \leqslant 1$ we have
$
\| {\mathbf V}\|_{H^{l-1}({\mathbb R}^d)} \leqslant \| {\mathbf v}\|_{H^{l-1}({\mathbb R}^d)}.
$
Combining this with (\ref{9.5}) and \eqref{M}, we obtain
\begin{equation}
\label{9.6}
\int_{{\mathbb R}^d} |\Lambda(\varepsilon^{-1} {\mathbf x})|^2
|{\mathbf v}({\mathbf x})|^2 \,d{\mathbf x}
 \leqslant c_\Omega c_\Omega'  M^2 \| {\mathbf v}\|_{H^{l-1}({\mathbb R}^d)}^2, \quad
{\mathbf v} \in C_0^\infty({\mathbb R}^d; {\mathbb C}^m),
\end{equation}
which proves \eqref{9.0} with $C^{(0)}= c_\Omega c_\Omega'  M^2$.

Next, by Lemma 1.5,
\begin{equation}
\label{9.7}
\begin{split}
&\| {\mathbf D}(\Lambda^\varepsilon {\mathbf v})\|^2_{L_2({\mathbb R}^d)}
 \\
&\leqslant 2 \varepsilon^{-2}
\int_{{\mathbb R}^d} | ({\mathbf D}\Lambda)^\varepsilon ({\mathbf x}) {\mathbf v}({\mathbf x})  |^2
\,d{\mathbf x} +
2 \int_{{\mathbb R}^d} | \Lambda^\varepsilon ({\mathbf x})|^2 | {\mathbf D}{\mathbf v}({\mathbf x})  |^2
\,d{\mathbf x}
\\
 &\leqslant 2 \beta_1 \varepsilon^{-2}
 \int_{{\mathbb R}^d} |{\mathbf v}({\mathbf x})  |^2\,d{\mathbf x}
+ 2 (1+\beta_2)\int_{{\mathbb R}^d} | \Lambda^\varepsilon ({\mathbf x})|^2 | {\mathbf D}{\mathbf v}({\mathbf x})  |^2
\,d{\mathbf x}.
\end{split}
\end{equation}
From \eqref{9.6} (with ${\mathbf v}$ replaced by the derivatives $\partial_j {\mathbf v}$) it follows that
\begin{equation}
\label{9.8}
\int_{{\mathbb R}^d} | \Lambda^\varepsilon ({\mathbf x})|^2 | {\mathbf D}{\mathbf v}({\mathbf x})  |^2
\,d{\mathbf x}
\leqslant c_\Omega c_\Omega'  M^2 \| {\mathbf v}\|_{H^{l}({\mathbb R}^d)}^2, \quad
{\mathbf v} \in C_0^\infty({\mathbb R}^d; {\mathbb C}^m).
\end{equation}

As a result, relations \eqref{9.6}--\eqref{9.8} imply \eqref{9.1} with the constants
$C^{(1)}=(2\beta_1)^{1/2}$ and $C^{(2)}=M (3+2 \beta_2)^{1/2} (c_\Omega c_\Omega')^{1/2}$.
\end{proof}

Using the extension operator $P_{\mathcal O}$ satisfying estimates \eqref{PO},
we deduce the following statement from Lemma 9.1($1^\circ$).

\begin{corollary}
Suppose that the assumptions of Lemma \textnormal{9.1} are satisfied.
Then the operator $[\Lambda^\eps]$ of multiplication by the matrix-valued function $\Lambda^\eps({\mathbf x})$
is continuous from $H^{l-1}({\mathcal O};{\mathbb C}^m)$ to $L_2({\mathcal O};{\mathbb C}^n)$,
and
\begin{equation*}
\|[\Lambda^\eps] \|_{ H^{l-1}({\mathcal O}) \to L_2({\mathcal O})}
\leqslant (C^{(0)})^{1/2} C^{(l-1)}_{\mathcal O}.
\end{equation*}
\end{corollary}

\subsection*{9.2. Proof of Lemma 3.8.}
Suppose that the assumptions of Lemma 3.8 are satisfied.
Let ${\mathbf u}_0(\cdot,t)= e^{-{\mathcal A}_D^0 t} {\boldsymbol \varphi}(\cdot)$,
where ${\boldsymbol \varphi} \in L_2({\mathcal O}; {\mathbb C}^n)$.
We put $\widetilde{\mathbf u}_0(\cdot,t)= P_{\mathcal O}{\mathbf u}_0(\cdot,t)$.
According to \eqref{K_D(t,e)} and \eqref{K_D^0(t,e)},
${\mathcal K}_D(t;\varepsilon) {\boldsymbol \varphi}= \Lambda^\varepsilon S_\varepsilon b({\mathbf D})  \widetilde{\mathbf u}_0$ and
${\mathcal K}_D^0(t;\varepsilon) {\boldsymbol \varphi}= \Lambda^\varepsilon b({\mathbf D})  {\mathbf u}_0$.
We have to estimate the following function
\begin{equation}
\label{9.8aa}
\begin{split}
\| {\mathcal K}_D(t;\varepsilon) {\boldsymbol \varphi}-{\mathcal K}^0_D(t;\varepsilon) {\boldsymbol \varphi}
\|_{H^1({\mathcal O})}
&=
\| \Lambda^\varepsilon \left((S_\varepsilon -I) b({\mathbf D})  \widetilde{\mathbf u}_0\right) (\cdot,t) \|_{H^1({\mathcal O})}
\\
&\leqslant
\| \Lambda^\varepsilon \left((S_\varepsilon -I) b({\mathbf D})  \widetilde{\mathbf u}_0\right) (\cdot,t) \|_{H^1({\mathbb R}^d)}.
\end{split}
\end{equation}
By Lemma 9.1($2^\circ$), we have
\begin{equation}
\label{9.8b}
\begin{split}
\| &{\mathcal K}_D(t;\varepsilon) {\boldsymbol \varphi}-{\mathcal K}^0_D(t;\varepsilon) {\boldsymbol \varphi}
\|_{H^1({\mathcal O})}
\\
&\leqslant
C^{(1)} \varepsilon^{-1} \| \left((S_\varepsilon -I) b({\mathbf D})  \widetilde{\mathbf u}_0\right) (\cdot,t) \|_{L_2({\mathbb R}^d)}
\\
&+ C^{(2)}  \| \left((S_\varepsilon -I) b({\mathbf D})  \widetilde{\mathbf u}_0\right) (\cdot,t) \|_{H^l({\mathbb R}^d)},\quad l=d/2.
\end{split}
\end{equation}

The first term in the right-hand side of \eqref{9.8b} is estimated by using Lemma 1.8 and
\eqref{<b^*b<}, \eqref{PO}, \eqref{3.5}:
\begin{equation}
\label{9.8c}
\begin{split}
&\varepsilon^{-1}
\|  \left((S_\varepsilon -I) b({\mathbf D})  \widetilde{\mathbf u}_0\right) (\cdot,t) \|_{L_2({\mathbb R}^d)}
\\
&\leqslant
r_1 \|{\mathbf D}  b({\mathbf D})  \widetilde{\mathbf u}_0(\cdot,t) \|_{L_2({\mathbb R}^d)}
\leqslant
r_1 \alpha_1^{1/2} \|  \widetilde{\mathbf u}_0(\cdot,t) \|_{H^2({\mathbb R}^d)}
\\
&\leqslant
r_1 \alpha_1^{1/2} C_{\mathcal O}^{(2)} \| {\mathbf u}_0(\cdot,t) \|_{H^2({\mathcal O})}
\leqslant
C^{(3)} t^{-1} e^{-c_* t/2}
 \|{\boldsymbol \varphi}\|_{L_2({\mathcal O})},
\end{split}
\end{equation}
where $C^{(3)}= r_1 \alpha_1^{1/2} C_{\mathcal O}^{(2)} \widehat{c}$.

In order to estimate the second term in the right-hand side of  \eqref{9.8b}, we apply
the inequality $\|S_\eps\|_{H^l({\mathbb R}^d) \to H^l({\mathbb R}^d)}
\leqslant 1$ and \eqref{<b^*b<}:
\begin{equation}
\label{9.8d}
\begin{split}
\| & \left((S_\varepsilon -I) b({\mathbf D})  \widetilde{\mathbf u}_0\right) (\cdot,t) \|_{H^l({\mathbb R}^d)}
\leqslant 2 \|  b({\mathbf D})  \widetilde{\mathbf u}_0 (\cdot,t) \|_{H^l({\mathbb R}^d)}
\\
&\leqslant 2 \alpha_1^{1/2} \|  \widetilde{\mathbf u}_0 (\cdot,t) \|_{H^{l+1}({\mathbb R}^d)}.
\end{split}
\end{equation}
By \eqref{PO} and Lemma 3.7,
\begin{equation}
\label{9.8e}
\begin{split}
\|  \widetilde{\mathbf u}_0 (\cdot,t) \|_{H^{l+1}({\mathbb R}^d)}
&\leqslant
C_{\mathcal O}^{(l+1)}\| {\mathbf u}_0 (\cdot,t) \|_{H^{l+1}({\mathcal O})}
\\
&\leqslant
C_{\mathcal O}^{(l+1)}
\widehat{\mathrm C}_{l+1} t^{-(l+1)/2} e^{-c_* t/2} \| {\boldsymbol \varphi}\|_{L_2({\mathcal O})}.
\end{split}
\end{equation}
Relations \eqref{9.8d} and \eqref{9.8e} imply that
\begin{equation}
\label{9.8h}
\| \left((S_\varepsilon -I) b({\mathbf D})  \widetilde{\mathbf u}_0\right) (\cdot,t) \|_{H^l({\mathbb R}^d)}
\leqslant
 C^{(4)} t^{-(l+1)/2} e^{-c_* t/2} \| {\boldsymbol \varphi}\|_{L_2({\mathcal O})},
\end{equation}
where $C^{(4)}= 2 \alpha_1^{1/2} C_{\mathcal O}^{(l+1)}\widehat{\mathrm C}_{l+1}$.

As a result, relations
\eqref{9.8b}, \eqref{9.8c}, and \eqref{9.8h} imply that
\begin{equation*}
\begin{split}
\| &{\mathcal K}_D(t;\varepsilon) {\boldsymbol \varphi}-{\mathcal K}^0_D(t;\varepsilon) {\boldsymbol \varphi}
\|_{H^1({\mathcal O})}
\\
&\leqslant
\left( C^{(1)} C^{(3)}  t^{-1} +
 C^{(2)} C^{(4)} t^{-(l+1)/2} \right) e^{-c_* t/2} \| {\boldsymbol \varphi}\|_{L_2({\mathcal O})},
\end{split}
\end{equation*}
which proves (\ref{*.3}) with the constant $\check{\mathrm C}_d=\max\{C^{(1)} C^{(3)};C^{(2)} C^{(4)}\}$.
This completes the proof of Lemma 3.8.

\subsection*{9.3. Proof of Theorem 3.9.}
Inequality \eqref{*.4} follows directly from \eqref{Th_exp_korrector} and \eqref{*.3}.
Herewith, ${\mathrm C}^*_{d}= C_{13}+\check{\mathrm C}_d$.

Let us check \eqref{*.5}. From \eqref{*.4} and \eqref{b(D)=}, \eqref{|b_l|<=} it follows that
\begin{equation}
\label{*.6}
\begin{split}
&\Vert g^\varepsilon b({\mathbf D}) \left(e^{-\mathcal{A}_{D,\varepsilon}t}- e^{-\mathcal{A}_D^0t}
-\varepsilon \Lambda^\varepsilon b({\mathbf D})  e^{-\mathcal{A}_D^0t} \right) \Vert _{L_2(\mathcal{O})\rightarrow L_2(\mathcal{O})}
\\
&\leqslant
\|g\|_{L_\infty} (d \alpha_1)^{1/2} {\mathrm C}^*_{d} (\varepsilon ^{1/2}t^{-3/4}+\varepsilon t^{-1} + \varepsilon t^{-d/4 -1/2})e^{-c_{*}t/2}.
\end{split}
\end{equation}
According to \eqref{b(D)=},
\begin{equation}
\label{*.6a}
\begin{split}
&g^\varepsilon b({\mathbf D}) \left(\varepsilon \Lambda^\varepsilon b({\mathbf D})  e^{-\mathcal{A}_D^0t} \right)
\\
&=g^\varepsilon (b({\mathbf D})  \Lambda)^\varepsilon b({\mathbf D})  e^{-\mathcal{A}_D^0t}
+ \varepsilon \sum_{i,j=1}^d g^\varepsilon b_i  \Lambda^\varepsilon b_j D_i D_j  e^{-\mathcal{A}_D^0t}.
\end{split}
\end{equation}
The norm of the second summand in the right-hand side of  \eqref{*.6a}
is estimated with the help of (1.5), Corollary 9.2 and Lemma 3.7:
\begin{equation}
\label{*.100}
\begin{split}
&\varepsilon \biggl\| \sum_{i,j=1}^d g^\varepsilon b_i  \Lambda^\varepsilon b_j D_i D_j  e^{-\mathcal{A}_D^0t}
\biggr\|_{L_2(\mathcal{O})\rightarrow L_2(\mathcal{O})}
\\
&\leqslant \eps \|g\|_{L_\infty} \alpha_1 d (C^{(0)})^{1/2}  C^{(l-1)}_{\mathcal O}
\| e^{-\mathcal{A}_D^0t}\|_{L_2(\mathcal{O})\rightarrow H^{l+1}(\mathcal{O})}
\\
&\leqslant C^{(5)} \eps t^{-(l+1)/2} e^{-c_{*}t/2},
\end{split}
\end{equation}
where $l=d/2$ and $C^{(5)}=\|g\|_{L_\infty} \alpha_1 d (C^{(0)})^{1/2} C^{(l-1)}_{\mathcal O} \widehat{\mathrm C}_{l+1}$.

Now, relations \eqref{*.6}--\eqref{*.100} and (1.9) imply \eqref{*.5} with the constant
$\widetilde{\mathrm C}^*_{d}= \|g\|_{L_\infty} (d \alpha_1)^{1/2} {\mathrm C}^*_{d}+C^{(5)}$.
This completes the proof of Theorem 3.9.

\subsection*{9.4. Proof of Lemma 7.8.}
We prove Lemma 7.8 by analogy with the proof of Lemma 3.8.
Let ${\boldsymbol \varphi} \in L_2({\mathcal O}; {\mathbb C}^n)$, and let
$\check{\boldsymbol \varphi}= {\mathcal P}{\boldsymbol \varphi}$. We put
${\mathbf u}_0(\cdot,t)= e^{-{\mathcal A}_N^0 t} \check{\boldsymbol \varphi}(\cdot)$ and
$\widetilde{\mathbf u}_0(\cdot,t)= P_{\mathcal O}{\mathbf u}_0(\cdot,t)$.
According to \eqref{K_N(t;eps)} and \eqref{K_N^0(t;e)},
${\mathcal K}_N(t;\varepsilon) {\boldsymbol \varphi}= \Lambda^\varepsilon S_\varepsilon b({\mathbf D})  \widetilde{\mathbf u}_0$ and
${\mathcal K}_N^0(t;\varepsilon) {\boldsymbol \varphi}= \Lambda^\varepsilon b({\mathbf D})  {\mathbf u}_0$.
We have taken into account that $b({\mathbf D}) e^{-{\mathcal A}_N^0 t} = b({\mathbf D}) e^{-{\mathcal A}_N^0 t} {\mathcal P}$.

Similarly to \eqref{9.8aa}, we obtain
\begin{equation*}
\| {\mathcal K}_N(t;\varepsilon) {\boldsymbol \varphi}-{\mathcal K}^0_N(t;\varepsilon) {\boldsymbol \varphi}
\|_{H^1({\mathcal O})}
\leqslant
 \| \Lambda^\eps \left((S_\varepsilon -I) b({\mathbf D})  \widetilde{\mathbf u}_0\right) (\cdot,t) \|_{H^1({\mathbb R}^d)}.
\end{equation*}
Further considerations are similar to \eqref{9.8b}--\eqref{9.8h};
instead of  \eqref{3.5} we use \eqref{7.5}, instead of Lemma 3.7 we apply Lemma 7.7.
This yields
\begin{equation*}
\begin{split}
\| &{\mathcal K}_N(t;\varepsilon) {\boldsymbol \varphi}-{\mathcal K}^0_N(t;\varepsilon) {\boldsymbol \varphi}
\|_{H^1({\mathcal O})}
\\
&\leqslant
\left( C^{(1)} C^{(3)}_N  t^{-1} +
 C^{(2)} C^{(4)}_N t^{-(l+1)/2} \right) e^{-c_\flat t/2} \| {\boldsymbol \varphi}\|_{L_2({\mathcal O})},
\end{split}
\end{equation*}
where $C^{(3)}_N= r_1 \alpha_1^{1/2} C_{\mathcal O}^{(2)} \check{c}^\circ$ and
 $C^{(4)}_N= 2 \alpha_1^{1/2} C_{\mathcal O}^{(l+1)} \widehat{\mathscr C}_{l+1}$.
We arrive at (\ref{7.*3}) with the constant $\check{\mathscr C}_d=\max\{C^{(1)} C_N^{(3)};C^{(2)} C_N^{(4)}\}$.
Lemma 7.8 is proved.

\subsection*{9.5. Proof of Theorem 7.9.}
Arguments are similar to the proof of Theorem 3.9.
Inequality \eqref{7.*4} follows directly from \eqref{Th exn_N L2->H1} and \eqref{7.*3}.
Herewith, ${\mathscr C}^*_{d}= {\mathcal C}_{13}+\check{\mathscr C}_d$.

Inequality \eqref{7.*5} is deduced from \eqref{7.*4} by analogy with  \eqref{*.6}--\eqref{*.100}.
In addition, we use the identity
\begin{equation}
\label{9.200}
\varepsilon \sum_{i,j=1}^d g^\varepsilon b_i  \Lambda^\varepsilon b_j D_i D_j  e^{-\mathcal{A}_N^0t}
= \varepsilon \sum_{i,j=1}^d g^\varepsilon b_i  \Lambda^\varepsilon b_j D_i D_j  e^{-\mathcal{A}_N^0t} {\mathcal P}
\end{equation}
and estimate this term with the help of Corollary 9.2 and Lemma 7.7. We arrive at
\eqref{7.*5} with the constant
$\widetilde{\mathscr C}^*_{d}= \|g\|_{L_\infty} (d \alpha_1)^{1/2} {\mathscr C}^*_{d}+
\|g\|_{L_\infty} \alpha_1 d (C^{(0)})^{1/2} C^{(l-1)}_{\mathcal O} \widehat{\mathscr C}_{l+1}$.
This completes the proof of Theorem 7.9.

\subsection*{9.6. One property of the operator $S_\eps$.}
We proceed to estimates in a strictly interior subdomain and start with one simple property of the operator
$S_\eps$.

Let ${\mathcal O}'$ be a strictly interior subdomain of the domain ${\mathcal O}$, and let $\delta = {\rm dist}\,\{{\mathcal O}';
\partial {\mathcal O}\}$. Denote
$$
{\mathcal O}''=\{ {\mathbf x}\in {\mathcal O}:\ {\rm dist}\,\{{\mathbf x};
\partial {\mathcal O}\} > \delta/2\},\quad
{\mathcal O}'''=\{ {\mathbf x}\in {\mathcal O}:\ {\rm dist}\,\{{\mathbf x};
\partial {\mathcal O}\} > \delta/4\}.
$$

\begin{lemma}
Let $S_\eps$ be the operator \textnormal{(\ref{S_e})}. Let $r_1$ be defined by \textnormal{(\ref{r})}.
Then for $0< \eps \leqslant (4r_1)^{-1}\delta$ and for any function  ${\mathbf v} \in H^{s}({\mathbb R}^d; {\mathbb C}^m)$,
$s \in {\mathbb Z}_+$, we have
\begin{equation*}
\| S_\eps  {\mathbf v}\|_{H^s({\mathcal O}'')} \leqslant
\|  {\mathbf v}\|_{H^s({\mathcal O}''')}.
\end{equation*}
\end{lemma}

\begin{proof}
According to (\ref{S_e}),
\begin{equation}
\label{10.2}
\begin{split}
\| S_\eps  {\mathbf v}\|^2_{H^s({\mathcal O}'')}
&=
|\Omega|^{-2} \sum_{|\alpha| \leqslant s} \int_{{\mathcal O}''} d{\mathbf x}\left| \int_{\Omega} {\mathbf D}^\alpha
{\mathbf v}({\mathbf x} - \eps {\mathbf z}) \, d{\mathbf z}\right|^2
\\
&\leqslant
|\Omega|^{-1} \sum_{|\alpha| \leqslant s} \int_{{\mathcal O}''} d{\mathbf x} \int_{\Omega} \left|{\mathbf D}^\alpha
{\mathbf v}({\mathbf x} - \eps {\mathbf z})\right|^2 d{\mathbf z}.
\end{split}
\end{equation}
Since $0< \eps r_1 \leqslant \delta/4$, then for ${\mathbf x} \in {\mathcal O}''$ and ${\mathbf z}\in \Omega$
we have ${\mathbf x} - \eps {\mathbf z}\in {\mathcal O}'''$.
Then, changing the order of integration in \eqref{10.2}, we obtain
\begin{equation*}
\| S_\eps  {\mathbf v}\|^2_{H^s({\mathcal O}'')}
\leqslant
 \sum_{|\alpha| \leqslant s} \int_{{\mathcal O}'''} d{\mathbf x}
\left|\mathbf{D}^\alpha {\mathbf v}({\mathbf x} )\right|^2 \,d{\mathbf x} = \|  {\mathbf v}\|_{H^s({\mathcal O}''')}^2.
\end{equation*}
\end{proof}

\subsection*{9.7. The cut-off function $\chi({\mathbf x})$.}
We fix a smooth cut-off function $\chi({\mathbf x})$ such that
\begin{equation}
\label{10.3}
\begin{split}
&\chi \in C^\infty_0({\mathbb R}^d),\quad 0 \leqslant \chi({\mathbf x}) \leqslant 1; \quad \chi({\mathbf x})=1,\ {\mathbf x}\in {\mathcal O}' ;
\\
&{\rm supp}\,\chi \subset   {\mathcal O}'';\quad |{\mathbf D}^\alpha \chi({\mathbf x})| \leqslant \varsigma_s \delta^{-s}, \ |\alpha|=s,\
\forall s\in {\mathbb N}.
\end{split}
\end{equation}
The constants $\varsigma_s$ depend only on $d$, $s$, and the domain ${\mathcal O}$.

\begin{lemma}
Suppose that the function $\chi({\mathbf x})$ satisfies conditions \textnormal{(\ref{10.3})}.
Let $k\in {\mathbb Z}_+$.

\noindent
$1^\circ$. For any function  ${\mathbf v} \in H^{k}({\mathbb R}^d; {\mathbb C}^m)$
we have
\begin{equation}
\label{10.4}
\| \chi {\mathbf v}\|_{H^k({\mathbb R}^d)}
\leqslant
C_k^{(6)} \sum_{s=0}^k \delta^{-(k-s)} \|{\mathbf v}\|_{H^{s}({\mathcal O}'')},
\end{equation}
where the constant  $C_k^{(6)}$ depends only on $d$, $k$, and the domain ${\mathcal O}$.

\noindent
$2^\circ$. For any function  ${\mathbf v} \in H^{k+1}({\mathbb R}^d; {\mathbb C}^m)$
we have
\begin{equation}
\label{10.5}
\begin{split}
\| \chi {\mathbf v}\|_{H^{k+1/2}({\mathbb R}^d)}
&\leqslant
C_{k+1/2}^{(6)} \biggl( \sum_{s=0}^{k+1} \delta^{-(k+1 -s)} \|{\mathbf v}\|_{H^{s}({\mathcal O}'')} \biggr)^{1/2}
\\
&\times \biggl( \sum_{\sigma=0}^{k} \delta^{-(k -\sigma)} \|{\mathbf v}\|_{H^{\sigma}({\mathcal O}'')} \biggr)^{1/2},
\end{split}
\end{equation}
where the constant  $C_{k+1/2}^{(6)}$ depends only on $d$, $k$, and the domain ${\mathcal O}$.
\end{lemma}

\begin{proof}
Inequality \eqref{10.4} follows from the Leibniz formula for derivatives of the product
$\chi {\mathbf v}$ and estimates for derivatives of $\chi$ (see \eqref{10.3}).
To check \eqref{10.5}, in addition one should use the obvious inequality
$$
\| {\mathbf w} \|_{H^{k+1/2}({\mathbb R}^d)}^2 \leqslant
\| {\mathbf w} \|_{H^{k+1}({\mathbb R}^d)} \|  {\mathbf w} \|_{H^{k}({\mathbb R}^d)},\quad
{\mathbf w} \in H^{k+1}({\mathbb R}^d; {\mathbb C}^m).
$$
\end{proof}

\subsection*{9.8. Proof of Lemma 3.15.}
Suppose that the assumptions of Lemma 3.15 are satisfied.
Let ${\mathbf u}_0(\cdot,t)= e^{-{\mathcal A}_D^0 t} {\boldsymbol \varphi}(\cdot)$,
where ${\boldsymbol \varphi} \in L_2({\mathcal O}; {\mathbb C}^n)$. By  \eqref{<a_D,0<} and \eqref{3.8},
\begin{equation}
\label{9.8ab}
\begin{split}
\| {\mathbf D}{\mathbf u}_0(\cdot,t)\|_{L_2({\mathcal O})} &\leqslant
c_0^{-1/2} \|({\mathcal A}_D^0)^{1/2} e^{-{\mathcal A}_D^0 t} \|_{L_2({\mathcal O})\to L_2({\mathcal O})}
 \|{\boldsymbol \varphi}\|_{L_2({\mathcal O})}
\\
&\leqslant
  c_0^{-1/2} t^{-1/2} e^{-c_* t/2}
 \|{\boldsymbol \varphi}\|_{L_2({\mathcal O})}.
\end{split}
\end{equation}

From \eqref{3.5} it follows that
\begin{equation}
\label{9.8a}
\|{\mathbf u}_0(\cdot,t)\|_{H^2({\mathcal O})}
\leqslant
\widehat{c} t^{-1} e^{-c_* t/2}
 \|{\boldsymbol \varphi}\|_{L_2({\mathcal O})}.
\end{equation}
Let $\widetilde{\mathbf u}_0= P_{\mathcal O} {\mathbf u}_0$.
According to \eqref{K_D(t,e)} and \eqref{K_D^0(t,e)},
${\mathcal K}_D(t;\varepsilon) {\boldsymbol \varphi}= \Lambda^\varepsilon S_\varepsilon b({\mathbf D})  \widetilde{\mathbf u}_0$ and
${\mathcal K}_D^0(t;\varepsilon) {\boldsymbol \varphi}= \Lambda^\varepsilon b({\mathbf D})  {\mathbf u}_0$.

We have to estimate the function
\begin{equation}
\label{9.11}
\| {\mathcal K}_D(t;\varepsilon) {\boldsymbol \varphi}-{\mathcal K}^0_D(t;\varepsilon) {\boldsymbol \varphi}
\|_{H^1({\mathcal O}')}
\leqslant
\| \Lambda^\varepsilon \chi \left((S_\varepsilon -I) b({\mathbf D})  \widetilde{\mathbf u}_0\right)(\cdot,t) \|_{H^1({\mathbb R}^d)}.
\end{equation}

By Lemma 9.1($2^\circ$),
\begin{equation}
\label{9.12}
\begin{split}
\| &\Lambda^\varepsilon \chi \left((S_\varepsilon -I) b({\mathbf D})  \widetilde{\mathbf u}_0\right)(\cdot,t) \|_{H^1({\mathbb R}^d)}
\\
&\leqslant
C^{(1)} \varepsilon^{-1}
\|  \chi \left((S_\varepsilon -I) b({\mathbf D})  \widetilde{\mathbf u}_0\right)(\cdot,t) \|_{L_2({\mathbb R}^d)}
\\
&+ C^{(2)}
\|  \chi \left((S_\varepsilon -I) b({\mathbf D})  \widetilde{\mathbf u}_0\right) (\cdot,t) \|_{H^l({\mathbb R}^d)},\quad l=d/2.
\end{split}
\end{equation}

The first term in the right-hand side of \eqref{9.12} is estimated with the help of Lemma 1.8 and
\eqref{<b^*b<}, \eqref{PO}, \eqref{9.8a}. Similarly to \eqref{9.8c}, we have
\begin{equation}
\label{9.13}
\begin{split}
\varepsilon^{-1}
\|  \chi \left((S_\varepsilon -I) b({\mathbf D})  \widetilde{\mathbf u}_0\right) (\cdot,t) \|_{L_2({\mathbb R}^d)}
&\leqslant
\varepsilon^{-1}
\|  \left( (S_\varepsilon -I) b({\mathbf D})  \widetilde{\mathbf u}_0\right) (\cdot,t) \|_{L_2({\mathbb R}^d)}
\\
&\leqslant
C^{(3)} t^{-1} e^{-c_* t/2}  \|{\boldsymbol \varphi}\|_{L_2({\mathcal O})}.
\end{split}
\end{equation}

We proceed to the second term in the right-hand side of \eqref{9.12}. Obviously,
\begin{equation}
\label{9.13a}
\begin{split}
\|  \chi \left((S_\varepsilon -I) b({\mathbf D})  \widetilde{\mathbf u}_0\right) (\cdot,t) \|_{H^l({\mathbb R}^d)}
&\leqslant
\|  \chi \left(S_\varepsilon b({\mathbf D})  \widetilde{\mathbf u}_0\right) (\cdot,t) \|_{H^l({\mathbb R}^d)}\\
&+
\|  \chi  b({\mathbf D})  \widetilde{\mathbf u}_0 (\cdot,t) \|_{H^l({\mathbb R}^d)}.
\end{split}
\end{equation}

To estimate the second term in the right-hand side of \eqref{9.13a}, we apply Lemma 9.4 and \eqref{b(D)=}, \eqref{|b_l|<=}.
In the case of integer $l=d/2$ (i.~e., $d$ is even) we have
\begin{equation}
\label{9.15}
\| \chi   b({\mathbf D})  \widetilde{\mathbf u}_0(\cdot,t)\|_{H^l({\mathbb R}^d)}
\leqslant
C_l^{(6)} (d\alpha_1)^{1/2}\sum_{s=0}^l \delta^{-(l-s)} \| {\mathbf D} {\mathbf u}_0(\cdot,t) \|_{H^{s}({\mathcal O}'')}.
\end{equation}
In the case of half-integer $l=d/2=k+1/2$ (i.~e., $d$ is odd) we have
\begin{equation}
\label{9.15a}
\begin{split}
\| \chi   b({\mathbf D})  \widetilde{\mathbf u}_0(\cdot,t)\|_{H^l({\mathbb R}^d)}
&\leqslant
C_l^{(6)} (d\alpha_1)^{1/2}
\biggl(\sum_{s=0}^{k+1} \delta^{-(k+1-s)} \| {\mathbf D} {\mathbf u}_0(\cdot,t) \|_{H^{s}({\mathcal O}'')}\biggr)^{1/2}
\\
&\times \biggl(\sum_{\sigma=0}^{k} \delta^{-(k-\sigma)} \| {\mathbf D} {\mathbf u}_0(\cdot,t) \|_{H^{\sigma}({\mathcal O}'')}\biggr)^{1/2}.
\end{split}
\end{equation}

The norms of ${\mathbf D} {\mathbf u}_0(\cdot,t)$ in $L_2({\mathcal O})$ and $H^1({\mathcal O})$
have been estimated in \eqref{9.8ab} and \eqref{9.8a}. By \eqref{apriori} (with ${\mathcal O}'$ replaced by ${\mathcal O}''$),
\begin{equation}
\label{9.16}
\| {\mathbf D}{\mathbf u}_0(\cdot,t)\|_{H^{s}({\mathcal O}'')}
\leqslant
{\rm C}'_{s+1} 2^s t^{-1/2} (\delta^{-2} + t^{-1})^{s/2}   e^{-c_* t/2}
 \|{\boldsymbol \varphi}\|_{L_2({\mathcal O})},\quad s \geqslant 2.
\end{equation}
Using \eqref{9.8ab}, \eqref{9.8a}, and \eqref{9.15}--\eqref{9.16}, we arrive at
\begin{equation}
\label{9.17}
\| \chi   b({\mathbf D})  \widetilde{\mathbf u}_0(\cdot,t)\|_{H^l({\mathbb R}^d)}
\leqslant
C^{(7)} t^{-1/2} (\delta^{-2} + t^{-1})^{d/4} e^{-c_* t/2}
 \|{\boldsymbol \varphi}\|_{L_2({\mathcal O})}.
\end{equation}
The constant $C^{(7)}$ depends on $d$, $\alpha_0$, $\alpha_1$, $\|g\|_{L_\infty}$, $\|g^{-1}\|_{L_\infty}$, and the domain $\mathcal O$.

To estimate the first term in the right-hand side of \eqref{9.13a}, we apply Lemmas 9.4 and 9.3.
Assume that $0< \eps \leqslant (4r_1)^{-1}\delta$. Taking  \eqref{b(D)=} and \eqref{|b_l|<=} into account,
in the case of integer $l$ we have
\begin{equation}
\label{9.18}
\begin{split}
\|  \chi \left(S_\varepsilon b({\mathbf D})  \widetilde{\mathbf u}_0\right) (\cdot,t) \|_{H^l({\mathbb R}^d)}
&\leqslant
C_l^{(6)} \sum_{s=0}^l \delta^{-(l-s)} \| (S_\eps b({\mathbf D}) \widetilde{\mathbf u}_0)(\cdot,t) \|_{H^{s}({\mathcal O}'')}
\\
&\leqslant
C_l^{(6)} \sum_{s=0}^l \delta^{-(l-s)} \|  b({\mathbf D}) {\mathbf u}_0(\cdot,t) \|_{H^{s}({\mathcal O}''')}
\\
&\leqslant
C_l^{(6)} (d \alpha_1)^{1/2}\sum_{s=0}^l \delta^{-(l-s)} \|  {\mathbf D} {\mathbf u}_0(\cdot,t) \|_{H^{s}({\mathcal O}''')}.
\end{split}
\end{equation}
The norms of ${\mathbf D} {\mathbf u}_0(\cdot,t)$ in $L_2({\mathcal O})$ and $H^1({\mathcal O})$
have been estimated in \eqref{9.8ab} and \eqref{9.8a}. By \eqref{apriori} (with ${\mathcal O}'$ replaced by ${\mathcal O}'''$), we have
\begin{equation}
\label{9.16a}
\| {\mathbf D}{\mathbf u}_0(\cdot,t)\|_{H^{s}({\mathcal O}''')}
\leqslant
{\rm C}'_{s+1} 4^s t^{-1/2} (\delta^{-2} + t^{-1})^{s/2}   e^{-c_* t/2}
 \|{\boldsymbol \varphi}\|_{L_2({\mathcal O})},\quad s \geqslant 2.
\end{equation}
From \eqref{9.8ab}, \eqref{9.8a}, \eqref{9.18}, and \eqref{9.16a} it follows that
\begin{equation}
\label{9.19}
\|  \chi \left(S_\varepsilon b({\mathbf D})  \widetilde{\mathbf u}_0\right) (\cdot,t) \|_{H^l({\mathbb R}^d)}
\leqslant
C^{(8)} t^{-1/2} (\delta^{-2} + t^{-1})^{d/4} e^{-c_* t/2}
 \|{\boldsymbol \varphi}\|_{L_2({\mathcal O})}.
\end{equation}
The constant $C^{(8)}$ depends on $d$, $\alpha_0$, $\alpha_1$, $\|g\|_{L_\infty}$, $\|g^{-1}\|_{L_\infty}$,
and the domain $\mathcal O$.
The case of half-integer $l$ is considered similarly.

As a result, relations \eqref{9.11}--\eqref{9.13a}, \eqref{9.17}, and \eqref{9.19} yield
\begin{equation}
\label{9.20}
\begin{split}
\|& {\mathcal K}_D(t;\varepsilon) {\boldsymbol \varphi}-{\mathcal K}^0_D(t;\varepsilon) {\boldsymbol \varphi}
\|_{H^1({\mathcal O}')}
\leqslant
C^{(1)}C^{(3)} t^{-1}  e^{-c_* t/2}  \|{\boldsymbol \varphi}\|_{L_2({\mathcal O})}
\\
&+
C^{(2)}(C^{(7)}+ C^{(8)}) t^{-1/2} (\delta^{-2} + t^{-1})^{d/4} e^{-c_* t/2}
 \|{\boldsymbol \varphi}\|_{L_2({\mathcal O})}.
\end{split}
\end{equation}
This implies \eqref{K_minus_K} with the constant
${\rm C}_d'' = \max \{C^{(1)}C^{(3)} ;C^{(2)}(C^{(7)}+ C^{(8)})\}$ and completes the proof of Lemma 3.15.

\subsection*{9.9. Proof of Theorem 3.16.}
Inequality \eqref{3.26} follows directly from \eqref{th3.8} and \eqref{K_minus_K};
herewith,  ${\mathrm C}_d = \max \{C_{16}; C_{17}\} + {\mathrm C}_d''$.

Let us check \eqref{3.27}.
From \eqref{3.26} and \eqref{b(D)=}, \eqref{|b_l|<=} it follows that
\begin{equation}
\label{*.101}
\begin{split}
&\Vert g^\varepsilon b({\mathbf D}) \left(e^{-\mathcal{A}_{D,\varepsilon}t}- e^{-\mathcal{A}_D^0t}
-\varepsilon \Lambda^\varepsilon b({\mathbf D})  e^{-\mathcal{A}_D^0t} \right) \Vert _{L_2(\mathcal{O})\rightarrow L_2(\mathcal{O}')}
\\
&\leqslant
\|g\|_{L_\infty} (d \alpha_1)^{1/2} {\mathrm C}_{d}
\varepsilon h_d(\delta;t) e^{-c_{*}t/2}.
\end{split}
\end{equation}
We apply the identity \eqref{*.6a}.
The norm of the second term in the right-hand side of  \eqref{*.6a}
is estimated with the help of (1.5) and Lemma 9.1($1^\circ$):
\begin{equation}
\label{*.102}
\begin{split}
&\varepsilon \biggl\| \sum_{i,j=1}^d g^\varepsilon b_i  \Lambda^\varepsilon b_j D_i D_j  e^{-\mathcal{A}_D^0t}
\biggr\|_{L_2(\mathcal{O})\rightarrow L_2(\mathcal{O}')}
\\
&\leqslant \eps \|g\|_{L_\infty} \alpha_1^{1/2}
\sum_{i,j=1}^d  \|  \Lambda^\varepsilon b_j \chi D_i D_j  e^{-\mathcal{A}_D^0t}
\|_{L_2(\mathcal{O})\rightarrow L_2(\mathbb{R}^d)}
\\
&\leqslant \eps \|g\|_{L_\infty} \alpha_1 (C^{(0)})^{1/2}
\sum_{i,j=1}^d  \|   \chi D_i D_j  e^{-\mathcal{A}_D^0t}
\|_{L_2(\mathcal{O})\rightarrow H^{l-1}(\mathbb{R}^d)},\quad l=d/2.
\end{split}
\end{equation}
Next, we apply Lemma 9.4. For the case of integer $l$ we obtain
\begin{equation}
\label{*.103}
\begin{split}
&\sum_{i,j=1}^d  \|   \chi D_i D_j  e^{-\mathcal{A}_D^0t}
\|_{L_2(\mathcal{O})\rightarrow H^{l-1}(\mathbb{R}^d)}
\\
&\leqslant d C^{(6)}_{l-1}
\sum_{s=0}^{l-1} \delta^{-(l-1-s)} \| e^{-\mathcal{A}_D^0t} \|_{L_2(\mathcal{O})\rightarrow H^{s+2}(\mathcal{O}'')}.
\end{split}
\end{equation}
The norm $\| e^{-\mathcal{A}_D^0 t} \|_{L_2(\mathcal{O})\rightarrow H^{2}(\mathcal{O})}$
has been estimated in  \eqref{3.5}. For $s \geqslant 1$, by
\eqref{apriori} (with ${\mathcal O}'$ replaced by ${\mathcal O}''$), we have
\begin{equation}
\label{*.104}
 \| e^{-\mathcal{A}_D^0t} \|_{L_2(\mathcal{O})\rightarrow H^{s+2}(\mathcal{O}'')}
\leqslant
{\mathrm C}'_{s+2} 2^{s+1} t^{-1/2} (\delta^{-2} + t^{-1})^{(s+1)/2} e^{-c_{*}t/2}.
\end{equation}
From \eqref{*.102}--\eqref{*.104} and  \eqref{3.5} it follows that
\begin{equation}
\label{*.105}
\begin{split}
&\varepsilon \biggl\| \sum_{i,j=1}^d g^\varepsilon b_i  \Lambda^\varepsilon b_j D_i D_j  e^{-\mathcal{A}_D^0t}
\biggr\|_{L_2(\mathcal{O})\rightarrow L_2(\mathcal{O}')}
\\
&\leqslant C^{(9)} \eps t^{-1/2} (\delta^{-2} + t^{-1})^{d/4} e^{-c_{*}t/2},
\end{split}
\end{equation}
where the constant $C^{(9)}$ depends only on
$d$, $\alpha_0$, $\alpha_1$, $\|g\|_{L_\infty}$, $\|g^{-1}\|_{L_\infty}$, and the domain $\mathcal{O}$.
In the case of half-integer $l$ inequality \eqref{*.105} is checked similarly by using Lemma 9.4($2^\circ$).

As a result, relations \eqref{*.6a},
 \eqref{*.101},  \eqref{*.105}, and (1.9) imply \eqref{3.27} with the constant
$\widetilde{\mathrm C}_{d}= \|g\|_{L_\infty} (d \alpha_1)^{1/2} {\mathrm C}_{d}+C^{(9)}$.
This completes the proof of Theorem 3.16.

\subsection*{9.10. Proof of Lemma 7.14.}
Arguments are similar to the proof of Lemma 3.15.
As in Subsection 9.4, let ${\boldsymbol \varphi} \in L_2({\mathcal O}; {\mathbb C}^n)$ and let
$\check{\boldsymbol \varphi}= {\mathcal P}{\boldsymbol \varphi}$. We put
${\mathbf u}_0(\cdot,t)= e^{-{\mathcal A}_N^0 t} \check{\boldsymbol \varphi}(\cdot)$,
$\widetilde{\mathbf u}_0(\cdot,t)= P_{\mathcal O}{\mathbf u}_0(\cdot,t)$.
According to \eqref{K_N(t;eps)} and \eqref{K_N^0(t;e)},
${\mathcal K}_N(t;\varepsilon) {\boldsymbol \varphi}= \Lambda^\varepsilon S_\varepsilon b({\mathbf D})  \widetilde{\mathbf u}_0$ and
${\mathcal K}_N^0(t;\varepsilon) {\boldsymbol \varphi}= \Lambda^\varepsilon b({\mathbf D})  {\mathbf u}_0$.

Similarly to \eqref{9.11} we obtain
\begin{equation*}
\| {\mathcal K}_N(t;\varepsilon) {\boldsymbol \varphi}-{\mathcal K}^0_N(t;\varepsilon) {\boldsymbol \varphi}
\|_{H^1({\mathcal O}')}
\leqslant
 \| \Lambda^\eps \chi \left((S_\varepsilon -I) b({\mathbf D})  \widetilde{\mathbf u}_0\right) (\cdot,t) \|_{H^1({\mathbb R}^d)}.
\end{equation*}
Further considerations are similar to \eqref{9.12}--\eqref{9.20};
one should use estimates \eqref{7.4}, \eqref{7.5}, and \eqref{apriori_N}.
This implies \eqref{K_minus_K_N}.

\subsection*{9.11. Proof of Theorem 7.15.}
Arguments are similar to the proof of Theorem 3.16.
Inequality \eqref{7.100}
follows directly from  \eqref{7.29a} and \eqref{K_minus_K_N}.
Herewith, ${\mathscr C}_d = \max\{{\mathcal C}_{16}; {\mathcal C}_{17}\} + {\mathscr C}_d''$.

Inequality \eqref{7.33} is deduced from \eqref{7.100} by analogy with \eqref{*.101}--\eqref{*.105}.
In addition, one should take \eqref{9.200} into account and apply estimates \eqref{7.5} and \eqref{apriori_N}.

\end{document}